\documentclass[final]{amsart}
\usepackage{amsxtra}
\usepackage{amssymb,amsmath,amsthm,stmaryrd}
\usepackage[all]{xy}
\usepackage{appendix}
\usepackage{cancel}
\usepackage{hyperref}
\usepackage[spacing]{microtype}
\theoremstyle{plain}

\newtheorem{theorem}{Theorem}[section]
\newtheorem{lemma}[theorem]{Lemma}
\newtheorem{proposition}[theorem]{Proposition}
\newtheorem{corollary}[theorem]{Corollary}
\newtheorem{definition}[theorem]{Definition} \theoremstyle{definition}
\newtheorem{example}[theorem]{Example}
\newtheorem{remark}[theorem]{Remark}

\newcommand{\R}{\mathbb{R}} 
 \newcommand{\inv}{^{-1}}
 
\newcommand{\mx}{\mathfrak{X}} \newcommand{\dr}{\mathbf{d}}
 
 \newcommand{\ldr}[1]{{{\pounds}}_{#1}}
\newcommand{\ip}[1]{{\mathbf{i}}_{#1}}
\newcommand{\an}[1]{\arrowvert_{#1}}  
 
\newcommand{\lb}{\llbracket} \newcommand{\rb}{\rrbracket}

 \DeclareMathOperator{\pr}{pr}
\DeclareMathOperator{\Hom}{Hom} \DeclareMathOperator{\graphe}{graph}
\DeclareMathOperator{\Skew}{Skew}
\DeclareMathOperator{\Id}{Id}
\DeclareMathOperator{\id}{id}

\begin{document}
\title{Dorfman connections and Courant algebroids}


\author{M. Jotz Lean}
\address{School of Mathematics and Statistics, The University of Sheffield.} 
\email{M.Jotz-Lean@sheffield.ac.uk}
\thanks{Supported  by the
\emph{Dorothea Schl\"ozer Program} of the  University of G\"ottingen
and a \emph{Fellowship for prospective researchers (PBELP2\_137534) of the
Swiss NSF} for a postdoctoral stay at UC Berkeley}


\begin{abstract}
  We define Dorfman connections, which are to Courant algebroids what
  connections are to Lie algebroids.  Several examples illustrate
  this analogy.

  A linear connection $\nabla\colon \mx(M)\times\Gamma(E)\to\Gamma(E)$
  on a vector bundle $E$ over a smooth manifold $M$ is tantamount to a
  linear splitting $TE\simeq T^{q_E}E\oplus H_\nabla$, where
  $T^{q_E}E$ is the set of vectors tangent to the fibres of $E$.
  Furthermore, the curvature of the connection measures the failure of
  the horizontal space $H_\nabla$ to be integrable.  We show that
  linear horizontal complements to $T^{q_E}E\oplus (T^{q_E}E)^\circ$
  in the Pontryagin bundle over the vector bundle $E$ can be described
  in the same manner via a certain class of Dorfman connections
  $\Delta\colon \Gamma(TM\oplus E^*)\times\Gamma(E\oplus
  T^*M)\to\Gamma(E\oplus T^*M)$.
  Similarly to the tangent bundle case, we find that, after the choice of a
  linear splitting, the standard Courant algebroid structure of
  $TE\oplus T^*E\to E$ can be completely described by properties of
  the Dorfman connection.

  As an application, we study splittings of $TA\oplus T^*A$ over a Lie
  algebroid $A$ and, following Gracia-Saz and Mehta, we compute the
  representations up to homotopy defined by any linear splitting of
  $TA\oplus T^*A$ and the linear Lie algebroid
  $TA\oplus T^*A\to TM\oplus A^*$.  Further, we characterise VB- and
  LA-Dirac structures in $TA\oplus T^*A$ via Dorfman connections.
\end{abstract}
\maketitle

\tableofcontents

\section{Introduction}
This paper introduces Dorfman connections, and studies in depth the
standard Courant algebroid over a vector bundle.
 Let us begin with a
simple observation.  Take a subbundle $F\subseteq TM$ of the tangent
bundle of a smooth manifold $M$. Then the $\R$-bilinear map
\[\tilde \nabla\colon  \Gamma(F)\times \mx(M)\to \Gamma(TM/F),\qquad \tilde \nabla_XY=\overline{[X,Y]}\]
measures the failure of vector fields on $M$ to preserve $F$.  The
subbundle $F$ is involutive if and only if $\tilde\nabla_XY=0$ for
all $X,Y\in\Gamma(F)$.   
In this case, $\tilde \nabla$ induces a
flat connection
\[\nabla\colon \Gamma(F)\times\Gamma(TM/F)\to\Gamma(TM/F),\qquad \nabla_X\bar Y=\overline{[X,Y]},\]
the \textbf{Bott connection} associated to $F$ \cite{Bott72}.

In the same manner, given a Courant algebroid $\mathsf E\to M$ with
bracket $\lb\cdot\,,\cdot\rb$, anchor $\rho$ and pairing
$\langle\cdot\,,\cdot\rangle$, and a subbundle $K\subseteq \mathsf E$,
we define an $\R$-bilinear map
\[\tilde\Delta\colon \Gamma(K)\times\Gamma(\mathsf E)\to \Gamma(\mathsf
E/K), \qquad \tilde\Delta_ke=\overline{\lb k,e\rb}.\] Again, we have $\tilde\Delta_kk'=0$ for all
$k,k'\in\Gamma(K)$ if and only if $\Gamma(K)$ is closed under the
bracket on $\Gamma(\mathsf E)$.   If $K$ is in
addition isotropic, it is a Lie algebroid over $M$ and the pairing on
$\mathsf E$ induces a pairing $K\times_M(\mathsf E/K)\to \R$.  The
$\R$-bilinear map
\[\Delta\colon \Gamma(K)\times\Gamma(\mathsf E/K)\to\Gamma(\mathsf
E/K),\qquad \tilde\Delta_k\bar e=\overline{\lb k,e\rb}\] that is
induced by $\tilde \Delta$ is not a connection because it is not
$C^\infty(M)$-homogeneous in the first argument, but the obstruction
to this is, as we will see, measured by the pairing, the anchor of the
Courant algebroid and the de Rham derivative on $C^\infty(M)$.  This
map is an example of what we call a Dorfman
connection, namely the \textbf{Bott--Dorfman connection} associated to
$K$ in $\mathsf E$.  Dorfman connections appear naturally in several
situations related to Courant algebroids and play a role similar to
the one that connections play for tangent bundles and Lie
algebroids. We illustrate this with a few examples.

\medskip

Our main motivation for introducing this new concept is the following.
It goes back to Dieudonn\'e that a linear $TM$-connection $\nabla$
on a vector bundle $q_E\colon E\to M$ corresponds
to 
a splitting $TE\simeq T^{q_E}E\oplus H_\nabla$, where
$T^{q_E}E\subseteq TE$ is the set of vectors tangent to the fibers of
the vector bundle $E$, and $H_\nabla$ is a subbundle of $TE\to E$ that
is also closed under the addition in $TE\to TM$.  There exists then for each vector field
$X\in\mx(M)$ a unique section $X^\nabla\in\Gamma(H_\nabla)\subseteq \mx(E)$ (a
\emph{horizontal} vector field) such that $Tq_E\circ X^\nabla=X\circ
q_E$. The Lie bracket of two such vector fields
$X^\nabla,Y^\nabla\in\Gamma(H_\nabla)$, for $X,Y\in\mx(M)$, is 
given by
\begin{equation*}\label{eq_1}
\left[X^\nabla, Y^\nabla\right]=[X,Y]^\nabla-\widetilde{R_\nabla(X,Y)},
\end{equation*}
where $\widetilde{R_\nabla(X,Y)}\in\mx(E)$ is given by 
\[\widetilde{R_\nabla(X,Y)}(e_m)=\left.\frac{d}{dt}\right\an{t=0}e_m+t\cdot R_\nabla(X,Y)(e_m)
\]
for all $e_m\in E$, and so has values in the vertical space
$T^{q_E}E$.  Since $\Gamma(H_\nabla)$ is generated as a
$C^\infty(E)$-module by the set of sections $\{X^\nabla \mid
X\in\mx(M)\}$, this means that the failure of the horizontal space
$H_\nabla$ to be involutive is measured by the curvature of the
connection.  The connection itself encodes the Lie bracket of
horizontal and vertical vector fields. The space $\Gamma(T^{q_E}E)$ is
indeed generated as a $C^\infty(E)$-module by the \emph{vertical}
vector fields $e^\uparrow$ with flow
$\phi^{e^\uparrow}_t(e'_m)=e_m'+te(m)$ for $e\in \Gamma(E)$,
and we have 
\begin{equation*}
\left[X^\nabla, e^\uparrow\right]=(\nabla_Xe)^\uparrow
\end{equation*}
for all $X\in\mx(M)$.

This paper answers the following question: what can be said about
linear\footnote{The subbundle $L\subseteq TE\oplus T^*E$ over $E$ is
  said to be linear if it is also closed under the addition of
  $TE\oplus T^*E$ as a vector bundle over $TM\oplus E^*$.} splittings
\[TE\oplus T^*E\simeq
(T^{q_E}E\oplus(T^{q_E}E)^\circ)\oplus L\] 
of the standard Courant algebroid over $E$?

Our first main result is a similar one-to-one correspondence
of such linear splittings with \emph{$TM\oplus E^*$--Dorfman
  connections $\Delta$ on $E\oplus T^*M$}. Then we prove that the
bundle $L_\Delta$ is isotropic (and thus also Lagrangian) relative to
the canonical pairing on $TE\oplus T^*E$ if and only if a bracket on
sections of $TM\oplus E^*$, that is dual of the Dorfman connection (in
the sense of connections), is skew-symmetric. Further, the set of
sections of $L_\Delta$ is closed under the Courant-Dorfman bracket if
and only if the curvature of the Dorfman connection vanishes.  The
Dorfman connection itself is the Courant-Dorfman bracket restricted to
horizontal and vertical sections of $TE\oplus T^*E\to E$.

The direct sum $TE\oplus T^*E$ has the structure of a \emph{double
  vector bundle} \cite{Pradines77, Mackenzie05} over the bases $E$ and
$TM\oplus E^*$.  Double vector subbundles of
$(TE\oplus T^*E;E,TM\oplus E^*,M)$ have a double vector bundle
structure over subbundles of $E$ and $TM\oplus E^*$.  After proving
the main results on splittings of $TE\oplus T^*E\to E$, we
characterise the double vector subbundles of $TE\oplus T^*E$ over the
sides $E$ and a subbundle $U\subseteq TM\oplus E^*$.  These double
vector subbundles can be described by triples $(U, K,\Delta)$, where
$\Delta$ is a Dorfman connection and $K$ is a subbundle of
$E\oplus T^*M$ (the \emph{core} or \emph{double kernel} of
$TE\oplus T^*E$).  We prove that both maximal isotropy and
integrability of this type of double subbundle depend only on simple
properties of the corresponding triple $(U,K,\Delta)$.

For instance, assume that the vector
bundle $E$ is endowed with a linear Poisson structure
$\{\cdot\,,\cdot\}$ (see Appendix~\ref{appendix_linear_Poisson}). It
might then be useful (for instance in geometric mechanics), to
describe in terms of linear and vertical vector fields the image of
the morphism of vector bundles
\[ \sharp\colon T^*E\to TE, \quad \dr F\mapsto \{F,\cdot\}, \quad F\in C^\infty(E)
\]
associated to the Poisson bracket.  The cotangent space $T^*E$ is
spanned by the exact sections $q_E^*\dr f$ for $f\in C^\infty(M)$ and
$\dr\ell_\varepsilon$ for $\varepsilon\in\Gamma(E^*)$, where
$\ell_\varepsilon\in C^\infty(E)$ is the linear function defined by
$\varepsilon\in\Gamma(E^*)$.  The images of $\sharp$ on these
one-forms are easy to describe, but there is in general no way of
finding a splitting of $TE$ such that they can be described in terms
of horizontal and vertical vector fields.  It turns out that the
\emph{graph} of the vector bundle morphism $\sharp$, can be shown to
have a vertical and a horizontal part in the space $TE\oplus T^*E$.
This is described in Examples~\ref{lie_algebroid_dual} and \ref{lie_algebroid_dual}.

\medskip Note that $TE\oplus T^*E$ has the natural structure of a
\emph{VB-Courant algebroid} with sides $E$ and $TM\oplus E^*$ and with
core $E\oplus T^*M$. We show in \cite{Jotz15} that the Dorfman
connections that we find here define (after a skew-symmetrisation) the
split Lie $2$-algebroids which are equivalent to decompositions of the
VB-Courant algebroid $TE\oplus T^*E$ \cite{Li-Bland12}.

\bigskip

If the vector bundle $E=:A$ has a Lie algebroid structure
$(q_A\colon A\to M, \rho, [\cdot\,,\cdot])$, then the standard Courant
algebroid $TA\oplus T^*A$ also has a naturally induced VB-algebroid
structure over $TM\oplus A^*$.  Given a $TM\oplus A^*$-Dorfman
connection $\Delta$ on $A\oplus T^*M$, we compute the representation
up to homotopy that corresponds to the linear splitting
$TA\oplus T^*A\simeq (T^{q_A}A\oplus (T^{q_A}A)^\circ)\oplus L_\Delta$
and describes the VB-algebroid $TA\oplus T^*A\to TM\oplus A^*$
\cite{GrMe10a}. This representation up to homotopy is in general not
the product of the two representations up to homotopy describing
$TA\to TM$ and $T^*A\to A^*$.  Furthermore, we describe the
sub-representations up to homotopy defined by linear Dirac structures
on $A$, that are at the same time Lie subalgebroids of
$TA\oplus T^*A\to TM\oplus A^*$ over a base $U\subseteq TM\oplus A^*$.
In that case, the Dirac structure has the induced structure of a
double Lie algebroid \cite{Jotz15}, and is called an \textbf{LA-Dirac
  structure} on $A$.  We elaborate on this in \cite{Jotz14} to
describe infinitesimally Dirac groupoids, i.e.~Lie groupoids with
Dirac structures that are compatible with their multiplication.

Let $(A,A^*)$ be a Lie bialgebroid \cite{MaXu00} and let $\pi_A$ be
the linear Poisson bivector field defined on $A$ by the Lie algebroid
structure on $A^*$.  The graph of $\pi_A^\sharp\colon T^*A\to TA$ is a
known example of an LA-Dirac structure on $A$.  The second most common
example of an LA-Dirac structure is the graph of a linear
presymplectic form $\sigma^*\omega_{\rm can}\in \Omega^2(A)$, for an
IM--$2$--form $\sigma\colon A\to T^*M$ \cite{BuCrWeZh04, BuCaOr09}.  A
third example is $F_A\oplus F_A^\circ$, where $F_A\to A$ is an
involutive subbundle that has at the same time a Lie algebroid
structure over some subbundle $F_M\subseteq TM$.  We describe the
2-term representations up to homotopy encoding linear splittings of
the three examples above.

\subsubsection*{Outline of the paper}
Some background on Courant algebroids and Dirac structures,
connections, and double vector bundles is collected in the second
section.  In the third section, Dorfman connections and dull
algebroids are defined, and some examples are discussed.  In the fourth
section, splittings of the standard Courant algebroid $TE\oplus T^*E$
over a vector bundle $E$ are shown to be equivalent to a certain class
of $TM\oplus E^*$-Dorfman connections on $E\oplus T^*M$.  Linear Dirac
structures on the vector bundle $E\to M$ are studied via Dorfman
connections. In the fifth section, the geometric structures on the
two sides of the standard LA-Courant algebroid $TA\oplus T^*A$ over a
Lie algebroid $A\to M$ are expressed via splittings of $TA\oplus
T^*A$, and LA-Dirac structures on $A$ are classified via Dorfman
connections and some adequate vector bundles over the units $M$.

\subsubsection*{Notation and conventions}
Let $M$ be a smooth manifold. We denote by $\mx(M)$ and
$\Omega^1(M)$ the spaces of smooth sections of the tangent and
the cotangent bundle, respectively. For an arbitrary vector bundle
$E\to M$, the space of sections of $E$ is written as
$\Gamma(E)$.
We write in general $q_E\colon E\to M$ for vector bundle projections, except
for $p_M=q_{TM}\colon TM\to M$, $c_M=q_{T^*M}\colon T^*M \to M$ and
$\pi_M=q_{TM\oplus T^*M}\colon TM\oplus T^*M\to M$.

The flow of a vector field $X\in\mx(M)$ is written as
$\phi^X_\cdot$, unless specified otherwise.  Let $f\colon M\to N$ be a
smooth map between two smooth manifolds $M$ and $N$.  Then two vector
fields $X\in\mx(M)$ and $Y\in\mx(N)$ are said to be
\textbf{$f$-related} if $Tf\circ X=Y\circ f$.  We then write $X\sim_f Y$.

Given a section $\varepsilon$ of $E^*$, we always write
$\ell_\varepsilon\colon E\to \R$ for the linear function associated to
it, i.e.~the function defined by $e_m\mapsto \langle \varepsilon(m),
e_m\rangle$ for all $e_m\in E$.  We write $\phi^t\colon B^*\to
A^*$ for the dual morphism to a morphism $\phi\colon A\to B$ of vector
bundles over the identity, and we write $F^*\omega$ for the
pullback of a form $\omega\in\Omega(N)$ under a smooth map $F\colon
M\to N$ of manifolds.

\subsubsection*{Acknowledgment}
The author thanks Thiago Drummond and Cristian Ortiz for many
discussions and for introducing representations up to homotopy to her.
Many thanks also to Yvette Kosmann-Schwarzbach, David Li-Bland, Kirill
Mackenzie, Rajan Mehta and Marco Zambon for interesting discussions or
remarks. Finally, the author is very thankful to Alan Weinstein for
his hospitality at UC Berkeley, and for some useful comments on an
early version of this paper.

\section{Preliminaries}
First we recall some necessary background on Courant
algebroids, on the double vector bundle structures on the tangent and
cotangent spaces $TE$ and $T^*E$ of a vector bundle $E$, and on linear
connections.
 
\subsection{Courant algebroids and Dirac structures}
A Courant algebroid \cite{LiWeXu97,Roytenberg99} over a manifold $M$
is a vector bundle $\mathsf E\to M$ equipped with a fibrewise
non-degenerate symmetric bilinear form $\langle\cdot\,,\cdot\rangle$, a
bilinear bracket $\lb\cdot\,,\cdot\rb$ on the smooth sections
$\Gamma(\mathsf E)$, and an anchor $\rho\colon \mathsf E\to TM$, which
satisfy the following conditions
\begin{enumerate}
\item $\lb e_1, \lb e_2, e_3\rb\rb = \lb \lb e_1, e_2\rb, e_3\rb + \lb e_2, \lb e_1, e_3\rb\rb$,
\item    $\rho(e_1 )\langle e_2, e_3\rangle = \langle\lb e_1, e_2\rb, e_3\rangle 
+ \langle e_2, \lb e_1 , e_3\rb\rangle$,
\item $\lb e_1, e_2\rb +\lb e_2, e_1\rb =\mathcal D\langle e_1 , e_2\rangle$
\end{enumerate}
for all $e_1, e_2, e_3\in\Gamma(\mathsf E)$.
Here, we  use the notation $\mathcal D := \rho^t\circ\dr \colon 
C^\infty(M)\to\Gamma(\mathsf E)$, using $\langle\cdot\,,\cdot\rangle$
to identify $\mathsf E$ with $\mathsf E^*$:
$\langle \mathcal D f, e\rangle=\rho(e)( f)$
for all $ f\in
C^\infty(M)$ and $e\in\Gamma(E)$.
The compatibility of the bracket with the anchor and the Leibniz identity
\begin{enumerate}
\setcounter{enumi}{3}
\item   $\rho(\lb e_1, e_2\rb) = [\rho(e_1), \rho(e_2)]$,
\item    $\lb e_1,  f e_2\rb =  f \lb e_1 , e_2\rb + (\rho(e_1 ) f )e_2$
\end{enumerate}
are then also satisfied. They are often part of the definition in the
literature, but \cite{Uchino02} observed that they follow from
(1)-(3).\footnote{We quickly give here a simple manner to get (4)-(5)
  from (1)-(3).  To get (5), replace $e_2$ by $ f e_2$ in (2). Then
  replace $e_2$ by $ f e_2$ in (1) in order to get (4).  } For a nice
overview of the history of Courant algebroids, consult
\cite{Kosmann-Schwarzbach13}.

\begin{example}\label{ex_pontryagin}[\cite{Courant90a}]
The direct sum $TM\oplus T^*M$ 
endowed with the projection on $TM$ as anchor map, $\rho=\pr_{TM}$, 
the symmetric bracket 
$\langle\cdot\,,\cdot\rangle$
given by 
\begin{equation}
\langle(v_m,\theta_m), (w_m,\eta_m)\rangle=\theta_m(w_m)+\eta_m(v_m)
\label{sym_bracket}
\end{equation}
for all $m\in M$, $v_m,w_m\in T_mM$ and $\alpha_m,\beta_m\in T_m^*M$
and the \textbf{Courant-Dorfman bracket} 
given by 
\begin{align}
\lb (X,\theta), (Y,\eta)\rb&=\left([X,Y], \ldr{X}\eta-\ip{Y}\dr\theta\right)\label{wrong_bracket}
\end{align}
for all $(X,\theta), (Y, \eta)\in\Gamma(TM\oplus T^*M)$,
yield the standard example of  a Courant algebroid, which is often called 
the \textbf{standard Courant algebroid over $M$}. 
The map $\mathcal D\colon  C^\infty(M)\to \Gamma(TM\oplus T^*M)$
is given by $\mathcal D f=(0, \dr f)$.

We are particularly interested in the standard
Courant algebroids over vector bundles.
\end{example}

A \textbf{Dirac structure} $\mathsf D\subseteq \mathsf E$ is a subbundle 
satisfying 
\begin{enumerate}
\item $\mathsf D^\perp=\mathsf D$ relative to the pairing on $\mathsf E$,
\item $\lb \Gamma(\mathsf D), \Gamma(\mathsf D)\rb\subseteq \Gamma(\mathsf D)$.
\end{enumerate}
The rank of the Dirac bundle $\mathsf D$ is then half the rank of
$\mathsf E$, and the triple \linebreak $(\mathsf D\to M,
\rho\an{\mathsf D}, \lb \cdot\,,\cdot\rb\an{\Gamma(\mathsf
  D)\times\Gamma(\mathsf D)})$ is a Lie algebroid on $M$.  Dirac
structures appear naturally in several contexts in geometry and
geometric mechanics (see for instance \cite{Bursztyn11} for an
introduction to the geometry and applications of Dirac structures).

\subsection{Basic facts about connections}
In this paper, connections will not be linear actions of Lie
algebroids, but more generally of \textbf{dull algebroids}.
\begin{definition}
  A \textbf{dull algebroid} is a vector bundle $Q\to M$ endowed with
  an \textbf{anchor}, i.e.~a vector bundle morphism
  $\rho_Q\colon Q\to TM$ over the identity on $M$ and a bracket
  $[\cdot\,,\cdot]_Q$ on $\Gamma(Q)$ with
\begin{equation*}
  \rho_Q[q_1,q_2]_Q=[\rho_Q(q_1),\rho_Q(q_2)]
\end{equation*}
for all $q, q'\in\Gamma(Q)$, and satisfying the Leibniz identity in
both terms
\begin{equation*} [ f_1 q_1, f_2 q_2]_Q= f_1 f_2[q_1, q_2]_Q+
  f_1\rho_Q(q_1)( f_2)q_2- f_2\rho_Q(q_2)( f_1)q_1
\end{equation*}
for all $ f_1, f_2\in C^\infty(M)$, $q_1, q_2\in\Gamma(Q)$.
\end{definition}
In other words, a dull algebroid is a Lie algebroid if its bracket is
in addition skew-symmetric and satisfies the Jacobi-identity.

\medskip

Let $(Q\to M,\rho_Q, [\cdot,\cdot]_Q)$ be a dull algebroid
and $B\to M$ a vector bundle.
A $Q$-connection on $B$ is a map
\[\nabla\colon \Gamma(Q)\times\Gamma(B)\to\Gamma(B),
\]
with the usual properties.
By the properties of a dull algebroid, one can still make sense
of the curvature $R_\nabla$ of the connection, which is an element of 
$\Gamma(Q^*\otimes Q^*\otimes B^*\otimes B)$.
The dual connection $\nabla^*\colon \Gamma(Q)\times\Gamma(B^*)\to\Gamma(B^*)$ to $\nabla$
is defined by 
\[\langle \nabla^*_q\beta, b\rangle=\rho_Q(q)\langle \beta, b\rangle- \langle\beta, \nabla_qb\rangle
\]
for all $q\in\Gamma(Q)$, $b\in\Gamma(B)$ and $\beta\in\Gamma(B^*)$.

\subsubsection{The Bott connection associated to a subbundle $F\subseteq
  TM$}
Recall the definition of the Bott connection 
associated to an involutive subbundle of $TM$:
Let  $F\subseteq TM$ be a subbundle, then the Lie bracket on vector fields on $M$ 
induces a map
\[\tilde\nabla^F\colon \Gamma(F)\times\Gamma(TM)\to\Gamma(TM/F), \qquad \tilde\nabla^F_XY=\overline{[X,Y]}.
\]
The subbundle $F$ is involutive if and only if 
$\tilde\nabla^F_XX'=0$ for all $X,X'\in\Gamma(F)$.
In that case, the map $\tilde \nabla^F$
quotients to a flat connection
\[\nabla^F\colon \Gamma(F)\times\Gamma(TM/F)\to\Gamma(TM/F),\]
the \textbf{Bott connection}.
\subsubsection{The basic connections associated to a connection on a
  dull algebroid}\label{basic_connections}
Consider here a dull algebroid $(Q, \rho_Q, [\cdot\,,\cdot]_Q)$ together
with a connection $\nabla\colon \mx(M)\times\Gamma(Q)\to\Gamma(Q)$.
The induced \textbf{basic connections} are $Q$-connections on $Q$ and
$TM$ that are defined as follows \cite{CrFe05}.
\[\nabla^{\rm bas}=\nabla^{{\rm bas},Q}\colon \Gamma(Q)\times\Gamma(Q)\to\Gamma(Q),
\qquad \nabla^{\rm bas}_qq'=[q,q']_Q+\nabla_{\rho_Q(q')}q\]
and
\[\nabla^{\rm bas}=\nabla^{{\rm bas},TM}\colon \Gamma(Q)\times\mx(M)\to\mx(M),
\qquad \nabla^{\rm bas}_qX=[\rho_Q(q),X]+\rho_Q(\nabla_{X}q).\] The
\textbf{basic curvature} is the map $R_\nabla^{\rm bas}\colon
\Gamma(Q)\times\Gamma(Q)\times\mx(M)\to\Gamma(Q)$,
\[R_\nabla^{\rm
  bas}(q,q')(X)=-\nabla_X[q,q']_Q+[\nabla_Xq,q']_Q+[q,\nabla_Xq']_Q+\nabla_{\nabla_{q'}^{\rm
    bas}X}q
-\nabla_{\nabla^{\rm bas}_qX}q'.
\]
The basic curvature is tensorial and we have the identities
\[ \nabla^{{\rm bas},TM}\circ \rho_Q=\rho_Q\circ \nabla^{{\rm bas},Q}, \qquad 
\rho_Q\circ R_\nabla^{\rm bas}=R_{\nabla^{{\rm bas},TM}}\quad
\text{ and }\quad 
R_\nabla^{\rm bas}\circ\rho_Q=R_{\nabla^{{\rm bas},Q}}.
\]
If the bracket $[\cdot\,,\cdot]_Q$ on $\Gamma(Q)$ is skew-symmetric,
then $R_\nabla^{\rm bas}$ is an element of\linebreak
$\Omega^2(Q,\operatorname{Hom}(TM,Q))$.

\subsection{Double vector bundles, VB-algebroids and representations up to homotopy}\label{background_DVB_etc}
We briefly recall the definitions of double vector bundles, of their
\textbf{linear} and \textbf{core} sections, and of their \textbf{linear
  splittings} and \textbf{lifts}. We refer to
\cite{Pradines77,Mackenzie05,GrMe10a} for more detailed treatments.
A \textbf{double vector bundle} is a commutative square
\begin{equation*}
\begin{xy}
\xymatrix{
D \ar[r]^{\pi_B}\ar[d]_{\pi_A}& B\ar[d]^{q_B}\\
A\ar[r]_{q_A} & M}
\end{xy}
\end{equation*}
of vector bundles such that
\begin{equation}\label{add_add} (d_1+_Ad_2)+_B(d_3+_Ad_4)=(d_1+_Bd_3)+_A(d_2+_Bd_4)
\end{equation}
for $d_1,d_2,d_3,d_4\in D$ with $\pi_A(d_1)=\pi_A(d_2)$,
$\pi_A(d_3)=\pi_A(d_4)$ and $\pi_B(d_1)=\pi_B(d_3)$,
$\pi_B(d_2)=\pi_B(d_4)$.
Here, $+_A$ and $+_B$ are the additions in $D\to A$ and $D\to B$,
respectively.
The vector bundles $A$
and $B$ are called the \textbf{side bundles}. The \textbf{core} $C$ of
a double vector bundle is the intersection of the kernels of $\pi_A$
and of $\pi_B$. From \eqref{add_add} follows easily the existence of a
natural vector bundle structure on $C$ over
$M$. The
inclusion $C \hookrightarrow D$ is denoted by
$
C_m \ni c \longmapsto \overline{c} \in \pi_A^{-1}(0^A_m) \cap \pi_B^{-1}(0^B_m).
$

The space of sections
$\Gamma_B(D)$ is generated as a $C^{\infty}(B)$-module by two
special classes of sections (see \cite{Mackenzie11}), the
\textbf{linear} and the \textbf{core sections} which we now describe.
For a section $c\colon M \rightarrow C$, the corresponding
\textbf{core section} $c^\dagger\colon B \rightarrow D$ is defined as
$c^\dagger(b_m) = \tilde{0}_{\vphantom{1}_{b_m}} +_A \overline{c(m)}$, $m \in M$, $b_m \in B_m$.
We denote the corresponding core section $A\to D$ by $c^\dagger$ also,
relying on the argument to distinguish between them. The space of core
sections of $D$ over $B$ is written as $\Gamma_B^c(D)$.

A section $\xi\in \Gamma_B(D)$ is called \textbf{linear} if $\xi\colon
B \rightarrow D$ is a bundle morphism from $B \rightarrow M$ to $D
\rightarrow A$ over a section $a\in\Gamma(A)$.  The space of linear
sections of $D$ over $B$ is denoted by $\Gamma^\ell_B(D)$.  Given
$\psi\in \Gamma(B^*\otimes C)$, there is a linear section
$\tilde{\psi}\colon B\to D$ over the zero section $0^A\colon M\to A$
given by
$\widetilde{\psi}(b_m) = \tilde{0}_{b_m}+_A \overline{\psi(b_m)}$.
We call $\widetilde{\psi}$ a \textbf{core-linear section}. 

\begin{example}\label{trivial_dvb}
  Let $A, \, B, \, C$ be vector bundles over $M$ and consider
  $D=A\times_M B \times_M C$. With the vector bundle structures
  $D=q^{!}_A(B\oplus C) \to A$ and $D=q_B^{!}(A\oplus C) \to B$, one
  finds that $(D; A, B; M)$ is a double vector bundle called the
  \textbf{decomposed double vector bundle with
    core $C$}. The core sections are given by
$$
c^\dagger\colon b_m \mapsto (0^A_m, b_m, c(m)), \text{ where } m \in M, \, b_m \in
B_m, \, c \in \Gamma(C),
$$
and similarly for $c^\dagger\colon A\to D$.  The space of linear
sections $\Gamma^\ell_B(D)$ is naturally identified with
$\Gamma(A)\oplus \Gamma(B^*\otimes C)$ via
$$
(a, \psi): b_m \mapsto (a(m), b_m, \psi(b_m)), \text{ where } \psi \in
\Gamma(B^*\otimes C), \, a\in \Gamma(A).
$$

In particular, the fibered product $A\times_M B$ is a double vector
bundle over the sides $A$ and $B$ and has core $M\times 0$.
\end{example}
A \textbf{linear splitting} of $(D; A, B; M)$ is an injective morphism
of double vector bundles $\Sigma\colon A\times_M B\hookrightarrow D$
over the identity on the sides $A$ and $B$.  That every double vector
bundle admits local linear splittings was proved by \cite{GrRo09}.
Local linear splittings are equivalent to double vector bundle
charts. Pradines originally defined double vector bundles as
topological spaces with an atlas of double vector bundle charts
\cite{Pradines74a}. Using a partition of unity, he proved that
(provided the double base is a smooth manifold) this implies the
existence of a global double splitting \cite{Pradines77}. Hence, any
double vector bundle in the sense of our definition admits a (global)
linear splitting.

A linear splitting $\Sigma$ of $D$ is also equivalent to a splitting
$\sigma_A$ of the short exact sequence of $C^\infty(M)$-modules
\begin{equation}\label{fat_seq_gamma}
0 \longrightarrow \Gamma(B^*\otimes C) \hookrightarrow \Gamma^\ell_B(D) 
\longrightarrow \Gamma(A) \longrightarrow 0,
\end{equation}
where the third map is the map that sends a linear section $(\xi,a)$
to its base section $a\in\Gamma(A)$.  The splitting $\sigma_A$ is
called a \textbf{horizontal lift}. Given $\Sigma$, the horizontal lift
$\sigma_A\colon \Gamma(A)\to \Gamma_B^\ell(D)$ is given by
$\sigma_A(a)(b_m)=\Sigma(a(m), b_m)$ for all $a\in\Gamma(A)$ and
$b_m\in B$.  By the symmetry of a linear splitting, we find that a
lift $\sigma_A\colon \Gamma(A)\to\Gamma_B^\ell(D)$ is equivalent to a
lift $\sigma_B\colon \Gamma(B)\to \Gamma_A^\ell(D)$.  Given a lift
$\sigma_A\colon\Gamma(A)\to\Gamma^\ell_B(D)$, the corresponding lift
$\sigma_B\colon\Gamma(B)\to\Gamma^\ell_A(D)$ is given by
$\sigma_B(b)(a(m))=\sigma_A(a)(b(m))$ for all $a\in\Gamma(A)$,
$b\in\Gamma(B)$.

\begin{example}\label{td_splitting}
Let $q_E\colon E\to M$ be a vector bundle.  Then the tangent bundle
$TE$ has two vector bundle structures; one as the tangent bundle of
the manifold $E$, and the second as a vector bundle over $TM$. The
structure maps of $TE\to TM$ are the derivatives of the structure maps
of $E\to M$.
\begin{equation*}
\begin{xy}
\xymatrix{
TE \ar[d]_{Tq_E}\ar[r]^{p_E}& E\ar[d]^{q_E}\\
 TM\ar[r]_{p_M}& M}
\end{xy}
\end{equation*} 
The space $TE$ is a double vector bundle with core bundle
$E \to M$. The map $\bar{}\,\colon E\to p_E^{-1}(0^E)\cap
(Tq_E)^{-1}(0^{TM})$ sends $e_m\in E_m$ to $\bar
e_m=\left.\frac{d}{dt}\right\an{t=0}te_m\in T_{0^E_m}E$.
Hence the core vector field corresponding to $e \in \Gamma(E)$ is the
vertical lift $e^{\uparrow}\colon E \to TE$, i.e.~the vector field with
flow $\phi^{e^\uparrow}\colon E\times \R\to E$, $\phi_t(e'_m)=e'_m+te(m)$. An
element of $\Gamma^\ell_E(TE)=\mx^\ell(E)$ is called a \textbf{linear
  vector field}. It is well-known (see e.g.~\cite{Mackenzie05}) that a
linear vector field $\xi\in\mx^l(E)$ covering $X\in\mx(M)$ corresponds
to a derivation $D\colon \Gamma(E) \to \Gamma(E)$ over $X\in
\mx(M)$. The precise correspondence is
given by the following equations
\begin{equation}\label{ableitungen}
\xi(\ell_{\varepsilon}) 
= \ell_{D^*(\varepsilon)} \,\,\,\, \text{ and }  \,\,\, \xi(q_E^*f)= q_E^*(X(f))
\end{equation}
for all $\varepsilon\in\Gamma(E^*)$ and $f\in C^\infty(M)$, where
$D^*\colon\Gamma(E^*)\to\Gamma(E^*)$ is the dual derivation to $D$.
We write $\widehat D$ for the linear vector field in $\mx^l(E)$
corresponding in this manner to a derivation $D$ of $\Gamma(E)$. 
Given a derivation $D$ over $X\in\mx(M)$, the explicit formula for $\widehat D$ is 
\begin{equation}\label{explicit_hat_D}
\widehat D(e_m)=T_meX(m)+_E\left.\frac{d}{dt}\right\an{t=0}(e_m-tD(e)(m))
\end{equation}
for $e_m\in E$ and any $e\in\Gamma(E)$ such that $e(m)=e_m$.  The
choice of a linear splitting $\Sigma$ for $(TE; TM, E; M)$ is
equivalent to the choice of a connection on $E$: Since a linear
splitting gives us for each $X\in \mx(M)$ exactly one linear vector
field $\sigma_{TM}(X)\in\mx^l(E)$ over $X$, we can define
$\nabla\colon \mx(M)\times\Gamma(E)\to \Gamma(E)$ by
$\sigma_{TM}(X)=\widehat{\nabla_X}$ for all $X\in\mx(M)$. Conversely,
a connection $\nabla\colon \mx(M)\times\Gamma(E)\to\Gamma(E)$ defines
a lift $\sigma_{TM}^\nabla\colon\mx(M)\to\mx^l(E)$ and a linear
splitting $\Sigma^\nabla\colon TM\times_M E \to TE$:
\[\Sigma^\nabla(v_m,e_m)=T_mev_m+_E\left.\frac{d}{dt}\right\an{t=0}(e_m-t\nabla_{v_m}e)
\]
for any $e\in\Gamma(E)$ such that $e(m)=e_m$.
Note that the image of $\Sigma^\nabla$ is a subbundle $H_\nabla\subseteq TE$ that is
linear, i.e.~also closed under the addition in $TE\to TM$ and satisfies 
$TE\simeq H_\nabla\oplus T^{q_E}E$
as a vector bundle over $E$.
Hence we have just described the correspondence of the two definitions of a connection; 
the first as the map
$\nabla\colon \mx(M)\times\Gamma(E)\to\Gamma(E)$,
the second as a linear splitting 
$TE\simeq T^{q_E}E\oplus H$.
Given $\nabla$ or $\Sigma^\nabla$ it is easy to see, using the equalities
in \eqref{ableitungen}, that
\begin{equation}\label{Lie_bracket_over_VB}
\begin{split}
\left[\sigma^\nabla(X), \sigma^\nabla(Y)\right]&=\sigma^\nabla[X,Y|-\widetilde{R_\nabla(X,Y)},\\
\left[\sigma^\nabla(X), e^\uparrow\right]&=(\nabla_Xe)^\uparrow,\\
\left[e_1^\uparrow,e_2^\uparrow\right]&=0
\end{split}
\end{equation}
for all $X,Y\in\mx(M)$ and $e, e_1, e_2\in\Gamma(E)$.
That is, the Lie bracket of vector fields on $E$ can be described
using the connection. The connection itself can also be seen as a
suitable quotient of the Bott connection $\nabla^{H_\nabla}$:
\[\nabla^{H_\nabla}_{\sigma^{\nabla}_{TM}(X)}\overline{e^\uparrow}=\overline{(\nabla_Xe)^\uparrow}
\]
for all $e\in\Gamma(E)$ and $X\in\mx(M)$. That is, the Bott connection
associated to $H_\nabla$ restricts well to linear (horizontal) and
vertical sections.
\end{example}

\begin{example}
Dualizing $TE$ over $E$, we get 
the double vector bundle 
\begin{align*}
\begin{xy}
\xymatrix{
T^*E\ar[r]^{c_E}\ar[d]_{r_E} &E\ar[d]^{q_E}\\
E^*\ar[r]_{q_{E^*}}&M
}\end{xy}.
\end{align*}
The map $r_E$ is given as follows. For $\theta_{e_m}$, 
$r_E(\theta_{e_m})\in E^*_m$, 
\[\langle r_E(\theta_{e_m}), e'_m\rangle=\left\langle \theta_{e_m},
  \left.\frac{d}{dt}\right\an{t=0}e_m+te'_m\right\rangle
\]
for all $e'_m\in E_m$.
The addition in $T^*E\to E^*$ is defined 
as follows. If $\theta_{e_m}$ and $\omega_{e_m'}$ are such that 
$r_E(\theta_{e_m})=r_E(\omega_{e_m'})=\varepsilon_m\in E^*_m$, then the sum
$\theta_{e_m}+_{r_E}\omega_{e_m'}\in T_{e_m+e_m'}^*E$
is given by 
\[\langle  \theta_{e_m}+_{E^*}\omega_{e_m'}, v_{e_m}+_{TM}v_{e_m'}\rangle
=\langle  \theta_{e_m}, v_{e_m}\rangle+\langle\omega_{e_m'}, v_{e_m'}\rangle
\]
for all $v_{e_m}\in T_{e_m}E$, $v_{e'_m}\in T_{e'm}E$ such that
$(q_E)_*(v_{e_m})=(q_E)_*(v_{e_m'})$.  

For $\varepsilon\in\Gamma(E^*)$, the one-form $\dr\ell_\varepsilon$ is
linear over $\varepsilon$, and for $\theta\in\Omega^1(M)$, the
one-form $q_E^*\theta$ is a core section of $TE\to E$.  We have
$r_E(\dr_{e_m}\ell_\varepsilon)=\varepsilon(m)$ and
$r_E((q_E^*\theta)(e_m))=0^{E^*}_m$.  The sum
$\dr_{e_m}\ell_\varepsilon+_{r_E}\dr_{e_m'}\ell_\varepsilon$ equals
$\dr_{e_m+e_m'}\ell_\varepsilon$. 
The vector space $T^*_{e_m}E$ is
spanned by $\dr_{e_m}\ell_\varepsilon$ and $\dr_{e_m}(q_E^* f)$ for all
$\varepsilon\in\Gamma(E^*)$ and $ f\in C^\infty(M)$. 
\end{example}

\begin{example}
  By taking the direct sum of the two double vector bundles in the two
  preceding examples, we get a double vector bundle
\begin{align*}
\begin{xy}
\xymatrix{
TE\oplus T^*E\ar[r]^{\quad \pi_E}\ar[d]_{\Phi_E} &E\ar[d]^{q_E}\\
TM\oplus E^*\ar[r]_{\quad q_{TM\oplus E^*}}&M
}\end{xy},
\end{align*}
with $\Phi_E=(q_E)_*\oplus r_E$.  

In the following, for any section $(e,\theta)$ of $E\oplus T^*M$, the
vertical section $(e,\theta)^\uparrow\in\Gamma_E(T^{q_E}E\oplus
(T^{q_E}E)^\circ)$ is the pair defined by
\begin{equation}\label{def_of_vert_pair}
(e,\theta)^\uparrow(e_m')=\left(\left.\frac{d}{dt}\right\an{t=0}e_m'+te(m), (T_{e'_m}q_E)^t\theta(m)\right)
\end{equation}
for all $e_m'\in E$.  Note that by construction the vertical sections 
$(e,\theta)^\uparrow$ are core sections of $TE\oplus T^*E$ as a vector
bundle over $E$.

A subbundle $L$ of $TE\oplus T^*E\to E$ is said to be \textbf{linear}
if it projects to a subbundle $U\subseteq TM\oplus E^*$ under $\Phi_E$
and if it is also closed under the addition on $TE\oplus T^*E$ as a
vector bundle over $TM\oplus E^*$. Such a linear subbundle defines a
\textbf{sub double vector bundle} of $TE\oplus T^*E$.
\end{example}

\bigskip
A double vector bundle $(D;A,B;M)$ is a \textbf{VB-algebroid}
(\cite{Mackenzie98x}; see also \cite{GrMe10a}) if there are Lie
algebroid structures on $D\to B$ and $A\to M$, such that the anchor
$\Theta\colon D \to TB$ is a morphism of double vector bundles over
$\rho_A\colon A \to TM$ on one side and if the Lie bracket is linear:
\begin{equation*} [\Gamma^\ell_B(D), \Gamma^\ell_B(D)] \subset
  \Gamma^\ell_B(D), \qquad [\Gamma^\ell_B(D), \Gamma^c_B(D)] \subset
  \Gamma^c_B(D), \qquad [\Gamma^c_B(D), \Gamma^c_B(D)]= 0.
\end{equation*}
The vector bundle $A\to M$ is then also a Lie algebroid, with anchor
$\rho_A$ and bracket defined as follows: if $\xi_1,
\xi_2\in\Gamma^\ell_B(D)$ are linear over $a_1,a_2\in\Gamma(A)$, then
the bracket $[\xi_1,\xi_2]$ is linear over $[a_1,a_2]$.

\medskip

Now let $A\to M$ be a Lie algebroid and consider an $A$-connection
$\nabla$ on a vector bundle $E\to M$.  Then the space
$\Omega^\bullet(A,E)$ of $E$-valued Lie algebroid forms has an induced
operator $\dr_\nabla$ given by the Koszul formula:
\begin{equation*}
\begin{split}
  \dr_\nabla\omega(a_1,\ldots,a_{k+1})=
&\sum_{i<j}(-1)^{i+j}\omega([a_i,a_j],a_1,\ldots,\hat a_i,\ldots,\hat a_j,\ldots, a_{k+1})\\
  &\qquad +\sum_i(-1)^{i+1}\nabla_{a_i}(\omega(a_1,\ldots,\hat
  a_i,\ldots,a_{k+1}))
\end{split}
\end{equation*}
for all $\omega\in\Omega^k(A,E)$ and $a_1,\ldots,a_{k+1}\in\Gamma(A)$.

Let $e_0,e_1$ be two vector bundles over the same base $M$ as $A$. A
\textbf{$2$-term representation up to homotopy of $A$ on $E_0\oplus
  E_1$} \cite{ArCr12,GrMe10a} is the collection of
\begin{enumerate}
\item [(1)] a map $\partial\colon E_0\to E_1$,
\item [(2)] two $A$-connections, $\nabla^0$ and $\nabla^1$ on $E_0$
  and $E_1$, respectively, such that $\partial \circ \nabla^0 =
  \nabla^1 \circ \partial$, \item [(3)] an element $R \in \Omega^2(A,
  \Hom(E_1, E_0))$ such that $R_{\nabla^0} = R\circ \partial$,
  $R_{\nabla^1}=\partial \circ R$ and $\dr_{\nabla^{\Hom}}R=0$,
  where $\nabla^{\Hom}$ is the connection induced on $\Hom(E_1,E_0)$
  by $\nabla^0$ and $\nabla^1$.
\end{enumerate}
Note that Gracia-Saz and Mehta \cite{GrMe10a} defined this concept
independently and called them ``superrepresentations''.  

\medskip

Consider again a VB-algebroid $(D\to B, A\to M)$ and choose a linear
splitting $\Sigma\colon A\times_MB\to D$. Since the anchor $\Theta_B$
is linear, it sends a core section $c^\dagger$, $c\in\Gamma(C)$ to a
vertical vector field on $B$.  This defines the \textbf{core-anchor}
$\partial_B\colon C\to B$ given by,
$\Theta(c^\dagger)=(\partial_Bc)^\uparrow$ for all $c\in\Gamma(C)$ and
does not depend on the splitting (see \cite{Mackenzie92}). Since the
anchor $\Theta$ of a linear section is linear, for each $a\in
\Gamma(A)$ the vector field $\Theta(\sigma_A(a))\in\mx^l(B)$ defines a
derivation of $\Gamma(B)$ with symbol $\rho(a)$. This defines a linear
connection $\nabla^{AB}\colon \Gamma(A)\times\Gamma(B)\to\Gamma(B)$:
\[\Theta(\sigma_A(a))=\widehat{\nabla_a^{AB}}\]
for all $a\in\Gamma(A)$.  Recall further that the anchor
$\Theta(c^\dagger)$ of a core section $c^\dagger\in\Gamma_B^c(D)$ is
given by $\Theta(c^\dagger)=(\partial_B c)^\uparrow$.  Since the
bracket of a linear section with a core section is again a core
section, we find a linear connection
$\nabla^{AC}\colon\Gamma(A)\times\Gamma(C)\to\Gamma(C)$ such
that \[[\sigma_A(a),c^\dagger]=(\nabla_a^{AC}c)^\dagger\] for all
$c\in\Gamma(C)$ and $a\in\Gamma(A)$.  The difference
$\sigma_A[a_1,a_2]-[\sigma_A(a_1), \sigma_A(a_2)]$ is a core-linear
section for all $a_1,a_2\in\Gamma(A)$.  This defines a vector valued
form $R\in\Omega^2(A,\operatorname{Hom}(B,C))$ such that
\[[\sigma_A(a_1), \sigma_A(a_2)]=\sigma_A[a_1,a_2]-\widetilde{R(a_1,a_2)},
\]
for all $a_1,a_2\in\Gamma(A)$. For more details on these
constructions, see \cite{GrMe10a}, where the following result is
proved.
\begin{theorem}\label{rajan}
  Let $(D \to B; A \to M)$ be a VB-algebroid and choose a linear
  splitting $\Sigma\colon A\times_MB\to D$.  The triple
  $(\nabla^{AB},\nabla^{AC},R)$ defined as above is a
  $2$-term representation up to homotopy of $A$ on the complex $\partial_B\colon C\to B$.

  Conversely, let $(D;A,B;M)$ be a double vector bundle such that $A$
  has a Lie algebroid structure and choose a linear splitting
  $\Sigma\colon A\times_MB\to D$. Then if
  $(\nabla^{AB},\nabla^{AC},R)$ is a $2$-term representation up to
  homotopy of $A$ on a complex $\partial_B\colon C\to B$, then the
  equations above define a VB-algebroid structure on $(D\to B; A\to
  M)$.

\end{theorem}

\begin{example}\label{tangent_double_2_reps}
  Let $E\to M$ be a vector bundle.  The tangent double $(TE;E,TM;M)$
  has a VB-algebroid structure $(TE\to E, TM\to M)$. Consider a linear
  splitting $\Sigma\colon E\times_M TM\to TE$ and the corresponding
  linear connection $\nabla\colon \mx(M)\times\Gamma(E)\to\Gamma(E)$
  as in Example~\ref{td_splitting}.  By \eqref{Lie_bracket_over_VB}, the
  representation up to homotopy corresponding to this splitting is
  given by $\partial_E=\id_E\colon E\to E$,
  $(\nabla,\nabla,R_\nabla)$.
\end{example}
\begin{example}\label{tangent_LA}
  Now assume that the vector bundle $E$ is a Lie algebroid $A$. 
Then the tangent prolongation $(TA\to TM, A\to M)$ has a VB-algebroid
structure; see Appendix~\ref{big_lie_algebroids}.
The linear splitting
  corresponding to a linear connection
  $\nabla\colon\mx(M)\times\Gamma(A)\to\Gamma(A)$ defines a horizontal
  lift $\sigma_{A}\colon\Gamma(A)\to\Gamma_{TM}^l(TA)$.  The
  corresponding $2$-term representation up to homotopy is given by
  $\partial_{TM}=\rho\colon A\to TM$, $(\nabla^{\rm bas}, \nabla^{\rm
    bas}, R_\nabla^{\rm bas})$, where $\nabla^{\rm bas}\colon
  \Gamma(A)\times\Gamma(A)\to\Gamma(A)$ and $\nabla^{\rm bas}\colon
  \Gamma(A)\times\mx(M)\to\mx(M)$ are the basic connections associated
  to $\nabla$. 
\end{example}

\begin{example}
Let $(A,\rho, [\cdot\,,\cdot])$ be a Lie algebroid over a smooth
manifold $M$. 
Then $(T^*A\to A^*, A\to M)$  is naturally a
VB-algebroid; see Appendix~\ref{big_lie_algebroids}. A linear splitting $\Sigma^\nabla$ of
$TA$ can be dualised to a linear splitting $\Sigma_\nabla^\star\colon
A\times_M A^*\to T^*A$. In this splitting, the VB-algebroid structure
is equivalent to the $2$-representation of $A$ on
the complex $\rho^t\colon T^*M\to A^*$ that is defined by the
connections
\begin{equation}\label{ruth_2a}
\begin{split} {\nabla^{\rm bas}}^*\colon \Gamma(A)\times\Gamma(A^*)\to \Gamma(A^*),
\qquad {\nabla^{\rm bas}}^*\colon \Gamma(A)\times\Omega^1(M)\to \Omega^1(M),
\end{split}
\end{equation} 
and the curvature term
\begin{equation}\label{ruth_2b}
-{R_{\nabla}^{\rm bas}}^t\in \Omega^2(A, \Hom(A^*,T^*M)). 
\end{equation}
For more details, consult \cite{GrJoMaMe14}.
\end{example}
\begin{example}\label{up_to_now_ex}
The linear splittings of $TA$ and $T^*A$ described in the previous
examples define a linear splitting of the VB-algebroid $(TA\oplus
T^*A\to TM\oplus A^*\to TM\oplus A^*, A\to M)$, the fibered product of
$TA\to TM $ and $T^*A\to A^*$. The representations up to homotopy
found in these two examples sum up to a representation up to homotopy
of $A$ on the complex $(\rho,\rho^t)\colon A\oplus T^*M \to TM\oplus
A^*$, which describes the VB-algebroid in this linear splitting.

One application of our main results is a general
description of linear splittings of $TA\oplus T^*A$, and explicit
formulas for the corresponding representations up to homotopy (see
Section~\ref{new_2_reps}).
\end{example}

\section{Dorfman connections: definition and examples}

\begin{definition}
Let $(Q\to M,\rho_Q,[\cdot\,,\cdot]_Q)$ be a dull algebroid.
 Let $B\to M$ be a vector bundle with a fiberwise
pairing $\langle\cdot\,,\cdot\rangle\colon Q\times_M B\to \R$ and a map 
$\dr_B\colon  C^\infty(M)\to \Gamma(B)$ such that 
\begin{equation}\label{compatibility_anchor_pairing}
\langle q, \dr_B f\rangle=\rho_Q(q)( f)
\end{equation}
for all $q\in\Gamma(Q)$ and $ f\in C^\infty(M)$.
Then $(B,\dr_B, \langle\cdot\,,\cdot\rangle)$  is called a \textbf{pre-dual} 
of $Q$ and 
$Q$ and $B$ 
are said to be \textbf{paired by $\langle\cdot\,,\cdot\rangle$}.
\end{definition}

\begin{remark}
  Note that if the pairing is non-degenerate, then $(B\to M,\dr_B,
  \langle\cdot\,,\cdot\rangle)$ is isomorphic to the \textbf{dual} of $(Q\to
  M,\rho_Q,[\cdot\,,\cdot]_Q)$ and $\dr_{Q^*}\colon
  C^\infty(M)\to\Gamma(Q^*)$ is \emph{defined} by
  \eqref{compatibility_anchor_pairing}, namely $\dr_{Q^*}
  f=\rho_Q^t\dr f$.
\end{remark}

The following is our main definition.
\begin{definition}\label{the_def}
  Let $(Q\to M,\rho_Q,[\cdot\,,\cdot]_Q)$ be a dull algebroid
  and  $(B\to M, \dr_B, \langle\cdot\,,\cdot\rangle)$ a pre-dual of $Q$.
\begin{enumerate}
\item A \textbf{Dorfman ($Q$-)connection on $B$} 
is an $\R$-bilinear
  map
\begin{equation*}
  \Delta\colon \Gamma(Q)\times\Gamma(B)\to\Gamma(B)
\end{equation*} 
such that 
\begin{enumerate}
\item $\Delta_{ f q}b= f\Delta_qb+\langle q, b\rangle \cdot \dr_B  f$,
\item $\Delta_q( f b)= f\Delta_qb+\rho_Q(q)( f)b$ and 
\item $\Delta_q(\dr_B f)=\dr_B(\ldr{\rho_Q(q)} f)$
\end{enumerate}
for all $ f\in C^\infty(M)$, $q,q'\in\Gamma(Q)$, $b\in\Gamma(B)$.
\item The curvature of $\Delta$ is the map
  \[R_\Delta\colon \Gamma(Q)\times\Gamma(Q)\to\Gamma(B^*\otimes B),\] defined
  on $q,q'\in\Gamma(Q)$ by
  $R_\Delta(q,q'):=\Delta_q\Delta_{q'}-\Delta_{q'}\Delta_q-\Delta_{[q,q']_Q}$.
\end{enumerate}
\end{definition}
The failure of a Dorfman connection to be a connection is hence
measured by the map $\dr_B$ and the pairing of $Q$ with $B$.  We omit
the proof of the following proposition.
\begin{proposition}\label{curvature_tensor}
Let $(Q\to M,\rho_Q,[\cdot\,,\cdot]_Q)$ be a dull algebroid and $(B,\dr_B,\langle\cdot\,,\cdot\rangle)$ a pre-dual of $Q$.
Let $\Delta$ be a Dorfman $Q$-connection on $B$. 
Then:
\begin{enumerate}
\item For all $ f\in C^\infty(M)$ and $q,q'\in \Gamma(Q)$, $b\in \Gamma(B)$, we have 
$R_\Delta(q,q')( f\cdot b)= f\cdot R_\Delta(q,q')$. 
\item For all $q_1,q_2,q_3\in\Gamma(Q)$ and $b\in\Gamma(B)$, we have 
\[\langle R_\Delta(q_1,q_2)(b),q_3\rangle=\langle[[q_1,q_2]_Q,q_3]_Q+[q_2,[q_1,q_3]_Q]_Q-[q_1,[q_2,q_3]_Q]_Q, b\rangle.
\]
\end{enumerate}
\end{proposition}

Note that this does not mean that the curvature of the Dorfman connection vanishes everywhere 
if $Q$ is a Lie algebroid, since 
the pairing of $Q$ and $B$ can be degenerate. The following example is a trivial example for this phenomenon.

\begin{example}\label{trivial}
Let $(Q\to M, \rho_Q, [\cdot\,,\cdot]_Q)$ be a dull algebroid and
$B\to M$ a vector bundle. Take the pairing
$\langle\cdot\,,\cdot\rangle\colon Q\times_M B\to \R$ and the map
$\dr_B\colon C^\infty(M)\to \Gamma(B)$ to be trivial. Then any $Q$-connection 
 on $B$ is also a Dorfman connection.
\end{example}

\begin{example}
The  easiest non-trivial example of a Dorfman connection 
is the map
$\ldr{}\colon \Gamma(Q)\times\Gamma(Q^*)\to\Gamma(Q^*)$, 
\[\langle\ldr{q}\tau,q'\rangle=\rho_Q(q)\langle q, \tau\rangle-\langle\tau, [q,q']_Q\rangle,
\]
for a dull algebroid $(Q\to M, \rho_Q,[\cdot\,,\cdot]_Q)$
and its dual $(Q^*,\dr_{Q^*})$, i.e.~with 
the canonical pairing $Q\times_MQ^*\to\R$ and 
$\dr_{Q^*}=\rho_Q^t\dr\colon C^\infty(M)\to\Gamma(Q^*)$.

The third property of a Dorfman connection is immediate by definition
of $\ldr{}$ and the first two properties are easily verified. The
curvature vanishes if and only if $[\cdot\,,\cdot]_Q$ satisfies the
Jacobi-identity in Leibniz form
$[[q_1,q_2]_Q,q_3]_Q+[q_2,[q_1,q_3]_Q]_Q=[q_1,[q_2,q_3]_Q]_Q$ for all
$q_1,q_2,q_3\in\Gamma(Q)$.
\end{example}

\bigskip The following proposition illustrates the general idea that
Dorfman connections are to Courant algebroids what linear connections
are to Lie algebroids. Our main result in Section~\ref{main} is a
further example for this analogy.

Let $(\mathsf E\to M, \rho\colon \mathsf E\to TM,
\langle\cdot\,,\cdot\rangle, \lb \cdot\,,\cdot\rb)$ be  a Courant algebroid.  If $K$ is a subalgebroid
of $\mathsf E$, the (in general singular) distribution $S:=\rho(K)\subseteq TM$ is
algebraically involutive 
and we can define the ``singular'' Bott connection
\[\nabla^S\colon \Gamma(S)\times\frac{\mx(M)}{\Gamma(S)}\to\frac{\mx(M)}{\Gamma(S)}
\]
by 
\[\nabla^S_{s}\bar X=\overline{[s,X]}\]
for all $X\in\mx(M)$ and $s\in\Gamma(S)$.
The anchor $\rho\colon \mathsf E\to TM$ induces a map $\bar\rho\colon \Gamma(\mathsf E/K)\to
\mx(M)/\Gamma(S)$, $\bar\rho(\bar e)=\rho(e)+\Gamma(S)$.

\begin{proposition}\label{Bott_Dorfman}
  Let $\mathsf E\to M$ be a Courant algebroid and $K\subseteq \mathsf E$ an
  isotropic subalgebroid.  Then the map
\begin{align*}
  \Delta\colon \Gamma(K)\times\Gamma(\mathsf E/K)&\to\Gamma(\mathsf
                                                   E/K),\qquad \Delta_k\bar e=\overline{\lb k,e\rb}
\end{align*}
is a Dorfman connection. The dull algebroid structure on $K$ is its
induced Lie algebroid structure, the map $\dr_{\mathsf E/K}$ is just 
$\mathcal D+\Gamma(K)$  and the pairing $\langle\cdot\,,\cdot\rangle\colon K\times_M (\mathsf E/K)\to \R$
is the natural pairing induced by the pairing on $\mathsf E$.

We have 
\[\bar\rho(\Delta_k\bar e)=\nabla^S_{\rho(k)}\bar\rho(\bar e)
\]
for all $k\in\Gamma(K)$ and $\bar e\in\Gamma(\mathsf E/K)$.
\end{proposition}

\begin{remark}
\begin{enumerate}
\item 
Because of the analogy of the Dorfman connection in the last proposition
with the Bott connection defined by involutive subbundles of $TM$, 
we name this  Dorfman connection the \textbf{Bott--Dorfman connection
  associated to $K$}.
\item Note that if $K$ is a Dirac structure $\mathsf D$ in $\mathsf
  E$, then $\mathsf E/\mathsf D\simeq \mathsf D^*$ and the Dorfman
  connection is just the Lie algebroid derivative of $\mathsf D$ on
  $\Gamma(\mathsf D^*)$.
\end{enumerate}
\end{remark}

\section{Linear splittings of \texorpdfstring{$TE\oplus T^*E$}{$TE+T^*E$}}\label{main}
Consider a vector bundle $q_E\colon E\to M$.  Recall from Example
\ref{td_splitting} that an ordinary connection $\nabla\colon
\mx(M)\times\Gamma(E)\to\Gamma(E) $ is equivalent to a linear
splitting $\Sigma\colon E\times_M TM\to TE$.  We show
that a Dorfman connection $\Delta\colon \Gamma(TM\oplus
E^*)\times\Gamma(E\oplus T^*M)\to \Gamma(E\oplus T^*M) $ is the same
as a linear splitting $\Sigma\colon (TM\oplus E^*)\times_M E\to
TE\oplus T^*E$.  Further, we show that the image $L_\Delta$ of
$\Sigma$ in $TE\oplus T^*E$ is maximally isotropic relatively to the
canonical pairing if and only if the bracket
$\lb\cdot\,,\cdot\rb_\Delta$ dual\footnote{Since the Dorfman connection and
  the dual dull bracket corresponding to a linear splitting of
  $TE\oplus T^*E$ encode the Courant-Dorfman bracket on $E$, we write
  the dull brackets on $\Gamma(TM\oplus E^*)$ with double bars, as we
  write Courant algebroid brackets.} to the Dorfman connection is
skew-symmetric, and we show how the failure of $\Gamma(L_\Delta)$ to
be closed under the Dorfman bracket is measured by the curvature
$R_\Delta$.

\bigskip Here, the vector bundle $TM\oplus E^*$ is always
anchored by the projection $\pr_{TM}\colon TM\oplus E^*\to TM$ and the
dual $E\oplus T^*M$ is always paired with $TM\oplus E^*$ via the
canonical non-degenerate pairing.  The map $\dr_{E\oplus T^*M}\colon
C^\infty(M)\to \Gamma(E\oplus T^*M)$ is consequently always
\[\dr_{E\oplus T^*M}=\pr_{TM}^t\circ\,\dr,
\]
i.e.~$\dr_{E\oplus T^*M} f =(0,\dr f )$ for all $ f \in C^\infty(M)$.
A Dorfman connection $\Delta$ is here always a $TM\oplus E^*$-Dorfman
connection on $E\oplus T^*M$, with dual $\lb\cdot\,,\cdot\rb_\Delta$.
Note that since the pairing is non-degenerate, the Dorfman connection
is completely determined by its dual structure, the associated dull
bracket $\lb\cdot\,,\cdot\rb_\Delta$ and vice-versa.  Hence, we can
say here that a Dorfman connection is equivalent to a dull algebroid
$(TM\oplus E^*, \pr_{TM},\lb\cdot\,,\cdot\rb_\Delta)$.  It is easy to
see, using Proposition~\ref{curvature_tensor}, that the curvature
$R_\Delta$ always vanishes on $(TM\oplus E^*)\otimes(TM\oplus
E^*)\otimes(0\oplus T^*M)$ and so it can be identified with an element
of $\Omega^2(TM\oplus E^*, \operatorname{Hom}(E,E\oplus T^*M))$.

\subsection{Dorfman connection associated to a linear splitting
  of \texorpdfstring{$TE\oplus T^*E$}{TE+T*E}}
Consider  a linear splitting 
\[\Sigma\colon E\times_M(TM\oplus E^*)\to TE\oplus T^*E
\]
and the corresponding horizontal lift $\sigma_{TM\oplus E^*}\colon
\Gamma(TM\oplus E^*)\to \Gamma_E^l(TE\oplus T^*E)$.  Note that by the
definition of the horizontal lift, we have $\sigma_{TM\oplus
  E^*}( f\cdot \nu)=q_E^* f \cdot \sigma_{TM\oplus
  E^*}(\nu)$ for all $ f \in C^\infty(M)$ and
$\nu\in\Gamma(TM\oplus E^*)$.  Also by definition, $((q_E)_*,
r_E)(\sigma_{TM\oplus E^*}(X,\varepsilon)(e(m)))=(X(m),\varepsilon(m))=((q_E)_*,
r_E)(T_meX(m),\dr_{e(m)}\ell_\varepsilon)$ for all $X\in\mx(M)$,
$e\in\Gamma(E)$ and $\varepsilon\in\Gamma(E^*)$. Hence the difference
\[(T_meX(m),\dr_{e(m)}\ell_\varepsilon)-\sigma_{TM\oplus E^*}(X,\varepsilon)(e(m))
\]
 is a
core element, which can be written
\[(\delta_{(X,\varepsilon)}e)^\uparrow(e(m)),
\]
defining a map\footnote{To see that $\delta_{(X,\varepsilon)}e$ is a
  smooth section of $E\oplus T^*M$, it suffices to show that its
  pairing with each section of $TM\oplus E^*$ is smooth.  For
  $(Y,\chi)\in\Gamma(TM\oplus E^*)$, we have $\langle
  \delta_{(X,\varepsilon)}e, (Y,\chi)\rangle(m) =\langle
  \delta_{(X,\varepsilon)}e^\uparrow(e(m)),
  (T_meY(m),\dr_{e(m)}\ell_\chi)\rangle$ and so
\[\langle \delta_{(X,\varepsilon)}e, (Y,\chi)\rangle(m)=\langle 
(T_meX(m),\dr_{e(m)}\ell_\varepsilon)-\sigma_{TM\oplus E^*}(X,\varepsilon)(e(m)),
(T_meY(m),\dr_{e(m)}\ell_\chi)\rangle,
\]
which is 
\[Y(m)\langle\varepsilon, e\rangle+X(m)\langle
\chi,e\rangle-\langle\Sigma((X,\varepsilon)(m),e(m)),
(T_meY(m),\dr_{e(m)}\ell_\chi)\rangle.
\]
This depends smoothly on $m$.} $\delta\colon \Gamma(TM\oplus
E^*)\times\Gamma(E)\to \Gamma(E\oplus T^*M)$.  
\medskip

 Set $\Delta\colon
\Gamma(TM\oplus E^*)\times\Gamma(E\oplus T^*M)\to\Gamma(E\oplus
T^*M)$,
$\Delta_{(X,\varepsilon)}(e,\theta)=\delta_{(X,\varepsilon
)}e+(0,\ldr{X}\theta)$. We
prove that $\Delta$ is a Dorfman connection.  First \begin{equation}\label{multiple_function1}
\begin{split}
&(T_m( f
e)X(m), \dr_{ f(m)e(m)}\ell_\varepsilon)\\
&=(T_m( f(m)
e)X(m)+X( f)(m)e^\uparrow( f(m)e(m)),
\dr_{ f(m)e(m)}\ell_\varepsilon)
\end{split}
\end{equation} and
$\sigma_{TM\oplus E^*}(X,\varepsilon)(( f e)(m))=\sigma_{TM\oplus E^*}(X,\varepsilon)( f(m) e(m))$
yield $\delta_{(X,\varepsilon)}( f
e)= f\delta_{(X,\varepsilon)}e+X( f)(e,0)$. This implies
\[\Delta_{(X,\varepsilon)}( f(e,\theta))= f\Delta_{(X,\varepsilon)}(e,\theta)+X( f)(e,\theta)\]
for all $(X,\varepsilon)\in\Gamma(TM\oplus E^*)$, $(e,\theta)\in\Gamma(E\oplus
T^*M)$ and $ f\in C^\infty(M)$.
Then 
\begin{equation}\label{multiple_function2}
(T_me( f X),
\dr_{e(m)}\ell_{ f\varepsilon})=(T_me( f(m)X(m)),
 f(m)\dr_{e(m)}\ell_\varepsilon+\langle \varepsilon, e\rangle(m)\dr_{e(m)}(q^* f))
\end{equation}
and $\sigma_{TM\oplus E^*}( f\cdot (X,\varepsilon))=q_E^* f \cdot
\sigma_{TM\oplus E^*}(X,\varepsilon)$ yield
$\delta_{ f(X,\varepsilon)}e= f\delta_{(X,\varepsilon)}e+(0,\langle
e,\varepsilon\rangle\dr f)$. Since $\ldr{ f
  X}\theta=X( f)\theta+\langle X,\theta\rangle\dr f$, we get
\[\Delta_{ f(X,\varepsilon)}(e,\theta)= f\Delta_{(X,\varepsilon)}(e,\theta)+\langle
(e,\theta),(X,\varepsilon)\rangle(0,\dr f)\] for all
$(X,\varepsilon)\in\Gamma(TM\oplus E^*)$, $(e,\theta)\in\Gamma(E\oplus T^*M)$
and $ f\in C^\infty(M)$. The equality
$\Delta_{(X,\varepsilon)}(0,\dr f)=(0,\dr\ldr{X} f)$ is immediate.

\bigskip

Conversely let $E\to M$ be a vector bundle and consider a Dorfman connection
$\Delta\colon \Gamma(TM\oplus E^*)\times\Gamma(E\oplus T^*M)\to
\Gamma(E\oplus T^*M)$.
We want to define a linear splitting
$\Sigma\colon (TM\oplus E^*)\times_M E\to TE\oplus T^*E$
by
\begin{equation}\label{linear_splitting}
\Sigma((v_m,\varepsilon_m), e_m)=\left(T_me X(m), \dr
    \ell_\varepsilon(e_m)\right)-\Delta_{(X,\varepsilon)}(e,0)^\uparrow(e_m)
\end{equation}
for any sections $(X,\varepsilon)\in\Gamma(TM\oplus E^*)$ and $e\in\Gamma(E)$
such that $X(m)=v_m$, $\varepsilon(m)=\varepsilon_m$ and $e(m)=e_m$.  
For $X\in\mx(M)$, $\varepsilon\in\Gamma(E^*)$ and $e\in\Gamma(E)$ define the element 
\begin{equation*}
\Pi(X,\varepsilon,e)(m)=\left(T_me X(m), \dr
    \ell_\varepsilon(e_m)\right)-\Delta_{(X,\varepsilon)}(e,0)^\uparrow(e_m)
\end{equation*}
of $TE\oplus T^*E$.  By \eqref{multiple_function1} and the properties
of the Dorfman connection we have $\Pi(X,\varepsilon, f
e)(m)= f(m)\cdot_{TM\oplus E^*}\Pi(X,\varepsilon, e)(m)$ and by
\eqref{multiple_function2} we have $\Pi( f
X, f\varepsilon,e)(m)= f(m)\cdot_{E}\Pi(X,\varepsilon,e)(m)$ and for all
$ f\in C^\infty(M)$, $(X,\varepsilon)\in\Gamma(TM\oplus E^*)$ and
$e\in\Gamma(E)$. Using this, it is easy to show that the map in
\eqref{linear_splitting} is a well-defined, injective morphism of
double vector bundles. Since it is the identity on the sides, it is a
linear splitting of $TE\oplus T^*E$.

\medskip

Hence, we have proved our main theorem:
\begin{theorem}\label{relation_Delta_L}
Let $E\to M$ be a vector bundle. A linear splitting
$\Sigma\colon (TM\oplus E^*)\times_M E\to TE\oplus T^*E$
 defines a Dorfman connection 
$\Delta^\Sigma\colon \Gamma(TM\oplus E^*)\times\Gamma(E\oplus T^*M)\to\Gamma(E\oplus T^*M)$
by 
\begin{equation}\label{eq_d_s}
\Sigma((X,\varepsilon)(m), e(m))=\left( T_meX(m),
  \dr_{e_m}\ell_\varepsilon\right)-\Delta^\Sigma_{(X,\varepsilon)}(e,0)^\uparrow(e(m))
\end{equation}
and $\Delta_{(X,\varepsilon)}(0,\theta)=(0,\ldr{X}\theta)$ for all
$e\in\Gamma(E)$, $(X,\varepsilon)\in\Gamma(TM\oplus E^*)$ and
$\theta\in\Omega^1(M)$.  Conversely, each Dorfman connection
$\Delta\colon \Gamma(TM\oplus E^*)\times\Gamma(E\oplus
T^*M)\to\Gamma(E\oplus T^*M)$ defines a linear splitting
$\Sigma^\Delta\colon (TM\oplus E^*)\times_M E\to TE\oplus T^*E$ as in
\eqref{eq_d_s} and the maps
\begin{align*}
\Delta&\mapsto \Sigma^\Delta,\qquad \qquad \Delta^\Sigma\mapsfrom \Sigma
\end{align*}
are inverse to each other.
\end{theorem}

In short we have a  bijection 
\begin{equation*}
\left\{\begin{array}{c}
(TM\oplus E^*)\text{-Dorfman connections }\\
\Delta \text{  on } E\oplus T^*M
\end{array} \right\}
\leftrightarrow
 \left\{\begin{array}{c}
\text{ Linear splittings }\\ \Sigma\colon (TM\oplus E^*)\times_M E\to TE\oplus T^*E
\end{array}\right\}.
\end{equation*}
Since a $(TM\oplus E^*)$-Dorfman connection $\Delta$ on $E\oplus T^*M$
is equivalent to a dull algebroid structure $(\pr_{TM},
\lb\cdot\,,\cdot\rb_\Delta)$ on $TM\oplus E^*$, we can reformulate
this bijection as follows:
\begin{equation*}
\left\{\begin{array}{c}
\text{Dull algebroids }\\
(TM\oplus E^*, \pr_{TM}, \lb\cdot\,,\cdot\rb)
\end{array} \right\}
\leftrightarrow
 \left\{\begin{array}{c}
\text{ Linear splittings }\\ \Sigma\colon (TM\oplus E^*)\times_M E\to TE\oplus T^*E
\end{array}\right\}.
\end{equation*}

\begin{example}\label{sec_standard_almost_dorfman}
\label{standard_almost_dorfman}\label{example_easy}
  Let $E\to M$ be a vector bundle with a linear connection
  $\nabla\colon \mx(M)\times\Gamma(E)\to\Gamma(E)$.  Then the \textbf{standard
  Dorfman connection associated to $\nabla$} is the map
  \[\Delta\colon \Gamma(TM\oplus E^*)\times\Gamma(E\oplus
  T^*M)\to\Gamma(E\oplus T^*M),\]
\begin{equation*}
  \Delta_{(X,\varepsilon)}(e,\theta)=(\nabla_Xe,\ldr{X}\theta+\langle
  \nabla^*_\cdot \varepsilon, e\rangle).
\end{equation*}
The dual bracket is in this case defined by 
\[\lb(X,\varepsilon), (Y,\chi)\rb_\Delta=([X,Y], \nabla_X^*\chi-\nabla_Y^*\varepsilon)
\]
for all $(X,\varepsilon), (Y,\chi)\in \Gamma(TM\oplus E^*)$.

The curvature of the standard Dorfman connection $\Delta$ associated
to $\nabla$ is given by
\[R_\Delta((X,\varepsilon),(Y,\eta))=(R_\nabla(X,Y),
R_{\nabla^*}(\cdot,X)(\eta)-R_{\nabla^*}(\cdot,Y)(\varepsilon)).\]
As a consequence, we find easily that  $(TM\oplus E^*, \operatorname{pr}_{TM},
\lb\cdot\,,\cdot\rb_\Delta)$ is a Lie algebroid if  and only if $\nabla$ is flat.

For any section $(X,\varepsilon)\in\Gamma(TM\oplus E^*)$, 
the horizontal lift is 
\begin{align*}
  \sigma^\Delta_{TM\oplus E^*}(X,\varepsilon)(e_m) &=\left(T_meX(m),
    \dr_{e_m}\ell_\varepsilon\right)-\left(\left.\frac{d}{dt}\right\an{t=0}e_m+t\nabla_Xe,(T_{e_m}q_E)^t\langle
    \nabla^*_\cdot \varepsilon, e\rangle\right)
\end{align*}
and the subbundle $L_\Delta$ spanned by these sections is equal to $H_\nabla\oplus H_\nabla^\circ$.
Hence, the standard Dorfman connection associated to a connection $\nabla$ is the same 
as the splitting 
\[TE\oplus T^*E\cong(T^{q_E}E\oplus (T^{q_E}E)^\circ)\oplus(H_\nabla\oplus H_\nabla^\circ),
\]
the sum of a (trivial) Dirac structure and an almost Dirac structure.

Note that $H_\nabla\oplus H_\nabla^\circ$ is a Dirac structure if and
only if $\nabla$ is flat, that is, if and only if $(TM\oplus E^*,
\operatorname{pr}_{TM}, \lb\cdot\,,\cdot\rb_\Delta)$ is a Lie
algebroid. This is not a coincidence, but a special case of our next
main result in Theorem~\ref{super}.
\end{example}

Now we discuss more intricate examples of Dorfman connections
$\Delta\colon \Gamma(TM\oplus E^*)\times\Gamma(E\oplus
T^*M)\to\Gamma(E\oplus T^*M)$.  The geometric meaning of the
corresponding linear splittings will be explained later.

\begin{example}\label{lie_algebroid_dual}
  Consider a dull algebroid $(A,\rho, [\cdot\,,\cdot])$ with
  \emph{skew-symmetric bracket}. We construct a $TM\oplus A$-Dorfman
  connection $\Delta$ on $A^*\oplus T^*M$, hence corresponding to a
  linear splitting $\Sigma\colon (TM\oplus A)\times_M A^*\to
  TA^*\oplus T^*A^*$ of the Pontryagin bundle over $A^*$.  Take any
  connection $\nabla\colon \mx(M)\times \Gamma(A)\to\Gamma(A)$ and
  recall the definition of the basic connection $\nabla^{\rm
    bas}\colon \Gamma(A)\times\Gamma(A)\to\Gamma(A)$ associated to
  $\nabla$ and the dull algebroid structure on $A$:
\[\nabla_a^{\rm bas}b=[a,b]+\nabla_{\rho(b)}a
\]
for all $a,b\in\Gamma(A)$.
The Dorfman connection 
\[\Delta\colon  \Gamma(TM\oplus A)\times\Gamma(A^*\oplus T^*M)\to
\Gamma(A^*\oplus T^*M)\]
is defined by
\[\Delta_{(X,a)}(\alpha,\theta)=\left(\langle \alpha, \nabla_\cdot^{\rm
    bas}a\rangle+\nabla_X^*\alpha-\rho^t\langle \nabla_\cdot a,
  \alpha\rangle,
\ldr{X}\theta+ \langle \nabla_\cdot a, \alpha\rangle
\right).
\]
The bracket $\lb\cdot\,,\cdot\rb_\Delta$ on sections of $TM\oplus A$ is then 
given by
\[\lb(X,a), (Y,b)\rb_\Delta=\left([X,Y],
  \nabla_Xb-\nabla_Ya+\nabla_{\rho(b)}a-\nabla_{\rho(a)}b+[a,b]
\right).
\]
Since it is skew-symmetric, the image $L_\Delta$ of $\Sigma$ is in
this case maximally isotropic.  The projection $\pr_{TM}$ obviously intertwines
this bracket with the Lie bracket of vector fields.  The
curvature of this Dorfman connection is given by
\begin{align}
&\langle R_\Delta((X,a),(Y,b))(\alpha,\theta), (Z,c)\rangle
=- \langle \lb(X,a),\lb(Y,b), (Z,c)\rb_\Delta\rb_\Delta+{\rm c.p.}, (\alpha,\theta)\rangle\label{curvature_poisson}\\
=&-\langle \bigl(R_\nabla(X,Y)c- R_\nabla(\rho(a),Y)c\bigr) +{\rm c.p.}, \alpha\rangle-\langle \bigl(R_\nabla(\rho(a),\rho(b))c-R_\nabla(X,\rho(b))c\bigr) +{\rm c.p.}, \alpha\rangle\nonumber\\
&-\langle \bigl(R^{\rm bas}_\nabla(a,b)Z- R^{\rm bas}_\nabla(a,b)\rho(c)\bigr) +{\rm c.p.}, \alpha\rangle-\langle [a,[b,c]]+[b,[c,a]]+[c,[a,b]], \alpha\rangle.\nonumber
\end{align}
The proof of this formula is a rather long, but straightforward
computation and we omit it here.

The next subsection explains the signification of this example
in terms of the linear almost Poisson structure defined on $A^*$ by
the skew-symmetric dull algebroid structure.
\end{example}

\begin{example}\label{IM_2_form}
Consider a vector bundle $E\to M$ endowed with a 
vector bundle morphism $\sigma\colon E\to T^*M$ over the identity and 
a connection $\nabla\colon \mx(M)\times\Gamma(E)\to\Gamma(E)$.
Define the Dorfman connection 
\[\Delta\colon  \Gamma(TM\oplus E^*)\times\Gamma(E\oplus T^*M)\to
\Gamma(E\oplus T^*M)
\]
by
\[\Delta_{(X,\varepsilon)}(e,\theta)=(\nabla_Xe,
\ldr{X}(\theta-\sigma(e))+\langle \nabla_\cdot^*(\sigma^tX+\varepsilon), e\rangle
+\sigma(\nabla_Xe)). 
\]
The bracket $\lb\cdot\,,\cdot\rb_\Delta$ on sections
of $TM\oplus E^*$ is here given by 
\[\lb (X,\varepsilon), (Y,\eta)\rb_\Delta=([X,Y], \nabla_X^*(\eta+\sigma^tY)-\nabla_Y^*(\varepsilon+\sigma^tX)-\sigma^t[X,Y]).
\]
In this case also, the image $L_\Delta$ of $\Sigma^\Delta$ is
maximally isotropic.

Here also, we give the curvature of the Dorfman connection in terms of the 
Jacobiator of the associated bracket:
\begin{equation}\label{curvature_sigma}
\lb(X,\varepsilon), \lb(Y,\eta), (Z,\gamma)\rb_\Delta\rb_\Delta+{\rm c.p.}= \Bigl(0, R_{\nabla^*}(X,Y)(\gamma+\sigma^tZ)+{\rm c.p.}\Bigl).
\end{equation} 
Example~\ref{IM_2_form_2} shows how this Dorfman connection is related 
to the $2$-form $\sigma^*\omega_{\rm can}\in\Omega^2(E)$, where 
$\omega_{\rm can}$ is the canonical symplectic form on $T^*M$.
\end{example}

\subsection{The canonical pairing, the anchor and the Courant-Dorfman
  bracket on \texorpdfstring{$TE\oplus T^*E$}{TE+T*E}}
This section shows that the image of a linear splitting $\Sigma\colon
(TM\oplus E^*)\times_M E\to TE\oplus T^*E$ is maximally isotropic if
and only if the corresponding dull bracket
$\lb\cdot\,,\cdot\rb_\Sigma$ is skew-symmetric, and its set of
sections is closed under the Courant-Dorfman bracket if and only if
the curvature of $\Delta^\Sigma$ vanishes.

Here and later, we need the following notation.  Let $E\to M$ be a
vector bundle and
$\Delta\colon \Gamma(TM\oplus E^*)\times\Gamma(E\oplus
T^*M)\to\Gamma(E\oplus T^*M)$
a Dorfman connection.  We call
$\Skew_\Delta\in\Gamma((TM\oplus E^*)\otimes(TM\oplus E^*)\otimes
E^*)$ the tensor defined by
\[\Skew_\Delta(\nu_1,\nu_2)=\pr_{E^*}(\lb \nu_1, \nu_2\rb_\Delta+\lb \nu_2, \nu_1\rb_\Delta)
\]
for all $\nu_1,\nu_2\in\Gamma(TM\oplus E^*)$. By the Leibniz identity,
this is indeed $C^\infty(M)$-linear in both arguments.
Note that the $TM$-part of $\lb \nu_1, \nu_2\rb_\Delta+\lb \nu_2, \nu_1\rb_\Delta$
always vanishes since the Lie bracket of vector fields is skew-symmetric.

In this subsection, given a Dorfman connection $\Delta\colon
\Gamma(TM\oplus E^*)\times\Gamma(E\oplus T^*M)\to\Gamma(E\oplus
T^*M)$, we always write $\sigma^\Delta$ for the induced
horizontal lift $\sigma^\Delta_{TM\oplus E^*}\colon \Gamma(TM\oplus
E^*) \to \Gamma^l_E(TE\oplus T^*E)$.

\begin{theorem}\label{Lagrangian}
  Let $\Delta\colon \Gamma(TM\oplus E^*)\times\Gamma(E\oplus
  T^*M)\to\Gamma(E\oplus T^*M)$ be a Dorfman connection and choose $\nu,
  \nu_1,\nu_2\in\Gamma(TM\oplus E^*)$ and $\tau, \tau_1,
  \tau_2\in\Gamma(E\oplus T^*M)$.  Then
 \begin{enumerate}
\item $\left\langle \sigma^\Delta(\nu_1), \sigma^\Delta(\nu_2)\right\rangle=\ell_{
  \Skew_\Delta(\nu_1, \nu_2)}$,
\item $\left\langle \sigma^\Delta(\nu),
    \tau^\uparrow\right\rangle=q_E^*\langle \nu, \tau\rangle$,
\item $\left\langle \tau_1^\uparrow, \tau_2^\uparrow\right\rangle=0$.
\end{enumerate}
\end{theorem}

\begin{proof}
  Since the second and third equalities are immediate by
  \eqref{linear_splitting}, we prove only the first one.  We write
  $\nu_1=(X,\varepsilon)$, $\nu_2=(Y,\eta)$ and compute for any
  section $e\in\Gamma(E)$:
\begin{align*}
&\left\langle \left(T_me X(m), \dr
    \ell_\varepsilon(e_m)\right)-\Delta_{(X,\varepsilon)}(e,0)^\uparrow(e_m),
\left(T_me Y(m), \dr
    \ell_\eta(e_m)\right)-\Delta_{(Y,\eta)}(e,0)^\uparrow(e_m)\right\rangle\\
=\,&X(m)\langle \eta, e\rangle
-\langle\pr_{T^*M}\Delta_{(Y,\eta)}(e,0), X(m)\rangle- \langle
\eta(m), \pr_E\Delta_{(X,\varepsilon)}(e,0)\rangle\\
&+ Y(m)\langle \varepsilon, e\rangle
-\langle\pr_{T^*M}\Delta_{(X,\varepsilon)}(e,0), Y(m)\rangle- \langle
\varepsilon(m), \pr_E\Delta_{(Y,\eta)}(e,0)\rangle\\
=\,&\left(X\langle \eta, e\rangle
-\langle\Delta_{(Y,\eta)}(e,0), (X,\varepsilon)\rangle+ Y\langle \varepsilon, e\rangle
-\langle\Delta_{(X,\varepsilon)}(e,0), (Y,\eta)\rangle\right)(m)\\
=\,&\langle(e,0), \lb \nu_2, \nu_1\rb_\Delta+\lb \nu_1, \nu_2\rb_\Delta\rangle. \qedhere
\end{align*}
\end{proof}

\begin{corollary}\label{ss_dc}
  The dull bracket $\lb\cdot\,,\cdot\rb_\Delta$ associated to a
  Dorfman connection $\Delta$ is skew-symmetric if and only if the
  image of $\Sigma^\Delta$ is maximally isotropic in $TE\oplus
  T^*E$. The corresponding splitting
  \[ TE\oplus T^*E\cong(T^{q_E}E\oplus (T^{q_E}E)^\circ)\oplus
  L_\Delta\] is then the direct sum of the Dirac structure
  $T^{q_E}E\oplus (T^{q_E}E)^\circ$ and the linear almost Dirac
  structure $L_\Delta=\Sigma^\Delta((TM\oplus E^*)\times_ME)$.
\end{corollary}

\begin{proof}
  Since the rank of $L_\Delta$ as a vector bundle over $E$ is equal to the dimension of $E$ as a
  manifold, we have only to show that $L_\Delta$ is isotropic if and
  only if $\lb\cdot\,,\cdot\rb_\Delta$ is skew-symmetric.  But this is
  immediate by the preceding theorem.
\end{proof}

\bigskip
 
Next we describe the anchor of the Courant algebroid $TE\oplus T^*E\to
E$ in terms of linear splittings and the corresponding Dorfman
connections. We begin with a proposition, the proof of which is left to the
reader.
\begin{proposition}\label{anchor_conn}
Let $\Delta\colon \Gamma(TM\oplus E^*)\times\Gamma(E\oplus
  T^*M)\to\Gamma(E\oplus T^*M)$ be a Dorfman connection. Then the map
\[\nabla\colon\Gamma(TM\oplus E^*)\times\Gamma(E)\to\Gamma(E), \qquad
\nabla_\nu e=\pr_E(\Delta_\nu(e,0))
\]
is a linear connection.
\end{proposition}
This linear connection encodes in the following manner the anchor $\pr_{TE}\colon TE\oplus
T^*E\to TE$.
\begin{theorem}\label{anchor_thm}
Let $\Delta\colon \Gamma(TM\oplus E^*)\times\Gamma(E\oplus
  T^*M)\to\Gamma(E\oplus T^*M)$ be a Dorfman connection and choose
$\nu\in\Gamma(TM\oplus E^*)$ and $\tau\in\Gamma(E\oplus T^*M)$. Then 
$\pr_{TE}\left(\sigma_{TM\oplus
    E^*}^\Delta(\nu)\right)=\widehat{\nabla_\nu}$
and $\pr_{TE}(\tau^\uparrow)=(\pr_E\tau)^\uparrow$.
\end{theorem}

\begin{proof}
  The second claim is immediate by the definition of $\tau^\uparrow$ in
  \eqref{def_of_vert_pair}.  For the first equality, note that
  by definition of $\nabla$ and $\sigma_{TM\oplus
      E^*}^\Delta(\nu)$, 
  \[\pr_{TE}\left(\sigma_{TM\oplus
      E^*}^\Delta(\nu)\right)(e(m))=T_me(\pr_{TM}\nu)(m)+_E\left.\frac{d}{dt}\right\an{t=0}e(m)-t\nabla_\nu
  e(m)\]
  for all $e\in\Gamma(E)$ and $m\in M$. 
By \eqref{explicit_hat_D}, this proves the claim.
\end{proof}

\bigskip

Finally, we show how the Dorfman connection encodes the
Courant-Dorfman bracket on linear and core sections.  The next theorem
shows how the integrability of $L_\Delta$ is related to the curvature
$R_\Delta$ of the Dorfman connection.
\begin{theorem}\label{super}Let $\Delta\colon \Gamma(TM\oplus E^*)\times\Gamma(E\oplus
  T^*M)\to\Gamma(E\oplus T^*M)$ be a Dorfman connection and choose $\nu,
  \nu_1,\nu_2\in\Gamma(TM\oplus E^*)$ and $\tau, \tau_1,
  \tau_2\in\Gamma(E\oplus T^*M)$.  Then
\begin{enumerate} 
\item  $\left\lb\tau_1^\uparrow, \tau_2^\uparrow\right\rb=0$,
 \item  $\left\lb\sigma^\Delta(\nu), \tau^\uparrow\right\rb=\left(\Delta_{\nu}\tau\right)^\uparrow$, 
\item $
\left\lb\sigma^\Delta(\nu_1),\sigma^\Delta(\nu_2)\right\rb
=\sigma^\Delta(\lb \nu_1, \nu_2\rb_\Delta)-\widetilde{R_\Delta(\nu_1, \nu_2)(\cdot,0)}$.
\end{enumerate}
\end{theorem}
The proof of this theorem is relatively long and technical, it can be found in
Appendix~\ref{proof_of_super}. 
\begin{remark}\label{twisted_shit}
\begin{enumerate}
\item If the Courant-Dorfman bracket is twisted by a linear closed
  $3$-form $H$ over a map $\bar H\colon TM\wedge TM\to E^*$
  \cite{BuCa12}, then the bracket $\lb \tilde \nu_1,\tilde \nu_2\rb$
  is linear over $\lb \nu_1, \nu_2\rb_{\bar H,\Delta}=\lb \nu_1,
  \nu_2\rb_\Delta+(0,\bar H(X_1,X_2))$.  Note that the Dorfman
  connection dual to this bracket is $\Delta^{\bar
    H}_v\sigma=\Delta_v\sigma+(0,\langle \bar H(X,\cdot), e\rangle)$.
  A more careful study of general exact Courant algebroids
  \cite{Roytenberg99} over vector bundles and of the corresponding
  twisting of the Dorfman connections and dull algebroids
  corresponding to splittings of $TE\oplus T^*E$ will be done
  somewhere else.
\item The \textbf{Courant bracket}, i.e.~the skew-symmetric counterpart of
  the Courant-Dorfman bracket,
is given by
\begin{enumerate}
\item  $\left\lb \tau_1^\uparrow, \tau_2^\uparrow\right\rb_C=0$,
 \item  $\left\lb \sigma^\Delta(\nu),
     \tau^\uparrow\right\rb_C=\left\lb \sigma^\Delta(\nu),
     \tau^\uparrow\right\rb-(0, \frac{1}{2}q_E^*\dr\langle
   \nu,\tau\rangle)=\left(\Delta_{\nu}\tau-(0, \frac{1}{2}\dr\langle
   \nu,\tau\rangle)\right)^\uparrow$, 
\item $
\left\lb \widetilde{\nu_1},\widetilde{\nu_2}\right\rb_C
=\widetilde{\lb \nu_1, \nu_2\rb_\Delta}-R_\Delta(\nu_1,
\nu_2)(\cdot,0)^\uparrow-(0, \frac{1}{2}\dr\ell_{\Skew_\Delta(\nu_1,\nu_2)})$.
\end{enumerate}
We chose to work with the Courant Dorfman bracket because it is
described naturally by Dorfman connections, as in
Proposition~\ref{relation_Delta_L}. This is why we chose to call the
Dorfman connections after I. Dorfman. Since Dorfman connections are
equivalent to linear splittings of the standard Courant algebroid over
a vector bundle, this suggests that the Courant Dorfman bracket is
more natural than the Courant bracket.
\end{enumerate}
\end{remark}

The following corollary of Corollary~\ref{ss_dc} and
Theorem~\ref{super} is immediate.

\begin{corollary}
Let $E\to M$ be a vector bundle
and consider a linear splitting $TE\oplus T^*E=(T^{q_E}E\oplus (T^{q_E}E)^\circ)\oplus L$.
Then the horizontal space $L$ is a Dirac structure if and only if the 
corresponding dull algebroid $(TM\oplus E^*,\pr_{TM},\lb\cdot\,,\cdot\rb_L)$ 
is a Lie algebroid. 
\end{corollary}
In the next section we study more general (non-horizontal) Dirac structures on $E$.

\subsection{VB-Dirac structures and Dorfman connections}
We consider linear subbundles
\begin{equation*}
\begin{xy}
\xymatrix{
D\ar[d]\ar[r]&U\ar[d]\\
E\ar[r]&M
}\end{xy}
\qquad \text{ of } \qquad
\begin{xy}
\xymatrix{
TE\oplus T^*E\ar[d]\ar[r]&TM\oplus E^*\ar[d]\\
E\ar[r]&M
}\end{xy}
\end{equation*}
The intersection of such a sub- double vector bundle $D$ with the
vertical space $T^{q_E}E\oplus (T^{q_E}E)^\circ$ always has constant
rank on $E$ and there is a subbundle $K\subseteq E\oplus T^*M$ such
that $D\cap (T^{q_E}E\oplus (T^{q_E}E)^\circ)$ is spanned over $E$ by
the sections $k^\uparrow$ for all $k\in \Gamma(K)$.  In other words
$K$ is the core of $D$.  The following proposition follows from this
observation.
\begin{proposition}
  Let $E$ be a vector bundle endowed with a linear subbundle
  $D\subseteq TE\oplus T^*E$ over $U\subseteq TM\oplus E^*$ and with
  core $K\subseteq E\oplus T^*M$.  Then there exists a Dorfman
  connection $\Delta$ such that $D$ is spanned by the sections
  $k^\uparrow$ for all $k\in\Gamma(K)$ and $\sigma^\Delta(u)$ for all
  $u\in\Gamma(U)$.
\end{proposition}
The Dorfman connection $\Delta$ is then said to be \textbf{adapted} to
$D$. Conversely, given a Dorfman connection and two subbundles
$U\subseteq TM\oplus E^*$ and $K\subseteq E\oplus T^*M$, we call
$D_{U,K,\Delta}$ the linear subbundle of $TE\oplus T^*E\to E$ that is
spanned by $k^\uparrow$, for all $k\in\Gamma(K)$ and
$\sigma^\Delta(u)$ for all $u\in\Gamma(U)$.

\begin{proof}  
To see that such a
splitting exist, we work with decompositions.  Since $D$ and
$TE\oplus T^*E$ are both double vector bundles, there exist two
decompositions $\mathbb I_D\colon E\times_MU\times_MK \to D$ and
$\mathbb I\colon E\times_M (TM\oplus E^*)\times_M (E\oplus T^*M)\to
TE\oplus T^*E$.
Let $\iota\colon D\to TE\oplus T^*E$ be the double vector bundle
inclusion, over $\iota_U\colon U\to TM\oplus E^*$ and the identity on
$E$, and with core $\iota_K\colon K\to E\oplus T^*M$.  Then there
exists $\phi\in\Gamma(E^*\otimes U^*\otimes E\oplus T^*M)$ such that
the map
$\mathbb I\inv\circ\iota\circ\mathbb I_D \colon E\times_MU\times_MK\to
E\times_M (TM\oplus E^*)\times_M (E\oplus T^*M)$
sends $(e_m,u_m,k_m)$ to
$(e_m,\iota_U(u_m),\iota_K(k_m)+\phi(e_m,u_m))$. Using local basis
sections of $TM\oplus E^*$ adapted to $U$ and a partition of unity on
$M$, extend $\phi$ to
$\hat\phi\in \Gamma(E^*\otimes (TM\oplus E^*)\otimes (E\oplus
T^*M))$.
Then define a new decomposition
$\tilde{\mathbb I}\inv\colon TE\oplus T^*E\to E\times_M (TM\oplus
E^*)\times_M (E\oplus T^*M)$
by
$\tilde{\mathbb I}\inv(\xi)= \mathbb
I\inv(e)+_E(e_m,0_m,-\hat\phi(e_m,\nu_m))=\mathbb I\inv(e)+_{TM\oplus
  E^*}(0_m,\nu_m,-\hat\phi(e_m,\nu_m))$
for $\xi\in T{e_m}E\times T_{e_m}^*E$ with $\Phi_E(\xi)=\nu_m$. Then
$(\tilde{\mathbb I}\circ\iota\circ\mathbb
I_D)(e_m,u_m,k_m)=(e_m,\iota_U(u_m),\iota_K(k_m))$
for all $(e_m,u_m,k_m)\in E\times_M U\times_MK$.  The corresponding
linear splitting
$\tilde\Sigma\colon E\times _M (TM\oplus E^*)\to TE\oplus T^*E$,
$\tilde \Sigma(e_m,\nu_m)=\tilde{\mathbb I}(e_m,\nu_m,0_m)$ sends
$(e_m,\iota_U(u_m))$ to $\iota(\mathbb I_D(e_m,u_m,0_m))\in \iota(D)$.
\end{proof}
Next we ask how many linear splittings are adapted to $D$, and how two
linear splittings that are adapted to $D$ are related.
\begin{definition}
Two Dorfman connections $\Delta,\Delta'$ are said to be $(U,K)$-\textbf{equivalent}
if $(\Delta-\Delta')(\Gamma(U)\times\Gamma(E\oplus 0))\subseteq\Gamma(K)$.
\end{definition}

The following proposition shows that this defines an
equivalence relation on the set of Dorfman connections.  We
write $[\Delta]_{U,K}$, or simply $[\Delta]$, for the
$(U,K)$-class of the Dorfman connection $\Delta$. 
By the next proposition, triples $(U,K,[\Delta])$ are in one-one
correspondence with linear subbundles of $TE\oplus T^*E\to E$.
\begin{proposition}
Choose two   Dorfman connections
$\Delta, \Delta'$ and assume that $\Delta$ is adapted to $D$.
Then $\Delta'$ is adapted to $D$ if and only if 
$\Delta$ and $\Delta'$ are $(U,K)$-equivalent.
\end{proposition}

\begin{proof}
  Assume that $\Delta$ is adapted to $D$. Then $D$ is spanned by the
  sections $\sigma^\Delta(u)$ and $k^\uparrow$ for all $k\in\Gamma(K)$
  and $u\in\Gamma(U)$.  If $\Delta$ and $\Delta'$ are
  $(U,K)$-equivalent, we have
  $\sigma_\Delta(u)-\sigma_{\Delta'}(u)=k^\uparrow$ for some
  $k\in\Gamma(K)$.  By \eqref{linear_splitting}, this implies
  immediately that $\Delta'$ is adapted to $D$.  The converse
  implication can be proved in a similar manner.
\end{proof}

The following theorem follows from the results in the
preceding subsection.
\begin{theorem}\label{dirac_triples_thm}
  Let $D$ be a linear subbundle of $TE\oplus T^*E\to E$ over
  $U\subseteq TM\oplus E^*$ and with core $K\subseteq E\oplus T^*M$,
  and choose a Dorfman connection $\Delta$ that is adapted to $D$.
  Then
\begin{enumerate}
\item $D$ is isotropic if and only if $\Skew_\Delta\an{U\otimes U}=0$ and $K\subseteq U^\circ$.
\item $D$ is maximally isotropic if and only if  $\Skew_\Delta\an{U\otimes U}=0$ 
 and $K=U^\circ$.
\item $\Gamma(D)$ is closed under the Courant-Dorfman bracket
if and only if 
\begin{enumerate}
\item $\Delta_uk\in\Gamma(K)$ for all $u\in \Gamma(U)$,
$k\in\Gamma(K)$, 
\item $\lb\Gamma(U), \Gamma(U)\rb_\Delta\subseteq \Gamma(U)$,
\item 
$R_\Delta\Bigl(U\otimes U\otimes(E\oplus T^*M)\Bigr)\subseteq K$.
\end{enumerate}
\end{enumerate}
\end{theorem}

\begin{proof}
This is an immediate corollary of the results in the preceding subsection, 
using $R_\Delta\Bigl((TM\oplus E^*)\otimes (TM\oplus E^*)\otimes(0\oplus T^*M)\Bigr)=0$.
To see this use (2) of  Proposition~\ref{curvature_tensor}, bearing in mind that the anchor is $\pr_{TM}$.
\end{proof}

\begin{corollary}\label{Dirac_triples}
Let $D$ be a linear subbundle of $TE\oplus T^*E\to E$ over
$U\subseteq TM\oplus E^*$ and with core $K\subseteq E\oplus T^*M$, and choose
a Dorfman connection $\Delta$ that is adapted to $D$.
Then 
\begin{enumerate}
\item $D$ is an isotropic subalgebroid of $TE\oplus T^*E\to E$ if and only if 
\begin{enumerate}
\item $U\subseteq K^\circ$,
\item $\Delta_uk\in\Gamma(K)$ for all $u\in \Gamma(U)$,
$k\in\Gamma(K)$, 
\item $(U, \pr_{TM}\an{U}, \lb\cdot\,, \cdot\rb_\Delta\an{\Gamma(U)\times\Gamma(U)})$ is a skew-symmetric dull algebroid.
\item the induced  Dorfman connection 
\[\bar\Delta\colon \Gamma(U)\times\Gamma((E\oplus T^*M)/K)\to \Gamma((E\oplus T^*M)/K)
\]
is flat.
\end{enumerate}
 \item $D$ is a Dirac structure if and only if $U=K^\circ$ and $(U,\pr_{TM}\an{U},
  \lb\cdot\,,\cdot\rb_\Delta\an{\Gamma(U)\times\Gamma(U)})$ is a Lie algebroid.
\end{enumerate}
\end{corollary}

Note that in the second situation, the induced Dorfman connection
$\bar\Delta$
 is just  the Lie derivative
\[\ldr{}=\bar\Delta\colon \Gamma(U)\times\Gamma(U^*)\to \Gamma(U^*),
\]
which flatness is equivalent to the restriction of
$\lb\cdot\,,\cdot\rb_\Delta$ to $\Gamma(U)$ satisfying the Jacobi-identity.
The Dorfman connection $\bar\Delta$ depends only on the class
$[\Delta]$ of the connection $\Delta$. Conversely, 
a Dorfman connection $\bar\Delta\colon \Gamma(U)\times\Gamma((E\oplus T^*M)/K)\to
\Gamma((E\oplus T^*M)/K)$, 
can be extended to a Dorfman
connection $\Delta\colon \Gamma(TM\oplus E^*)\times\Gamma(E\oplus
T^*M)\to\Gamma(E\oplus T^*M)$ (by extending in a dull manner the corresponding Lie
algebroid bracket on $U$).  Two such extensions of $\bar\Delta$ are
automatically $(U,K)$-equivalent.
\begin{proof}[Proof of Corollary~\ref{Dirac_triples}]
  The proof is immediate. For (2), note only that $K=U^\circ$ and
  $\Delta_uk\in\Gamma(K)$ for all $u\in\Gamma(U)$, $k\in\Gamma(K)$
  imply together that the dull bracket restricts to a bracket on
  $\Gamma(U)$, and vice-versa.
\end{proof}

\begin{remark}
\begin{enumerate}
\item Using the following Proposition~\ref{brackets_equal_on_U}, one
  can see that if the conditions in (2) of
  Corollary~\ref{Dirac_triples} are satisfied for $\Delta$, then they
  are also satisfied for any $\Delta'$ that is $(U,K)$-equivalent to
  $\Delta$.
\item We say that  $(U,K,[\Delta])$ is a Dirac triple 
if the corresponding linear subbundle $D_{(U,K,[\Delta])}$ is a Dirac structure on $E$.
By the considerations above, we find that linear Dirac structures in
$TE\oplus T^*E\to E$ are in one-one correspondence with Dirac triples. 
\end{enumerate}
\end{remark}

\begin{proposition}\label{brackets_equal_on_U}
Let $E\to M$ be a vector bundle and choose a triple $(U,K,[\Delta]_{U,K})$
such that $U=K^\circ$.
Then for any two representatives $\Delta,\Delta'\in [\Delta]_{U,K}$, we have 
\[\lb u_1, u_2\rb_\Delta=\lb u_1, u_2\rb_{\Delta'}
\]
for all $u_1,u_2\in\Gamma(U)$.
\end{proposition}

\begin{proof}
Since $\pr_{TM}\lb u_1, u_2\rb_\Delta=[\pr_{TM}u_1, \pr_{TM}u_2]=\pr_{TM}\lb u_1, u_2\rb_{\Delta'}$, 
we need only to check that 
\[\langle \lb u_1, u_2\rb_\Delta, (e,0)\rangle=\langle\lb u_1, u_2\rb_{\Delta'}, (e,0)\rangle
\]
for all $e\in\Gamma(E)$.  But this is immediate by the hypothesis, the
duality of $\Delta$ and $\lb\cdot\,,\cdot\rb_\Delta$ and the
definition of $(U,K)$-equivalence.
\end{proof}

Since a linear Dirac structure $D$ in $TE\oplus T^*E$ over the base
$U\subseteq TM\oplus E^*$ is a VB-algebroid $(D\to E, U\to M)$, we get
the following corollary from Theorem~\ref{dirac_triples_thm},
Corollary~\ref{Dirac_triples} and
Proposition~\ref{brackets_equal_on_U}.
\begin{corollary}\label{2-rep_1}
  Let $(D;E,U;M)$ be a linear Dirac structure in $(TE\oplus
  T^*E;E,TM\oplus E^*;M)$. A linear splitting $\Sigma^\Delta$ of
  $TE\oplus T^*E$ that is adapted to $D$ defines a linear splitting
  $\Sigma$ of $D$. Then $(U,\pr_{TM}\an{U},
  \lb\cdot\,,\cdot\rb_\Delta\an{\Gamma(U)\times\Gamma(U)})$ is a Lie
  algebroid (that does not depend on the splitting), the restriction
  $\tilde\Delta$ of $\Delta$ to
  $\Gamma(U)\times\Gamma(U^\circ)\to\Gamma(U^\circ)$ is a linear
  connection, the linear connection $\nabla$ restricts to
  $\tilde\nabla\colon\Gamma(U)\times\Gamma(E)\to\Gamma(E)$ and the
  vector-valued $2$-form $R_\Delta$ restricts to $\tilde
  R_\Delta\in\Omega^2(U,\operatorname{Hom}(E,U^\circ))$.

  The triple $(\tilde\Delta,\tilde\nabla, \tilde R_\nabla)$ is a
  $2$-term representation up to homotopy of $U$ on
  $\pr_E\an{U^\circ}\colon U^\circ\to E$, that describes the
  VB-algebroid structure on $D$ in the linear splitting $\Sigma$.
\end{corollary}

We conclude with a series of examples.

\begin{example}\label{foliation_example}
In the situation of Example~\ref{example_easy}, choose 
two subbundles $F_M\subseteq TM$ and $C\subseteq E$.
Set $U:=F_M\oplus C^\circ$ and $K:=C\oplus F_M^\circ=U^\circ$.
The linear subbundle $D_{U,K,\Delta}$ 
corresponding to $U$, $K$ and the standard Dorfman connection
associated to $\nabla$ is then the direct sum of a subbundle $F_E\subseteq TE$, 
with  $C_E\subseteq T^*E$. 
Since $U=K^\circ$, we get immediately that $C_E=F_E^\circ$ and 
$D_{U,K,[\Delta]}$ is the trivial almost Dirac structure $F_E\oplus F_E^\circ$.
An application of Corollary~\ref{Dirac_triples} to this situation yields that
 $F_E\oplus F_E^\circ$ is a Dirac structure if and only if 
\begin{enumerate}
\item $F_M$ is involutive,
\item $\nabla_Xc\in\Gamma(C)$ for all $X\in\Gamma(F_M)$ and $c\in\Gamma(C)$
and 
\item the induced connection $\tilde \nabla\colon \Gamma(F_M)\times \Gamma(E/C)\to\Gamma(E/C)$ 
is flat.
\end{enumerate}
 Since $F_E\oplus F_E^\circ$ is Dirac if and only if $F_E\subseteq TE$ is involutive, 
we have recovered a result in \cite{JoOr14}, see also
\cite{DrJoOr15}.
\end{example}

\begin{example}\label{lie_algebroid_dual_2}
In the situation of  Example~\ref{lie_algebroid_dual}, consider
$U=\graphe(\rho\colon A\to TM)\subseteq TM\oplus A^*$ and $K=\graphe(-\rho^t\colon T^*M\to
A^*)=U^\circ$.
A straightforward computation shows that 
\[\Delta_{(\rho(a), a)}(-\rho^t(\theta),\theta)=\left(-\rho^t\left({\nabla^{\rm
    bas}_a}^*\theta\right), {\nabla^{\rm bas}_a}^*\theta\right)\in\Gamma(K)\] for all $
a\in\Gamma(A)$ and $\theta\in\Omega^1(M)$.
Furthermore, we have 
\[ \lb(\rho(a_1),a_1), (\rho(a_2),a_2)\rb_\Delta=(\rho([a_1,a_2]), [a_1,a_2])
\]
for all $a_1,a_2\in\Gamma(A)$, which shows that $(U,\pr_{TM},
\lb\cdot\,,\cdot\rb_\Delta)$
is a Lie algebroid if and only if $A$ is a Lie algebroid.
We have:
\begin{align*}
\bar\Delta_{(\rho(a),a)}\overline{(\alpha,0)}&=\overline{\left(\langle \alpha, \nabla_\cdot^{\rm
    bas}a\rangle+\nabla_{\rho(a)}^*\alpha-\rho^t\langle \nabla_\cdot a,
  \alpha\rangle, \langle \nabla_\cdot a, \alpha\rangle
\right)}\\
&=\overline{\left(\langle \alpha, \nabla_\cdot^{\rm
    bas}a\rangle+\nabla_{\rho(a)}^*\alpha, 0
\right)}=\overline{\left(\ldr{a}\alpha, 0
\right)}.
\end{align*}
Finally, the right-hand side of
\eqref{curvature_poisson}
vanishes for $(\rho(a), a), (\rho(b), b), (\rho(c),c)\in\Gamma(U)$ and
arbitrary $(\alpha,\theta)\in\Gamma(A^*\oplus T^*M)$  if and only if $A$ is a Lie algebroid.

Hence, we find that the linear subbundle $D$ of $TA^*\oplus T^*A^*\to A^*$
associated to $U, K$ and $\Delta$ is an almost Dirac structure on
$A^*$, and that is is a Dirac structure if and only if $A$ is a Lie algebroid.
The vector bundle $D\to A^*$ is the graph of the vector bundle morphism 
\[\pi_{A}^\sharp\colon T^*A^*\to TA^*
\]
associated to the linear almost Poisson structure defined on $A^*$ by the skew-symmetric dull
algebroid structure on $A$.
More precisely, $D$ is spanned by the sections $k^\uparrow$ for $k\in
\Gamma(K)$ and $\sigma_\Delta(u)$ for $u\in\Gamma(U)$, 
or, equivalently, by the sections
\[ (-\rho^t\theta^\uparrow, q_{A^*}^*\theta)  
\qquad \text{ and } \qquad  \left(\widehat{\ldr{a}}, \dr\ell_a\right) 
\]
for $\theta\in\Omega^1(M)$
and for $a\in \Gamma(A)$. 
By Appendix~\ref{appendix_linear_Poisson}, these are exactly the
sections $(\pi_A^\sharp(q_{A^*}^*\theta), q_{A^*}^*\theta)$ and 
$(\pi_A^\sharp(\dr \ell_a), \dr \ell_a)$. 
\end{example}

\begin{example}\label{IM_2_form_2}
Consider, in the situation of Example~\ref{IM_2_form}, 
$U:=\graphe(-\sigma^t\colon TM\to E^*)$ and $K:=\graphe(\sigma\colon E\to T^*M)$.
Then $U=K^\circ$ by definition and 
since 
\[\Delta_{(X,-\sigma^tX)}(e,\sigma(e))=(\nabla_Xe,\sigma(\nabla_Xe)),
\] we find that $\Delta_uk\in\Gamma(K)$ for all
$u\in\Gamma(U)$ and $k\in\Gamma(K)$.
Furthermore, we have \[\lb (X,-\sigma^tX),
(Y,-\sigma^tY)\rb_\Delta=([X,Y], -\sigma^t[X,Y])\] for all
$X,Y\in\mx(M)$ and $U$ is a Lie algebroid (isomorphic to $TM$ with the
Lie bracket of vector fields). Alternatively, the  Jacobiator in \eqref{curvature_sigma} is easily seen to
vanish on sections of $U$.
This shows that the double vector subbundle $D\subseteq TE\oplus T^*E$
defined 
by $U,K$ and $\Delta$ is a Dirac structure.

By the considerations in Appendix~\ref{pullback_canonical_symplectic},
$D$ is the graph of the vector bundle morphism $TE\to T^*E$ 
defined by the closed $2$-form $\sigma^*\omega_{\rm can}$.
\end{example}

\begin{example}\label{ex_Dirac_manifolds}
In this example, we consider the vector bundle $E=TM$, for a smooth
manifold $M$.
Consider a Dirac structure $D$ on $M$ and the Bott-Dorfman connection 
\[\Delta^D\colon \Gamma(D)\times\Gamma(TM\oplus T^*M/D)\to\Gamma(TM\oplus
T^*M/D)\]
defined by $D$ (see Proposition~\ref{Bott_Dorfman}). Choose 
an extension $\Delta\colon \Gamma(TM\oplus T^*M)\times\Gamma(TM\oplus
T^*M)\to \Gamma(TM\oplus T^*M)$ of $\Delta^D$, i.e.~a dull extension 
of the restriction to $\Gamma(D)$ of the Courant-Dorfman bracket.

It is easy to check that the triple $(D,D,[\Delta])=(D,D,\Delta^D)$
is a Dirac triple. Later we will see the meaning  of the Dirac
structure on $TM$ associated to it.
\end{example}

\begin{example}
We now combine Examples~\ref{foliation_example} and
\ref{IM_2_form_2} to recover an example in \cite{JoRa12a}.

We consider the vector bundle $T^*M\to M$ endowed 
with a $TM$-connection $\nabla$
and the Dorfman connection
\[\Delta\colon \Gamma(TM\oplus TM)\times \Gamma(T^*M\oplus T^*M)\to
\Gamma(T^*M\oplus T^*M),\]
\[\Delta_{(X,Y)}(\theta,\omega)=(\nabla_X\theta,
\ldr{X}(\omega-\theta)+\langle \nabla^*_\cdot(X+Y),\omega\rangle+\nabla_X\theta).
\]
Consider a subbundle $F\subseteq TM$ and $U:=\{(v,-v)\mid x\in
F\}\subseteq TM\oplus TM$. The annihilator $K=U^\circ$
is then given by $K=\{(\theta,\omega)\in T^*M\oplus T^*M\mid
\theta-\omega\in F^\circ\}$.

Note that by Example~\ref{IM_2_form}, the dull bracket 
on $TM\oplus TM$ is skew-symmetric. It is easy to see that 
its restriction to $U$  is just the Lie bracket of vector fields
\[ \lb (X,-X), (Y,-Y)\rb_\Delta=([X,Y],-[X,Y])
\]
for all $X,Y\in\Gamma(F)$.  Hence, we know already that the linear
subbundle $D_{(U,K,[\Delta])}$ is an almost Dirac structure on $T^*M$.
An easy computation using Appendix~\ref{pullback_canonical_symplectic}
yields that
\[D_{(U,K,[\Delta])}
(\theta)
=\{(v_\theta, \omega_{\rm can}^\flat(v_\theta)+\eta_\theta)\mid
v_\theta\in\mathcal F(\theta), \eta_\theta\in \mathcal F^\circ(\theta)
\}
\]
for all $\theta\in T^*M$,
where $\mathcal F=(Tc_M)\inv(F)$.
Assume that  $M$ is the configuration space of a nonholonomic mechanical
system and $F$ the constraints distribution. If $L$ is the Lagrangian
of the system, then the pullback to the contraints submanifold $\mathbb FL(F)\subseteq T^*M$
of the Dirac structure $D_{(U,K,[\Delta])}$
is one of the frameworks proposed in \cite{JoRa12a} for the study of
the nonholonomic system.
\end{example}

\section{Application: the prolongation 
\texorpdfstring{$TA\oplus T^*A\to TM\oplus A^*$}{TA+T*A->TM+A*} 
of a Lie algebroid $A$}
\label{new_2_reps}
We consider a Lie algebroid $(A\to M, \rho, [\cdot\,,\cdot])$ and a Dorfman connection
\[\Delta\colon \Gamma(TM\oplus A^*)\times\Gamma(A\oplus T^*M)\to
\Gamma(A\oplus T^*M)
\]
with corresponding dull bracket $\lb\cdot\,,\cdot\rb_{\Delta}$ and
anchor $\pr_{TM}$ on $TM\oplus A^*$.  By Theorem
\ref{relation_Delta_L}, this Dorfman connection corresponds to a
linear splitting of $TA\oplus T^*A$. Our goal is to compute the
representation up to homotopy defined by this linear splitting of the
VB-algebroid $(TA\oplus T^*A\to TM\oplus A^*, A\to M)$. Note that
until now, only the representations up to homotopy defined by standard
Dorfman connections were known (Example~\ref{up_to_now_ex}). The
results in this section are used in \cite{Jotz14} to describe
infinitesimally Dirac groupoids.

\medskip

We define a map $\Omega\colon \Gamma(TM\oplus A^*)\times\Gamma(A)\to\Gamma(A\oplus T^*M)$ by
\[\Omega_{(X,\alpha)}a=\Delta_{(X,\alpha)}(a,0)-(0,\dr\langle\alpha, a\rangle).
\]
 $\Omega$ satisfies
$\Omega_{ f (X,\alpha)}a= f \Omega_{(X,\alpha)}a$ and $\Omega_{(X,\alpha)}( f a)= f \Omega_{(X,\alpha)}a+X( f )(a,0)-\langle \alpha,
  a\rangle (0,\dr  f )$
for all $ f\in C^\infty(M)$, $a\in\Gamma(A)$ and
$(X,\alpha)\in\Gamma(TM\oplus A^*)$.  
For each $a\in\Gamma(A)$, we have two derivations over
$\rho(a)\in\mx(M)$:
\[\ldr{a}\colon \Gamma(A\oplus T^*M)\to\Gamma(A\oplus T^*M),\qquad \ldr{a}(a',\theta)=([a,a'], \ldr{\rho(a)}\theta)\]
and
\[\ldr{a}\colon \Gamma(TM\oplus A^*)\to\Gamma(TM\oplus A^*), \qquad
\ldr{a}(X,\alpha)=([\rho(a),X], \ldr{a}\alpha).\]
Note that 
$\ldr{ f a}(a',\theta)= f \ldr{a}(a',\theta)
+(-\rho(a')( f)a,\langle \theta, \rho(a)\rangle \dr f)$.

\subsection{The basic connections associated to $\Delta$}
\begin{proposition}\label{prop_def_basic_connections}
The two maps 
\begin{align*}
\nabla^{\rm bas}\colon  \Gamma(A)\times\Gamma(TM\oplus A^*)&\to \Gamma(TM\oplus A^*), \quad \nabla^{\rm bas}_a(X,\alpha)
=(\rho,\rho^t)(\Omega_{(X,\alpha)}a)+\ldr{a}(X,\alpha)\\
\nabla^{\rm bas}\colon  \Gamma(A)\times\Gamma(A\oplus T^*M)&\to \Gamma(A\oplus T^*M),\quad 
\nabla^{\rm
  bas}_a(a',\theta)
=\Omega_{(\rho,\rho^t)(a',\theta)}a+\ldr{a}(a',\theta)
\end{align*}
are ordinary linear connections.
\end{proposition}

\begin{proof}
The proof is straightforward and left to the reader.
\end{proof}

The following proposition is easily checked, and shows that the connections are  dual to each other 
if and only if the dull bracket on $\Gamma(TM\oplus A^*)$ is skew-symmetric.
\begin{proposition}\label{prop_basic_connections}
We have 
\begin{equation}\label{not_dual}
\langle \nabla_a^{\rm bas}\nu, \tau\rangle
+\langle \nu, \nabla_a^{\rm bas}\tau\rangle
=\rho(a)\langle \nu, \tau\rangle-\langle \Skew_\Delta(\nu,
(\rho,\rho^t)\tau), a\rangle
\end{equation}
\begin{equation}\label{intertwined}
\nabla_a^{\rm bas}(\rho,\rho^t)\tau=(\rho,\rho^t)\nabla_a^{\rm
    bas}\tau
\end{equation}
for all $a\in \Gamma(A)$, $\nu\in\Gamma(TM\oplus A^*)$ and $\tau\in\Gamma(A\oplus T^*M)$.
\end{proposition}

\begin{definition}
The connections in Proposition~\ref{prop_def_basic_connections} is called
the \textbf{basic connections} associated to $\Delta$.
\end{definition}

\begin{proposition}\label{prop_tensoriality}
The  map
\[R_\Delta^{\rm bas}\colon \Gamma(A)\times\Gamma(A)\times\Gamma(TM\oplus A^*)\to\Gamma(A\oplus T^*M)
\]
given by 
\begin{align*}
R_\Delta^{\rm bas}(a,b)(X,\alpha)=-\Omega_{(X,\alpha)}[a,b] +\ldr{a}\left(\Omega_{(X,\alpha)}b\right)-\ldr{b}\left(\Omega_{(X,\alpha)}a\right)
                                                    + \Omega_{\nabla^{\rm bas}_b(X,\alpha)}a-\Omega_{\nabla^{\rm bas}_a(X,\alpha)}b.
\end{align*}
is  tensorial, i.e.~it defines $R_\Delta^{\rm bas}\in\Omega^2(A,\operatorname{Hom}(TM\oplus A^*, A\oplus T^*M))$.
\end{proposition}

\begin{proof}
This proof also is a straightforward computation.
\end{proof}

\begin{definition}
  We call the tensor $R_\Delta^{\rm bas}$ the \textbf{basic curvature}
  associated to $\Delta$.
\end{definition}

\begin{proposition}\label{basic_curvatures}
The basic curvature satisfies
$R_{\nabla^{\rm bas}}=R^{\rm bas}_\Delta\circ
  (\rho,\rho^t)$ and $R_{\nabla^{\rm bas}}=(\rho,\rho^t)\circ R^{\rm bas}_\Delta$.
\end{proposition}

\begin{proof}
For $\tau\in\Gamma(A\oplus T^*M)$ and $a,b\in\Gamma(A)$,
  we have 
\begin{equation*}
\begin{split}
R_\Delta(a,b)((\rho,\rho^t)\tau)
&=-\Omega_{(\rho,\rho^t)\tau}[a,b] +\ldr{a}\left(\Omega_{(\rho,\rho^t)\tau}b\right)-\ldr{b}\left(\Omega_{(\rho,\rho^t)\tau}a\right)\\
    &\qquad + \Omega_{\nabla^{\rm bas}_b(\rho,\rho^t)\tau}a-\Omega_{\nabla^{\rm bas}_a(\rho,\rho^t)\tau}b\\
&=-\Omega_{(\rho,\rho^t)\tau}[a,b] -\ldr{[a,b]}\tau+\ldr{a}\left(\Omega_{(\rho,\rho^t)\tau}b+\ldr{b}\tau\right)-\ldr{b}\left(\Omega_{(\rho,\rho^t)\tau}a+\ldr{a}\tau\right)\\
                                               &\qquad +                                             \Omega_{\nabla^{\rm
                                                   bas}_b(\rho,\rho^t)\tau}a-\Omega_{\nabla^{\rm
                                                   bas}_a(\rho,\rho^t)\tau}b\\
&=-\nabla_{[a,b]}^{\rm bas}\tau+\nabla_a^{\rm bas}\nabla_b^{\rm
  bas}\tau-\nabla_b^{\rm bas}\nabla_a^{\rm bas}\tau=R_{\nabla^{\rm
    bas}}(a,b)\tau.
\end{split}
\end{equation*}
Note that in the second equality, we use $\ldr{a}\ldr{b}-\ldr{b}\ldr{a}-\ldr{[a,b]}=0$.
The second equality is shown in a similar manner.
\end{proof}

\subsection{The Lie algebroid structure on  \texorpdfstring{$TA\oplus T^*A\to TM\oplus A^*$}{TA+T*A->TM+A*}}
Consider a Lie algebroid $A$ and a    Dorfman connection
$\Delta\colon \Gamma(TM\oplus A^*)\times\Gamma(A\oplus T^*M)\to
\Gamma(A\oplus T^*M)$.
Then, for any section $a\in\Gamma(A)$, the horizontal lift
$\sigma_A(a)\in\Gamma_{TM\oplus A^*}(TA\oplus T^*A)$ is given by
\begin{align*}
\sigma^\Delta_A(a)(v_m,\alpha_m)&=(T_mav_m,\dr_{a_m}
\ell_\alpha)-\Delta_{(X,\alpha)}(a,0)^\uparrow(a_m)
\end{align*}
for any choice of section $(X,\alpha)\in\Gamma(TM\oplus A^*)$ such that
$(X,\alpha)(m)=(v_m,\alpha_m)$.
That is, we have 
\[\sigma_A^\Delta(a)=(Ta, R(\dr \ell_a))-\widetilde{\Omega_\cdot a}=a^l-\widetilde{\Omega_\cdot a}
\]
for all $a\in \Gamma(A)$ (using the notation of Appendix
\ref{big_lie_algebroids}). For simplicity, we write $\sigma_A$
for $\sigma_A^\Delta$.

\begin{theorem}\label{rep_up_to_hom}
The Lie algebroid structure on $TA\oplus T^*A\to TM\oplus A^*$ with anchor $\Theta\colon  TA\oplus T^*A\to T(TM\oplus A^*)$
is described as follows:
\begin{enumerate}
\item $[\sigma_A(a_1)\sigma_A(a_2)]=\sigma_A([a_1,a_2])-\widetilde{R_\Delta^{\rm bas}(a_1,a_2)}$,
\item $[\sigma_A(a),\tau^\dagger]=(\nabla_a^{\rm bas}\tau)^\dagger$,
\item $[\tau_1^\dagger, \tau_2^\dagger]=0$,
\item $\Theta(\sigma_A(a))=\widehat{\nabla_a^{\rm bas}}\in\mx(TM\oplus A^*)$,
\item $\Theta(\tau^\dagger)=((\rho,\rho^t)\tau)^\uparrow \in\mx(TM\oplus A^*)$.
\end{enumerate}
\end{theorem}

\begin{corollary}
In other words, $(\rho,\rho^t)\colon A\oplus T^*M\to TM\oplus A^*$, the basic connections $\nabla^{\rm bas}$ and 
the basic curvature $R_\Delta^{\rm bas}$ define the representation up to homotopy describing the VB-Lie algebroid
structure on $TA\oplus T^*A\to TM\oplus A^*$  in the linear splitting 
given by $\Delta$.
\end{corollary}

\begin{proof}[Proof of Theorem~\ref{rep_up_to_hom}]
  The proof of this theorem consists in checking the formulas, using
  the description of the Lie algebroid structure on $TA\oplus T^*A\to
  TM\oplus A^*$ in Appendix~\ref{big_lie_algebroids}.  We begin with
  the Lie algebroid brackets. Choose $a_1,a_2\in\Gamma(A)$ and
  $\tau\in\Gamma(A\oplus T^*M)$.  Using Proposition~\ref{structure_of_TAT*A}, we find
\begin{align*}
&[\sigma_A(a_1), \sigma_A(a_2)]=\left[a_1^l-\widetilde{\Omega_\cdot a_1},
a_2^l-\widetilde{\Omega_\cdot a_2}\right]\\
&=[a_1,a_2]^l-\widetilde{\ldr{a_1}\Omega_\cdot
a_2}+\widetilde{\ldr{a_2}\Omega_\cdot a_1}+\widetilde{\Omega_\cdot
a_2\circ(\rho,\rho^t)\circ \Omega_\cdot a_1}- \widetilde{\Omega_\cdot
a_1\circ(\rho,\rho^t)\circ \Omega_\cdot a_2}\\
&=\sigma_A[a_1,a_2]-\widetilde{R_\Delta^{\rm bas}(a_1,a_2)}.
\end{align*}
We have used 
\begin{align*}
&-(\ldr{a_1}\Omega_\cdot
a_2)(v)+(\ldr{a_2}\Omega_\cdot a_1)(v)+(\Omega_\cdot
a_2\circ(\rho,\rho^t)\circ \Omega_\cdot a_1)(v)- (\Omega_\cdot
a_1\circ(\rho,\rho^t)\circ \Omega_\cdot a_2)(v)\\
=&-\ldr{a_1} \Omega_va_2+\Omega_{\ldr{a_1}v}a_2+\ldr{a_2}
\Omega_va_1-\Omega_{\ldr{a_2}v}a_1
+\Omega_{(\rho,\rho^t)\Omega_v a_1}a_2- \Omega_{(\rho,\rho^t)\Omega_v
  a_2}a_1\\
=&-\ldr{a_1} \Omega_va_2+\ldr{a_2}
\Omega_va_1
+\Omega_{\nabla_{a_1}^{\rm bas}v}a_2- \Omega_{\nabla_{a_2}^{\rm bas}v}a_1
=-R_\Delta^{\rm bas}(a_1,a_2)v-\Omega_v[a_1,a_2]
\end{align*}
for all $v\in \Gamma(TM\oplus A^*)$.
Next, we find $[\sigma_A(a),\tau^\dagger]=(\ldr{a}\tau)^\dagger+\Omega_{(\rho,\rho^t)\tau}a^\dagger=(\nabla_a^{\rm
    bas}\tau)^\dagger$.
For the anchor map, we compute $\Theta(\sigma_A(a))(\ell_\tau)=\ell_{\ldr{a}\tau-(\Omega_\cdot a)^t((\rho,\rho^t)\tau)}$, which yields 
the desired equality 
since \begin{align*}
\langle(\Omega_\cdot a)^t((\rho,\rho^t)\tau),v\rangle
&=\langle(\rho,\rho^t)\Omega_v a, \tau\rangle=\langle\nabla_a^{\rm bas}v-\ldr{a}v, \tau\rangle\\
&=\rho(a)\langle v, \tau\rangle-\langle v,  {\nabla_a^{\rm bas}}^*\tau\rangle-\langle\ldr{a}v,\tau\rangle
\end{align*}
and consequently
$\langle v, \ldr{a}\tau-(\Omega_\cdot a)^t((\rho,\rho^t)\tau)\rangle=
\langle v, {\nabla_a^{\rm bas}}^*\tau\rangle$.
The remaining equalities follow from Proposition~\ref{structure_of_TAT*A}.
\end{proof}

\begin{theorem}\label{subalgebroids}
Consider a Lie algebroid $A$ and a Dorfman connection
$\Delta\colon \Gamma(TM\oplus A^*)\times\Gamma(A\oplus T^*M)\to \Gamma(A\oplus T^*M)$.
Let $U\subseteq TM\oplus A^*$ and $K\subseteq A\oplus T^*M$ be subbundles. Then 
the linear subbundle $D_{(U,K,[\Delta])}$ is a subalgebroid of $TA\oplus T^*A\to TM\oplus A^*$
over $U$ if and only if:
\begin{enumerate}
\item $(\rho,\rho^t)(K)\subseteq U$,
\item $\nabla_a^{\rm bas}k\in\Gamma(K)$ for all $a\in\Gamma(A)$ and $k\in\Gamma(K)$,
\item  $\nabla_a^{\rm bas}u\in\Gamma(U)$ for all  $a\in\Gamma(A)$ and $u\in\Gamma(U)$,
\item $R_\Delta^{\rm bas}(a_1,a_2)u\in\Gamma(K)$ for all $u\in\Gamma(U)$, $a_1,a_2\in\Gamma(A)$.
\end{enumerate}
\end{theorem}

\begin{proof}
  Assume that $D_{(U,K,[\Delta])}\to U$ is a subalgebroid of
  $TA\oplus T^*A\to TM\oplus A^*$. Then we have
  $((\rho,\rho^t)k)^\uparrow\an{U}=\Theta(k^\dagger\an{U})\in\mx(U)$
  and
  $\widehat{\nabla_a^{\rm
      bas}}\an{U}=\Theta(\sigma_A(a)\an{U})\in\mx(U)$
  for all $a\in\Gamma(A)$ and $k\in\Gamma(K)$.  This is the case if
  and only if $((\rho,\rho^t)k)^\uparrow(\ell_{\tau})\an{U}=0$ and
  $\widehat{\nabla_a^{\rm bas}}(\ell_{\tau})\an{U}=0$ for all
  $\tau\in\Gamma(U^\circ)$.  Since
  $((\rho,\rho^t)k)^\uparrow(\ell_{\tau})=\pi^*\langle (\rho,\rho^t)k,
  \tau\rangle$
  and
  $\widehat{\nabla_a^{\rm bas}}(\ell_{\tau})=\ell_{{\nabla_a^{\rm
        bas}}^*\tau}$,
  we find that $(\rho,\rho^t)k$ must be a section of $U$ and
  ${\nabla_a^{\rm bas}}^*\tau\in\Gamma(U^\circ)$ for all
  $\tau\in\Gamma(U^\circ)$.  But the latter is equivalent to
  $\nabla_a^{\rm bas}u\in\Gamma(U)$ for all $u\in\Gamma(U)$.  We have
  in the same manner
  $(\nabla_a^{\rm bas}k)^\dagger\an{U}=[\sigma_A(a),
  k^\dagger]\an{U}\in\Gamma(D_{(U,K,[\Delta])})$
  and
  $\bigl(\sigma_A[a_1,a_2]-\widetilde{R_\Delta^{\rm
      bas}(a_1,a_2)}\bigr)\an{U}=[\sigma_A(a_1),
  \sigma_A(a_2)]\an{U}\in\Gamma(D_{(U,K,[\Delta])})$
  for all $a_1,a_2\in\Gamma(A)$ and $k\in\Gamma(K)$. But this is only
  the case if $\nabla_a^{\rm bas}k\in\Gamma(K)$ and, since
  $\sigma_A[a_1,a_2]\an{U}\in\Gamma(D_{(U,K,[\Delta])})$, if 
  $R_\Delta^{\rm
    bas}(a_1,a_2)^\dagger\an{U}\in\Gamma(D_{(U,K,[\Delta])})$.
  This holds only if $R_\Delta^{\rm bas}(a_1,a_2)u\in\Gamma(K)$
  for all $u\in\Gamma(U)$.

The converse implication is shown in a similar manner.
\end{proof}

\subsection{LA-Dirac structures in 
\texorpdfstring{$TA\oplus T^*A$}{TA+T*A}}
Our last result is a description of the triples $(U,K,[\Delta]_{U,K})$
associated to Dirac structures on $A$ that are at the same time Lie
subalgebroids of $TA\oplus T^*A\to TM\oplus A^*$.  We call such a
Dirac structure $D_A$ an \textbf{LA-Dirac structure} on $A$, and we
call the pair $(A,D_A)$ a \textbf{Dirac algebroid}.
 
\begin{theorem}\label{morphic_Dirac_triples}
  Consider a Lie algebroid $A$ and a Dorfman connection
$\Delta\colon \Gamma(TM\oplus A^*)\times\Gamma(A\oplus T^*M)\to \Gamma(A\oplus T^*M)$.
Let $U\subseteq TM\oplus A^*$ and $K\subseteq A\oplus T^*M$ be
subbundles. Then $D_{(U,K,[\Delta])}$ is a Dirac structure in
$TA\oplus T^*A\to A$ and a subalgebroid of $TA\oplus T^*A\to TM\oplus
A^*$ over $U$ if and only if:
\begin{enumerate}
\item $K=U^\circ$
\item $(\rho,\rho^t)(K)\subseteq U$,
\item $(U,\pr_{TM}, \lb\cdot\,,\cdot\rb_\Delta)$ is a Lie algebroid,
\item $\nabla_a^{\rm bas}k\in\Gamma(K)$ for all $a\in\Gamma(A)$ and $k\in\Gamma(K)$,
\item $R_\Delta^{\rm bas}(a_1,a_2)u\in\Gamma(K)$ for all $u\in\Gamma(U)$, $a_1,a_2\in\Gamma(A)$.
\end{enumerate}
\end{theorem}

\begin{proof}
  This theorem follows from (2) in Corollary~\ref{Dirac_triples},
  together with Theorem~\ref{subalgebroids}.  Note that if
  $U=K^\circ$, $(\rho,\rho^t)K\subseteq U$ and $(U,\pr_{TM},
  \lb\cdot\,,\cdot\rb_\Delta)$ is a Lie algebroid, then $\nabla^{\rm
    bas}_a$ preserves $\Gamma(U)$ if and only if $\nabla^{\rm bas}_a$
  preserves $\Gamma(K)$. So (2) and (3) in Theorem~\ref{subalgebroids}
  become one single condition.
\end{proof}

\begin{corollary}
  We have found a one-one correspondence between triples
  $(U,K,[\Delta]_{U,K})$ satisfying (1)--(5) in
  Theorem~\ref{subalgebroids} and LA-Dirac structures on the Lie
  algebroid $A$.
\end{corollary}

\begin{remark}
  If $D$ is an LA-Dirac structure over $U\subseteq TM\oplus A^*$ in
  $TA\oplus T^*A$, then $(D\to U, A\to M)$ and $(D\to A,U\to M)$ are
  VB-algebroids. A linear splitting $\Sigma^\Delta$ that is adapted to
  $D$ defines a linear splitting $\Sigma$ of $D$. The $2$-term
  representation up to homotopy
  $(\tilde\Delta,\tilde\nabla,\tilde R_\Delta)$ of $U$ on
  $pr_E\colon U^\circ \to A$ describes $(D\to A,U\to M)$ in this
  splitting (see Corollary~\ref{2-rep_1}).
  Theorem~\ref{subalgebroids} shows that the representation up to
  homotopy $(\nabla^{\rm bas},\nabla^{\rm bas},R_\Delta^{\rm bas})$ of
  $A$ on $(\rho,\rho^t)\colon A\oplus T^*M\to TM\oplus A^*$ restricts
  to a representation up to homotopy
  $(\widetilde{\nabla^{\rm bas}},\widetilde{\nabla^{\rm
      bas}},\widetilde{R_\Delta^{\rm bas}})$
  of $A$ on $(\rho,\rho^t)\an{U^\circ}\colon U^\circ\to U$. One can
  check that these two $2$-term representations up to homotopy form a
  matched pair, which implies that $D$ is a double Lie algebroid
  \cite{GrJoMaMe14}. The computation is very similar to the one for
  the double Lie algebroid $TA$ in \cite[Section 3.3]{GrJoMaMe14}.
\end{remark}

Finally we discuss our previous examples. To avoid
  confusions, we write $\nabla^{A}$ for the $A$-basic connections
  induced on $A$ and $TM$ by the Lie algebroid structure on $A$ and
  the connection $\nabla$, and $R_{\nabla}^A$ for the basic curvature
  associated to it (\S\ref{basic_connections}). 

\begin{example}
  In the situation of Examples~\ref{example_easy} and
  \ref{foliation_example}, assume that the vector bundle $E$ is a Lie
  algebroid $A$. We show that the conditions in
  Theorem~\ref{morphic_Dirac_triples} define in this case an
  \emph{infinitesimal ideal system} \cite{JoOr14}, see also
  \cite{Hawkins08}. Condition (1) is trivially satisfied by
  construction and Condition (3) is the involutivity of $F_M$ and the
  quotient of $\nabla$ to a flat connection
  $\tilde\nabla\colon\Gamma(F_M)\times\Gamma(A/C)\to\Gamma(A/C)$
  (Example~\ref{foliation_example}). Condition (2) is
  $\rho(C)\subseteq F_M$, Condition (4) is $\nabla^A_ac\in\Gamma(C)$
  for all $c\in\Gamma(C)$ and $\nabla^A_aX\in\Gamma(F_M)$ for all
  $X\in\Gamma(F_M)$. To see this, note that $\nabla^{\rm
    bas}\colon\Gamma(A)\times\Gamma(A\oplus T^*M)\to\Gamma(A\oplus
  T^*M)$ is here just $\nabla_a^{\rm
    bas}(a',\theta)=(\nabla^A_aa',{\nabla^A}^*_a\theta)$.  Finally, an
  easy computation shows $R_\Delta^{\rm
    bas}(a_1,a_2)(X,\alpha)=(R_\nabla^A(a_1,a_2)X,-R_\nabla^A(a_1,a_2)^*\alpha)$,
  which implies the equivalence of Condition (5) with
  $R_\nabla^A(a_1,a_2)X\in\Gamma(C^\circ)$ for all $X\in\Gamma(F_M)$
  and all $a_1,a_2\in\Gamma(A)$. Therefore, following \cite[\S
  5.2]{DrJoOr15} we find that the conditions of
  Theorem~\ref{morphic_Dirac_triples} are satisfied if and only if
  $(F_M,C,\tilde\nabla)$ is an infinitesimal ideal system in $A$.
\end{example}

\begin{example}\label{lie_algebroid_dual_3}
  Consider again Examples~\ref{lie_algebroid_dual} and
  \ref{lie_algebroid_dual_2}.  Assume that $A^*$ has itself also a Lie
  algebroid structure with anchor $\rho_*$ and bracket
  $[\cdot\,,\cdot]_*$.  For simplicity, we switch the roles of $A$ and
  $A^*$ in Examples~\ref{lie_algebroid_dual} and
  \ref{lie_algebroid_dual_2}.  We show that $U$, $K$, $\Delta$ satisfy
  the conditions of Theorem~\ref{morphic_Dirac_triples} if and only if
  $(A,A^*)$ is a Lie bialgebroid.  Recall that we have already found
  that (1) and (3) are equivalent to $A^*$ being a Lie
  algebroid. Then, (2) in Theorem~\ref{morphic_Dirac_triples} is
  equivalent to
\begin{equation}\label{antisym}
\rho_*\circ\rho^t=-\rho\circ\rho_*^t.
\end{equation}
We assume in the following that this condition is satisfied.
We also have:
\[\Omega_{(\rho_*(\alpha),\alpha)}a=\left(\ldr{\alpha}a-\rho_*^t\langle \nabla^*_\cdot \alpha,
  a\rangle,\langle \nabla^*_\cdot \alpha, a\rangle
\right)-(0,\dr\langle\alpha,a \rangle)
=\left(\ip{\alpha}\dr_Aa+\rho_*^t\langle \alpha, \nabla_\cdot
  a\rangle,-\langle \alpha, \nabla_\cdot a\rangle \right),
\]
for all $\alpha\in\Gamma(A^*)$ and $a\in\Gamma(A)$
and so 
\begin{align*}
  \Omega_{(\rho,\rho^t)(-\rho_*^t\theta,\theta)}a&=\Omega_{(\rho_*(\rho^t\theta),\rho^t\theta)}a=\left(\ip{\rho^t\theta}\dr_Aa+\rho_*^t\langle
    \rho^t\theta, \nabla_\cdot a\rangle,-\langle \rho^t\theta,
    \nabla_\cdot a\rangle \right).
\end{align*}
for all $\theta\in\Omega^1(M)$.
In particular, if $\theta=\dr f$ for some $ f\in C^\infty(M)$, we get:
\begin{align*}
  \nabla^{\rm bas}_a(-\rho_*^t\dr f, \dr f)
  &=\Omega_{(\rho,\rho^t)(-\rho_*^t\dr f,\dr f)}a+\ldr{a}(-\rho_*^t\dr f,\dr f)\\
  &=\left(\ip{\dr_{A^*} f}\dr_Aa+\rho_*^t\langle \dr_{A^*} f,
    \nabla_\cdot a\rangle -[a,\dr_A f], -\langle \dr_{A^*} f,
    \nabla_\cdot a\rangle +\dr(\rho(a)( f)) \right).
\end{align*}
Thus, using (1) in Theorem~\ref{morphic_Dirac_triples}, $\nabla^{\rm
  bas}_a(-\rho_*^t\dr f, \dr f)\in\Gamma(K)$ if and only if
\[\langle \left(\ip{\dr_{A^*} f}\dr_Aa+\rho_*^t\langle \dr_{A^*} f,
  \nabla_\cdot a\rangle  -[a,\dr_A f], -\langle  \dr_{A^*} f, \nabla_\cdot  a\rangle
+\dr(\rho(a)( f))
\right), (\rho_*\alpha,\alpha)\rangle=0
\]
for all $\alpha\in\Gamma(A^*)$.
But this pairing equals
\begin{align*}
  &\langle \left(\ip{\dr_{A^*} f}\dr_Aa+\rho_*^t\langle \dr_{A^*} f,
    \nabla_\cdot a\rangle -[a,\dr_A f], -\langle \dr_{A^*} f,
    \nabla_\cdot a\rangle +\dr(\rho(a)( f)) \right),
  (\rho_*\alpha,\alpha)\rangle,
\end{align*}
which is easily shown to be the same as
$\bigl([\rho_*(\alpha),\rho(a)]+\rho_*(\ldr{a}\alpha)-\rho(\ldr{\alpha}a)+\rho(\dr_A\langle
\alpha,a\rangle)\bigr)( f)$.
Since $ f$ was arbitrary, we have shown that the fourth condition is
satisfied if and only if
\begin{equation}\label{compatibility_of_derivations}
[\rho(a),\rho_*(\alpha)]-\rho_*(\ldr{a}\alpha)+\rho(\ldr{\alpha}a)=\rho(\dr_A\langle \alpha,a\rangle)
\end{equation}
for all $a\in\Gamma(A)$ and $\alpha\in\Gamma(A^*)$.  Thus, we have
found until here \eqref{antisym} and
\eqref{compatibility_of_derivations}, which are properties of Lie
bialgebroids (see \cite{Mackenzie05}).

Using these equations, we study Condition (5), on the basic curvature.
Since $\Omega_{(\rho_*(\alpha),\alpha)}a=(\ip{\alpha}\dr_Aa,
0)-(-\rho_*^t\langle \alpha, \nabla_\cdot a\rangle,\langle \alpha,
\nabla_\cdot a\rangle)$, we find
$\langle \Omega_{(\rho_*(\alpha),\alpha)}a, (\rho_*\alpha',\alpha')\rangle=(\dr_Aa)(\alpha,\alpha')$
for all $a\in\Gamma(A)$, $\alpha,\alpha'\in\Gamma(A^*)$.  The fourth
condition together with \eqref{not_dual} and the first and third
conditions imply that $\nabla_a^{\rm bas}u\in\Gamma(A)$ for all
$u\in\Gamma(U)$.  Hence,
\begin{align*}
  \nabla_a^{\rm
    bas}(\rho_*(\alpha),\alpha)&=(\rho,\rho^t)\left(\ip{\alpha}\dr_Aa+\rho_*^t\langle
    \alpha,
    \nabla_\cdot a\rangle,-\langle  \alpha, \nabla_\cdot  a\rangle\right)+\ldr{a}(\rho_*(\alpha),\alpha)\\
  &=\left(\rho_*(-\rho^t\langle \alpha, \nabla_\cdot
    a\rangle+\ldr{a}\alpha),-\rho^t\langle \alpha, \nabla_\cdot
    a\rangle+\ldr{a}\alpha\right)
\end{align*}
and 
\[\langle \Omega_{\nabla^{\rm bas}_a(\rho_*\alpha,\alpha)}a', (\rho_*\alpha',\alpha')\rangle=
(\dr_Aa')(-\rho^t\langle  \alpha, \nabla_\cdot  a\rangle+\ldr{a}\alpha,\alpha')
\]
for all $a,a'\in\Gamma(A)$, $\alpha,\alpha'\in\Gamma(A^*)$.
Then a computation yields
\begin{align*}
  &\langle R_\Delta^{\rm bas}(a,a')(\rho_*\alpha,\alpha), (\rho_*\alpha',\alpha')\rangle=(\dr_A[a,a']-[a,\dr_Aa']+[a',\dr_Aa])(\alpha,\alpha')\\
  &\qquad \qquad+\langle  \alpha,\nabla_{\rho_*(\ldr{a}\alpha')} a'\rangle-\langle  \alpha, \nabla_{[\rho(a),\rho_*\alpha']}  a'\rangle-\langle  \alpha,\nabla_{\rho_*(\ldr{a'}\alpha')} a\rangle+\langle  \alpha, \nabla_{[\rho(a'),\rho_*\alpha']}  a\rangle\\
  &\qquad \qquad+\langle \alpha, \nabla_{\rho(\dr_A\langle
    a,\alpha'\rangle)} a'\rangle -\langle \alpha,
  \nabla_{\rho(\ldr{\alpha'}a)} a'\rangle-\langle \alpha,
  \nabla_{\rho(\dr_A\langle a',\alpha'\rangle)} a\rangle +\langle
  \alpha, \nabla_{\rho(\ldr{\alpha'}a')} a\rangle.
\end{align*}
By \eqref{compatibility_of_derivations}, the second and the third lines vanish.  We find
hence that the last condition is satisfied if and only if $(A,A^*)$ is
a Lie bialgebroid. Hence, $(U,K,[\Delta])$ is an LA-Dirac triple if
and only if $(A,A^*)$ is a Lie bialgebroid, and so the graph of
$\pi_A$ is a subalgebroid and Dirac if and only if $(A,A^*)$ is a Lie
bialgebroid. This was already found in \cite{MaXu00}.
\end{example}

\begin{example}\label{IM_2_form_3}
  In the situation of Examples~\ref{IM_2_form} and \ref{IM_2_form_2},
  assume furthermore that $E=:A$ is a Lie algebroid.  Condition (2) in
  Theorem~\ref{morphic_Dirac_triples} reads here
  $(\rho,\rho^t)(a,\sigma(a))=(\rho(a),-\sigma^t\rho(a))$ for all
  $a\in\Gamma(A)$, that is, $\rho^t\circ\sigma=-\sigma^t\circ\rho$.
  This is equivalent to the first axiom defining an IM--$2$--form
  $\sigma\colon A\to T^*M$ \cite{BuCrWeZh04,BuCaOr09}, namely $\langle
  \sigma(a_1),\rho(a_2)\rangle=-\langle\rho(a_1), \sigma(a_2)\rangle$
  for all $a_1,a_2\in\Gamma(A)$.  Next we compute $\nabla^{\rm
    bas}_a(a',\sigma(a'))$.  We have
\begin{equation*}
\begin{split}
  \Omega_{(X,-\sigma^tX)}a&=(\nabla_Xa,-\ldr{X}\sigma(a)+\sigma(\nabla_Xa))+(0,\dr\langle\sigma(a),X\rangle)=(\nabla_Xa,-\ip{X}\dr\sigma(a)+\sigma(\nabla_Xa))
\end{split}
\end{equation*}
and as a consequence
\begin{equation*}
\begin{split}\nabla_a^{\rm
    bas}(a,\sigma(a'))&=\Omega_{(\rho,\rho^t)(a',\sigma(a'))}a+\ldr{a}(a',\sigma(a'))\\
  &=(\nabla_{\rho(a')}a+[a,a'],
  \ldr{\rho(a)}\sigma(a')-\ip{\rho(a')}\dr\sigma(a)+\sigma(\nabla_{\rho(a')}a)).
\end{split}
\end{equation*}
Hence we find that $\nabla_a^{\rm bas}(a',\sigma(a'))\in\Gamma(K)$ if
and only if $([a,a'],
\ldr{\rho(a)}\sigma(a')-\ip{\rho(a')}\dr\sigma(a))\in\Gamma(K)$,
i.e.~if and only if
$\sigma([a,a'])=\ldr{\rho(a)}\sigma(a')-\ip{\rho(a')}\dr\sigma(a)$.
Since this is the second axiom in the definition of an IM--$2$--form,
we find that the graph of $(\sigma^*\omega_{\rm can})^\flat\colon
TA\to T^*A$ is a subalgebroid of $TA\oplus T^*A\to TM\oplus A^*$ over
$U=\graphe(-\sigma^t)$ only if $\sigma\colon A\to T^*M$ is an
IM--$2$--form.  To recover the equivalence \cite{BuCaOr09}, we show
that in this example, Condition (5) follows from the four previous
conditions.  We have also for $a,a'\in\Gamma(A)$ and $X\in\mx(M)$:
\begin{align*}
  &\ldr{a'}\Omega_{(X,-\sigma^tX)}a=-(0,\ldr{\rho(a')}\ip{X}\dr\sigma(a))+([a',\nabla_Xa],\ldr{\rho(a')}\sigma(\nabla_Xa))\\
  &=-(0,\ip{[\rho(a'),X]}\dr\sigma(a)+\ip{X}\ldr{\rho(a')}\dr\sigma(a))+([a',\nabla_Xa],\sigma([a',\nabla_Xa])+\ip{\rho(\nabla_Xa)}\dr\sigma(a'))
\end{align*}
and 
\begin{align*}
\nabla_a^{\rm
  bas}(X,-\sigma^tX)=-(\rho,\rho^t)(0,\ip{X}\dr\sigma(a))+(\rho,\rho^t)(\nabla_Xa,\sigma(\nabla_Xa))
+\ldr{a}(X,-\sigma^tX)
\end{align*}
which equals $(\nabla_a^AX,-\sigma^t(\nabla_a^AX))$.
Then we easily get
\begin{align*}
  R_\Delta^{\rm bas}(a,a')(X,-\sigma^tX) =&\bigl( R_\nabla^A(a,a')(X),
  \sigma(R_\nabla^A(a,a')(X))\bigr)\in\Gamma(K).
\end{align*}
\end{example}

\begin{example}
  We are here in the situation of
  Example~\ref{ex_Dirac_manifolds}. Recall that $TM\to M$ with the Lie
  bracket of vector fields and the anchor $\Id_{TM}$ is the standard
  example of a Lie algebroid. We check here that the Dirac triple
  $(D,D,\Delta^D)$ satisfies the conditions of
  Theorem~\ref{morphic_Dirac_triples}. First, we obviously have
  $(\Id_{TM},\Id_{TM}^t)(D)\subseteq D$.  Then, note that for all
  $X,Y\in\mx(M)$ and $\theta\in\Omega^1(M)$, we have
\begin{align*}
\nabla^{\rm bas}_X(Y,\theta)&=\ldr{X}(Y,\theta)+\Omega_{(Y,\theta)}X=\lb
(X,0),(Y,\theta)\rb+\Delta_{(Y,\theta)}(X,0)-(0,\dr\langle\theta,
X\rangle)\\
&=\Delta_{(Y,\theta)}(X,0)-\lb(Y,\theta), (X,0)\rb.
\end{align*}
Thus, we can compute for $X\in\mx(M)$ and $d_1,d_2\in\Gamma(D)$:
\begin{align*}
  \langle \nabla_X^{\rm bas}d_1, d_2\rangle&=\langle
  \Delta_{d_1}(X,0)-\lb d_1, (X,0)\rb,
  d_2\rangle=\langle\Delta_{d_1}^D\overline{(X,0)},d_2\rangle-\langle\lb
  d_1, (X,0)\rb,
  d_2\rangle\\
  &=\langle\lb d_1, (X,0)\rb,d_2\rangle-\langle\lb d_1, (X,0)\rb,
  d_2\rangle=0.
\end{align*}
This shows that $\nabla_X^{\rm bas}d\in\Gamma(D)$ for all
$d\in\Gamma(D)$.  Finally we check Condition (5), involving the
basic curvature.  For this, note first that an easy computation using
$\langle \Delta_d(X,0), d'\rangle=\langle
\Delta^D\overline{(X,0)},d'\rangle=\langle\lb d, (X,0)\rb, d'\rangle$
yields $\langle\Omega_dX,d'\rangle=-\langle \ldr{X}d,d'\rangle$ for
all $X\in\mx(M)$ and $d,d'\in\Gamma(D)$. We get
\begin{equation*}
\begin{split}
 \langle R_\Delta^{\rm bas}(X_1,X_2)d,d'\rangle
&=\langle -\Omega_{d}[X_1,X_2]+\ldr{X_1}\Omega_dX_2-\ldr{X_2}\Omega_dX_1+
\Omega_{\nabla^{\rm bas}_{X_2}d}X_1-\Omega_{\nabla^{\rm bas}_{X_1}d}X_2, d'\rangle\\
&=\langle \ldr{[X_1,X_2]}d+\ldr{X_1}\Omega_dX_2-\ldr{X_2}\Omega_dX_1-
\ldr{X_1}\nabla^{\rm bas}_{X_2}d+\ldr{X_2}\nabla^{\rm bas}_{X_1}d, d'\rangle,
\end{split}
\end{equation*}
since we have found above that $\nabla^{\rm bas}_{X_2}d$, $\nabla^{\rm
  bas}_{X_1}d\in\Gamma(D)$.  But since
$\ldr{X_1}\Omega_dX_2-\ldr{X_1}\nabla^{\rm
  bas}_{X_2}d=-\ldr{X_1}\ldr{X_2}d$, we find $\langle R_\Delta^{\rm
  bas}(X_1,X_2)d,d'\rangle =\langle
\ldr{[X_1,X_2]}d-\ldr{X_1}\ldr{X_2}d+\ldr{X_2}\ldr{X_1}d,
d'\rangle=0$.

There is a canonical
isomorphism from the Courant algebroid over $TM$
\begin{equation*}
\begin{xy}
\xymatrix{
TTM\oplus T^*TM\ar[r]\ar[d]& TM\oplus T^*M\ar[d]\\
TM\ar[r]&M
}
\end{xy}\qquad \longrightarrow \qquad \begin{xy}
\xymatrix{
T(TM)\oplus T(T^*M)\ar[r]\ar[d]& TM\oplus T^*M\ar[d]\\
 TM\ar[r]&M
}
\end{xy}
\end{equation*}
to the double tangent 
of the vector bundle $TM\oplus T^*M$
\cite{Tulczyjew77}, see also \cite{Mackenzie05}. One can check in a
straightforward manner (using for instance \cite{Mackenzie05}) that
this isomorphism is nothing else than the anchor of the VB-Lie
algebroid $(T(TM)\oplus T^*(TM),TM\oplus T^*M; TM,M)$.

$D_{(D,D,\Delta^D)}$ is spanned as a vector bundle over $D$ by the
sections $\sigma^\Delta_{TM}(X)\an{D}$ for all $X\in\mx(M)$ and
$d^\dagger\an{D}$ for all $d\in\Gamma(D)$.  By Theorem
\ref{rep_up_to_hom}, the image of $\sigma^\Delta_{TM}(X)$ under the
anchor $\Theta$ is $\widehat{\nabla_X^{\rm bas}}$ and the image of
$d^\dagger$ is $d^\uparrow$.
Hence, since $\nabla^{\rm bas}$ restricts to a $TM$-connection on $D$,
we get that the linear subbundle $(D_{(D,D,\Delta^D)},D;TM,M)$ is sent
via this isomorphism to $(TD,D;TM,M)$, the tangent Dirac structure in
\cite{Courant90a}.
\end{example}

\appendix
\section{Linear almost Poisson structures and the canonical symplectic
  form on 
\texorpdfstring{$T^*E$}{T*E}}
\subsection{Linear almost Poisson structures}\label{appendix_linear_Poisson}
Consider here a skew-symmetric dull algebroid $(A, \rho,
[\cdot\,,\cdot])$.  This is equivalent to a linear almost Poisson
bracket on the vector bundle $A^*\to M$, i.e.~a skew-symmetric bracket
$\{\cdot\,,\cdot\}\colon C^\infty(A^*)\times C^\infty(A^*)\to
C^\infty(A^*)$ such that
\begin{enumerate}
\item $\{\cdot\,,\cdot\}$ satisfies the Leibniz identity, 
\item $\{\ell_a,\ell_b\}=\ell_{[a, b]}$ is again linear for two
  sections $a, b\in\Gamma(A)$ and
\item $\{\ell_a, q_{A^*}^* f\}=q_{A^*}^*(\rho(a)( f))$ is again a pullback for all $a\in
  \Gamma(A)$ and $ f\in C^\infty(M)$.
\end{enumerate}
Let $\pi_{A}\in\mx^2(A^*)$ be the bivector field associated to this
almost Poisson structure.  We describe the vector bundle morphism
$\pi_{A}^\sharp\colon T^*A^*\to TA^*$, $\dr F\mapsto \{F,\cdot\}$,
$F\in C^\infty(A^*)$, associated to it.

We compute the vector fields $\pi_A^\sharp(\dr \ell_a)$ and
$\pi_A^\sharp(q_{A*}^*\theta)$ for all $a\in\Gamma(A)$ and
$\theta\in\Omega^1(M)$.  Since
$\pi_A^\sharp(q_{A^*}^*\dr f)(q_{A^*}^*g)=\pi_A^\sharp(\dr q_{A^*}^*
f)(q_{A^*}^*g)=0$
and $\pi_A^\sharp(\dr q_{A^*}^* f)(\ell_a)=-q_{A^*}^*(\rho(a)( f))$
for all $ f,\psi\in C^\infty(M)$ and $a\in\Gamma(A)$, we find
$\pi_A^\sharp(q_{A^*}^*\dr f)=-(\rho^t(\dr f))^\uparrow$ for all
$ f\in C^\infty(M)$ and consequently
$\pi_A^\sharp(q_{A^*}^*\theta)=-(\rho^t\theta)^\uparrow$ for all
$\theta\in\Omega^1(M)$.  In the same manner, we have
\begin{align*}
  \pi_A^\sharp(\dr \ell_a)(\ell_b)&=\ell_{[a,b]}\qquad \text{ and
  }\qquad \pi_A^\sharp(\dr \ell_a)(q_{A^*}^* f)=q_{A^*}^*(\rho(a)( f))
\end{align*}
for $a,b\in\Gamma(A)$ and $ f\in C^\infty(M)$.
Recall that the vector field $\widehat{\ldr{a}}\in\mx(A^*)$ 
satisfies 
$\widehat{\ldr{a}} (\alpha_m)(\ell_b)
=\ell_{[a,b]}(\alpha_m)$ and 
$\widehat{\ldr{a}} (\alpha_m)(q_{A^*}^* f)=\rho(a(m))( f)$
for $\alpha_m\in A^*$.  This shows the equality
$\pi_A^\sharp(\dr\ell_a)=\widehat{\ldr{a}}$.

\subsection{The canonical symplectic form on
  \texorpdfstring{$T^*E$}{T*E}}\label{pullback_canonical_symplectic}
Now let $M$ be a smooth manifold and $c_M\colon T^*M\to M$ its cotangent bundle.
Recall that there is a canonical $1$-form $\theta_{\rm
  can}\in\Omega^1(T^*M)$, 
given by 
\[ \langle \theta_{\rm can}(\eta_m), v_{\eta_m}\rangle=\langle \eta_m,
T_{\eta_m}c_M(v_{\eta_m})\rangle\]
for all $\eta_m\in T^*M$ and $v_{\eta_m}\in T_{\eta_m}(T^*M)$.  The
canonical symplectic form $\omega_{\rm can}\in\Omega^2(T^*M)$ is
defined by $\omega_{\rm can}=-\dr\theta_{\rm can}$.

Consider a vector bundle $E\to M$ endowed with a vector bundle
morphism $\sigma\colon E\to T^*M$ over the identity, and a connection
$\nabla\colon \mx(M)\times\Gamma(E)\to\Gamma(E)$.  The one-form
$\sigma^*\theta_{\rm can}\in\Omega^1(E)$ can be described as follows
\begin{align*}
  \langle(\sigma^*\theta_{\rm can})(e'_m),
  \sigma^\nabla_{TM}(X)(e'_m)\rangle
  &=\langle \theta_{\rm can}(\sigma(e'_m)), T_{e'_m}\sigma(\sigma^\nabla_{TM}(X) (e'_m))\rangle=\langle \sigma(e'_m), X(m)\rangle\\
  \langle(\sigma^*\theta_{\rm can})(e'_m),
  e^\uparrow(e'_m)\rangle&=\langle \theta_{\rm can}(\sigma(e'_m)),
  \sigma(e)^\uparrow(e'_m)\rangle=0
\end{align*}
for all $e_m'\in E$, $X\in\mx(M)$ and $e\in\Gamma(E)$.  This shows in
particular the equality \linebreak$\langle \sigma^*\theta_{\rm can},
\sigma_{TM}^\nabla(X)\rangle=\ell_{\sigma^t(X)}$.  As a consequence,
we get for all $e,e_1,e_2\in\Omega^1(M)$, $X,Y\in\mx(M)$:
\begin{equation*}
\begin{split}
  \sigma^*\omega_{\rm can}(\sigma_{TM}^\nabla(X),
  \sigma_{TM}^\nabla(Y))=&\,\sigma_{TM}^\nabla(X)((\sigma^*\theta_{\rm
    can})(\sigma_{TM}^\nabla(
  Y)))- \sigma_{TM}^\nabla(Y)((\sigma^*\theta_{\rm can})(\sigma_{TM}^\nabla(X)))\\
  &\qquad \qquad \qquad- (\sigma^*\theta_{\rm
    can})\left(\sigma_{TM}^\nabla[X,Y]-\widetilde{R_{\nabla}(X,Y)}\right)\\
  =&\,\ell_{\nabla^*_X(\sigma^tY)}-\ell_{\nabla^*_Y(\sigma^tX)}-\ell_{\sigma^t[X,Y]}=\ell_{\nabla^*_X(\sigma^tY)-\nabla^*_Y(\sigma^tX)-\sigma^t[X,Y]}\\
  \sigma^*\omega_{\rm can}(\sigma_{TM}^\nabla(X),
  e^\uparrow)=&\,\sigma_{TM}^\nabla(X)(0)-
  e^\uparrow(\ell_{\sigma^tX})- \sigma^*\theta_{\rm
    can}\left(\left[\sigma_{TM}^\nabla(X),
      e^\uparrow\right]\right)\\
&\,=-q_E^*\langle \sigma(e), X\rangle
\end{split}
\end{equation*}
and $\sigma^*\omega_{\rm can}(e_1^\uparrow, e_2^\uparrow)=0$.  Hence,
the one-forms $(\sigma^*\omega_{\rm
  can})^\flat(\sigma_{TM}^\nabla(X))$ and $(\sigma^*\omega_{\rm
  can})^\flat(e^\uparrow) \in\Omega^1(E)$ are given by
\[(\sigma^*\omega_{\rm can})^\flat(\sigma_{TM}^\nabla(X))=\dr
\ell_{-\sigma^tX}+ \widetilde{\sigma(\nabla_X\cdot)-\ldr{X}(\sigma(\cdot))},
\]
where $\sigma(\nabla_X\cdot)-\ldr{X}(\sigma(\cdot))$ is seen as a
section of $\operatorname{Hom}(E, T^*M)$, and
$(\sigma^*\omega_{\rm can})^\flat(e^\uparrow)=q_E^*(\sigma(e))$.

\section{Proof of Theorem \ref{super}}\label{proof_of_super}
In this section we prove Theorem~\ref{super}.
For simplicity, given a Dorfman connection 
\[\Delta\colon \Gamma(TM\oplus E^*)\times\Gamma(E\oplus T^*M)\to
\Gamma(E\oplus T^*M),
\]
we write $\tilde X=\pr_{TE}\left(\sigma^\Delta_{TM\oplus
    E^*}(X,\varepsilon)\right)$ and
$\tilde\varepsilon=\pr_{T^*E}\left(\sigma^\Delta_{TM\oplus
    E^*}(X,\varepsilon)\right)$.  The reader should bear in mind that
both $\tilde X$ and $\tilde\varepsilon$ depend on $X$ and
$\varepsilon$. More precisely, $\tilde X$ is the linear vector field
$\widehat{\nabla_{(X,\xi)}}$, with the connection $\nabla\colon
\Gamma(TM\oplus E^*)\times\Gamma(E)\to\Gamma(E)$ in Proposition
\ref{anchor_conn}.  Further, recall that by construction, the Dorfman
connection can be written
\begin{equation}\label{lem_structure_of_Dorfman}
\Delta_{(X,\varepsilon)}(e,\theta)=\Delta_{(X,\varepsilon)}(e,0)+(0,\ldr{X}\theta)
\end{equation}
for all $(X,\varepsilon)\in\Gamma(TM\oplus E^*)$ and
$(e,\theta)\in\Gamma(E\oplus T^*M)$.
We begin by proving three lemmas.
\begin{lemma}\label{lemma_before_main}
Choose $(X_1,\varepsilon_1), (X_2,\varepsilon_2)\in\Gamma(TM\oplus E^*)$ and
$e\in\Gamma(E)$.
Then
  \[\nabla_{(X_2,\varepsilon_2)}\nabla_{(X_1,\varepsilon_1)}e
=\pr_E\left(\Delta_{(X_2,\varepsilon_2)}\Delta_{(X_1,\varepsilon_1)}(e,0)
\right).\]
\end{lemma}

\begin{proof}
Choose $\varepsilon_3\in\Gamma(E^*)$.
Then
\begin{align*}
  &\langle \varepsilon_3, \pr_E\Delta_{(X_2,\varepsilon_2)}\left(\pr_E\Delta_{(X_1,\varepsilon_1)}(e,0), 0\right)\rangle=\langle (0,\varepsilon_3), \Delta_{(X_2,\varepsilon_2)}\left(\pr_E\Delta_{(X_1,\varepsilon_1)}(e,0), 0\right)\rangle\\
  &=\,X_2\langle (0,\varepsilon_3), \left(\pr_E\Delta_{(X_1,\varepsilon_1)}(e,0), 0\right)\rangle-\langle\lb(X_2,\varepsilon_2),(0,\varepsilon_3)\rb_\Delta, \left(\pr_E\Delta_{(X_1,\varepsilon_1)}(e,0), 0\right)\rangle\\
  &=\,X_2\langle (0,\varepsilon_3), \Delta_{(X_1,\varepsilon_1)}(e,0)\rangle-\langle\lb(X_2,\varepsilon_2),(0,\varepsilon_3)\rb_\Delta, \Delta_{(X_1,\varepsilon_1)}(e,0)\rangle\\
  &=\,\langle
  (0,\varepsilon_3),\Delta_{(X_2,\varepsilon_2)}\Delta_{(X_1,\varepsilon_1)}(e,0)\rangle=\langle
  \varepsilon_3,
  \pr_E\Delta_{(X_2,\varepsilon_2)}\Delta_{(X_1,\varepsilon_1)}(e,0)\rangle.
\end{align*}
In the third equality, we have used $\pr_{TM}\lb(X_2,\varepsilon_2),(0,\varepsilon_3)\rb_\Delta=[X_2,0]=0$.
\end{proof}

\begin{lemma}\label{technical_one}
Choose $(X,\varepsilon)\in\Gamma(TM\oplus E^*)$ and
$e\in\Gamma(E)$. Then 
\begin{enumerate}
\item $\langle\tilde\varepsilon, e^\uparrow\rangle=q_E^*\langle \varepsilon, e\rangle$.
\item $\ldr{e^\uparrow}\tilde\varepsilon=q_E^*\left(\dr\langle
    \varepsilon,e\rangle-\pr_{T^*M}\Delta_{(X,\varepsilon)}(e,0)\right)$.
\item  $\left[\tilde X, e^\uparrow\right]=(\nabla_{(X,\varepsilon)}e)^\uparrow$.
\end{enumerate}
\end{lemma}

\begin{proof}
The first claim is immediate by the definition of $\tilde\varepsilon$.
For any $e'\in\Gamma(E)$, we have 
\begin{align*}
  \langle \ldr{e^\uparrow}\tilde\varepsilon,
  e'^\uparrow\rangle&=e^\uparrow(\langle \tilde \varepsilon,
  e'^\uparrow\rangle)-\langle \tilde \varepsilon, [e^\uparrow,
  e'^\uparrow]\rangle=e^\uparrow(q_E^*\langle \varepsilon,
  e'\rangle)-\langle \tilde \varepsilon, 0\rangle=0.
\end{align*}
This shows that $\ldr{e^\uparrow}\tilde\varepsilon$ is vertical, i.e. the
pullback under $q_E$ of a $1$-form on $M$. Thus, we just need 
to compute 
$\langle (\ldr{e^\uparrow}\tilde\varepsilon)(e'(m)), T_me'v_m\rangle$
for $e'\in\Gamma(E)$ and $v_m\in TM$.
But we have 
\begin{align*}
  \langle (\ldr{e^\uparrow}\tilde\varepsilon)(e'(m)),
  T_me'v_m\rangle&=\left.\frac{d}{dt}\right\an{t=0}
  \langle \tilde \varepsilon (e'(m)+te(m)), T_{e'(m)}\phi_t^{e^\uparrow}(T_me'v_m)\rangle\\
  &=\left.\frac{d}{dt}\right\an{t=0}
  \langle \tilde \varepsilon (e'(m)+te(m)), T_m(e'+te)v_m\rangle\\
  &=\left.\frac{d}{dt}\right\an{t=0} v_m\langle \varepsilon,e'+te\rangle- \langle\pr_{T^*M}\Delta_{(X,\varepsilon)}(e'+te,0),v_m\rangle\\
  &=v_m\langle \varepsilon,e\rangle-
  \langle\pr_{T^*M}\Delta_{(X,\varepsilon)}(e,0),v_m\rangle.
\end{align*}

For the third equality we just need to compute $[\tilde X,
e^\uparrow](\ell_{\varepsilon'})$ for sections $\varepsilon'\in\Gamma(E^*)$ and
$[\tilde X, e^\uparrow](q_E^*f)$ for functions $f\in C^\infty(M)$.  We
have
\begin{align*} [\tilde X, e^\uparrow](\ell_{\varepsilon'})&=\tilde
  X(e^\uparrow(\ell_{\varepsilon'}))-e^\uparrow(\tilde X(\ell_{\varepsilon'}))
  =\tilde X(q_E^*\langle\varepsilon',e\rangle)-e^\uparrow(\ell_{\nabla_{(X,\varepsilon)}^*\varepsilon'})\\
  &=q_E^*\left(X\langle\varepsilon',e\rangle-\langle\nabla_{(X,\varepsilon)}^*\varepsilon',
    e\rangle\right)=(\nabla_{(X,\xi)}e)^\uparrow(\ell_{\varepsilon'}),
\end{align*}
and $[\tilde X, e^\uparrow](q_E^*f)=0=(\nabla_{(X,\xi)}e)^\uparrow(q_E^*f)$ since $e^\uparrow\sim_{q_E} 0$
and $\tilde X\sim_{q_E}X$.
\end{proof}

\medskip

Next note that since $\tilde X$ is linear over $X$, the flow
$\phi^{\tilde X}_t$ of $\tilde X$ is a vector bundle morphism $E\to E$
over $\phi^X_t\colon M\to M$, for any $t\in\R$ where this is
defined. Hence, for any section $e\in\Gamma(E)$, we can define a new
section $\psi^{
    (X,\varepsilon)}_t(e)\in\Gamma(E)$ by
\[\psi^{
    (X,\varepsilon)}_t(e)=\phi_{-t}^{\tilde X}\circ e\circ \phi^X_t.
\]
\begin{lemma}\label{derivation_and_connection}
The time derivative of $\psi^{
    (X,\varepsilon)}_t$ satisfies
\[\left.\frac{d}{dt}\right\an{t=0}\psi^{
    (X,\varepsilon)}_t(e)=\nabla_{(X,\varepsilon)}e.\]
\end{lemma}

\begin{proof}
  The curve $c\colon t\mapsto \psi^X_t(e)(m)$ is a curve in $E$ with
  $c(0)=e_m$ and satisfying $q_E\circ c=m$.  Hence, the derivative
  $\dot c(0)$ is a vertical vector over $e_m$.  Since $\phi^{\tilde
    X}_t$ is linear, we have
  \[((\phi^{\tilde X}_t)^*e^\uparrow)(e_m')
  =\left.\frac{d}{ds}\right\an{s=0}\phi^{\tilde X}_{-t}(\phi_t^{\tilde
    X}(e_m')+se(\phi^X_t(m)))
  =\left.\frac{d}{ds}\right\an{s=0}e_m'+s\psi^{
    (X,\varepsilon)}_{t}(e)(m)
\]
for $e_m'\in E$.
Thus, we get for any $\varepsilon\in\Gamma(E^*)$:
\begin{align*}
[\tilde X,
e^\uparrow](\ell_\varepsilon)(e_m')&=\left.\frac{d}{dt}\right\an{t=0}\langle
\dr_{e_m'}\ell_\varepsilon, ((\phi^{\tilde
  X}_t)^*e^\uparrow)\rangle=\left.\frac{d}{dt}\right\an{t=0}\left.\frac{d}{ds}\right\an{s=0}\langle
\varepsilon(m), e_m'+s\psi^{
    (X,\varepsilon)}_t(e)(m)\rangle\\
&=\left.\frac{d}{ds}\right\an{s=0}\left.\frac{d}{dt}\right\an{t=0}
\langle \varepsilon(m), e_m'+s\psi^{\  X}_{t}(e)(m)\rangle\\
&=\left.\frac{d}{ds}\right\an{s=0}s\left\langle\varepsilon(m),
\left.\frac{d}{dt}\right\an{t=0}\psi^X_t(e)(m)\right\rangle=\left\langle\varepsilon(m),
\left.\frac{d}{dt}\right\an{t=0}\psi^X_t(e)(m)\right\rangle.
\end{align*}
This shows that 
\[[\tilde X, e^\uparrow]=\left( \left.\frac{d}{dt}\right\an{t=0}\psi^X_t(e)\right)^\uparrow.
\]
By (3) of Lemma~\ref{technical_one}, we are done.
\end{proof}

Now we can prove Theorem~\ref{super}. We write $\tau=(e,\theta)$,
$\tau_i=(e_i,\theta_i)$ and $\nu=(X,\varepsilon)$, $\nu_i=(X_i,\varepsilon_i)$
for $i=1,2$.
\begin{proof}[Proof of Theorem~\ref{super}]
  The first equality is easy to check: for the tangent part, we use
  the commutativity of the flows of the vertical vector fields. For
  the cotangent part, note that since $e_i^\uparrow\sim_{q_E}0$ for
  $i=1,2$, we get immediately
  $\ldr{e_1^\uparrow}q_E^*\theta_2-\ip{e_2^\uparrow}\dr
  q_E^*\theta_1=q_E^*\left(\ldr{0}\theta_2-\ip{0}\dr\theta_1\right)=0$.

  For the second equality, we know by Lemma~\ref{technical_one} that
  $[\tilde X,
  e^\uparrow]=(\pr_E\Delta_{(X,\varepsilon)}(e,0))^\uparrow$.  We
  compute the cotangent part of the Courant-Dorfman bracket.  Using
  Lemma~\ref{technical_one}, we have
\begin{align*}
  \ldr{\tilde X}q_E^*\theta-\ip{e^\uparrow}\dr\tilde\varepsilon&=
  \ldr{\tilde X}q_E^*\theta-\ldr{e^\uparrow}\tilde\varepsilon+\dr\langle \tilde \varepsilon,e^\uparrow\rangle\\
  &=\ldr{\tilde
    X}q_E^*\theta-q_E^*\left(\dr\langle\varepsilon,e\rangle- \pr_{T^*M}\Delta_{(X,\varepsilon)}(e,0)\right)+\dr q_E^*\langle \varepsilon,e\rangle\\
  &=q_E^*(\ldr{X}\theta+\pr_{T^*M}\Delta_{(X,\varepsilon)}(e,0))\overset{\eqref{lem_structure_of_Dorfman}}{=}
  q_E^*\pr_{T^*M}\Delta_\nu\tau.
\end{align*}
This leads to our claim $\left\lb \sigma_{TM\oplus E^*}^\Delta(\nu),
  \tau^\uparrow\right\rb=\Delta_{\nu}\tau^\uparrow$.  \medskip

For the last equality choose a section $\tau=(e,\theta)$ of $E\oplus T^*M$. Then:
\begin{align*}
  \left\langle \ldr{\tilde X_1}\tilde\varepsilon_2-\ip{\tilde
      X_2}\dr\tilde\varepsilon_1, e^\uparrow\right\rangle &=\tilde
  X_1\langle \tilde\varepsilon_2, e^\uparrow\rangle
  -\left\langle\tilde\varepsilon_2, \left[\tilde
      X_1,e^\uparrow\right]\right\rangle -\tilde X_2\langle
  \tilde\varepsilon_1, e^\uparrow\rangle+
  e^\uparrow\langle\tilde\varepsilon_1,\tilde X_2\rangle
  +\left\langle\tilde\varepsilon_1, \left[\tilde
      X_2,e^\uparrow\right]\right\rangle
\end{align*}
First we have $\langle \tilde\varepsilon_2,
e^\uparrow\rangle=q_E^*\langle\varepsilon_2,e\rangle$ by
Lemma~\ref{technical_one}, and consequently $\tilde X_1\langle
\tilde\varepsilon_2,
e^\uparrow\rangle=q_E^*(X_1\langle\varepsilon_2,e\rangle)$.  Then, we
get $\left\langle\tilde\varepsilon_2, \left[\tilde
    X_1,e^\uparrow\right]\right\rangle =q_E^*\langle \varepsilon_2,
\nabla_{(X_1,\varepsilon_1)}e\rangle$ by (3) of
Lemma~\ref{technical_one}, and $\langle\tilde\varepsilon_1,\tilde
X_2\rangle(e_m)=X_2(m)\langle\varepsilon_1,e\rangle-\langle\varepsilon_1,\nabla_{\nu_2}e\rangle(m)-\langle
X_2, \pr_{T^*M}\Delta_{\nu_1}(e,0)\rangle(m)$, which defines a linear
function on $E$.  This yields
$e^\uparrow\langle\tilde\varepsilon_1,\tilde
X_2\rangle=q_E^*\left(X_2\langle
  \varepsilon_1,e\rangle-\langle\varepsilon_1,\nabla_{\nu_2}e\rangle-\langle
  X_2, \pr_{T^*M}\Delta_{\nu_1}(e,0)\rangle\right)$.  Thus, we get
\begin{align*}
  \left\langle \ldr{\tilde X_1}\tilde\varepsilon_2-\ip{\tilde
      X_2}\dr\tilde\varepsilon_1, e^\uparrow\right\rangle
  =\,&q_E^*\left(X_1\langle\varepsilon_2,e\rangle-\langle\varepsilon_2,
    \nabla_{\nu_1}e\rangle
    -\cancel{X_2\langle \varepsilon_1, e\rangle}+\cancel{ X_2\langle \varepsilon_1,e\rangle}-\cancel{\langle\varepsilon_1,\nabla_{\nu_2}e\rangle}\right.\\
  &\left.-\langle X_2, \pr_{T^*M}\Delta_{\nu_1}(e,0)\rangle
    +\cancel{\langle\varepsilon_1, \nabla_{\nu_2}e\rangle}
  \right)\\
  =\,&q_E^*\left(X_1\langle\varepsilon_2,e\rangle-\langle (X_2,\varepsilon_2),
    \Delta_{\nu_1}(e,0)\rangle \right)
  =q_E^*\langle\lb\nu_1,\nu_2\rb_\Delta,(e,0)\rangle.
\end{align*}
This leads to $\left\langle \left\lb \sigma_{TM\oplus
      E^*}^\Delta(\nu_1), \sigma_{TM\oplus
      E^*}^\Delta(\nu_2)\right\rb, (e,\theta)^\uparrow\right\rangle
=q_E^*\langle \lb\nu_1, \nu_2\rb_\Delta,(e,\theta)\rangle$, which
shows that
\[\left\lb \sigma_{TM\oplus E^*}^\Delta(\nu_1), \sigma_{TM\oplus E^*}^\Delta(\nu_2)\right\rb(e_m)=(T_me[X_1,X_2](m),
\dr_{e_m}\ell_{\pr_{E^*}\lb \nu_1, \nu_2\rb_\Delta})+\tau^\uparrow(e_m),
\]
for some $\tau\in\Gamma(E\oplus T^*M)$.

Hence we know that
for any $(X,\varepsilon)\in\Gamma(TM\oplus E^*)$, we have
\begin{equation}\label{first_eq}
\begin{split}
\langle \tau, (X,\varepsilon)\rangle(m)
&=\left\langle \left\lb \sigma_{TM\oplus E^*}^\Delta(\nu_1), \sigma_{TM\oplus E^*}^\Delta(\nu_2)\right\rb(e_m), (T_meX(m), \dr_{e_m}\ell_\varepsilon)\right\rangle\\
&\qquad\qquad-X(m)\langle \lb \nu_1,\nu_2\rb_\Delta, (e,0)\rangle 
-[X_1,X_2]\langle \varepsilon, e\rangle.
\end{split}
\end{equation}

First we find $\left[\tilde X_1, \tilde
  X_2\right](\ell_\varepsilon)=\ell_{\nabla^*_{\nu_1}\nabla^*_{\nu_2}\varepsilon-\nabla^*_{\nu_2}\nabla^*_{\nu_1}\varepsilon}$.
Next, we compute $\langle\ldr{\tilde X_1}\tilde \varepsilon_2,
T_meX(m)\rangle$.  Using Lemma~\ref{derivation_and_connection} and the
identity $\phi^{\tilde
  X_1}_t(e_m)=\psi_{-t}^{X_1}(e)(\phi^{X_1}_t(m))$, we find
\begin{align*}
&\langle \ldr{\tilde X_1}\tilde\varepsilon_2(e_m), T_meX(m)\rangle=
\left.\frac{d}{dt}\right\an{t=0}\left\langle \tilde\varepsilon_2(\phi_t^{\tilde
    X_1}(e_m)), (T_{e_m}\phi_t^{\tilde X_1}\circ T_me)X(m)\right\rangle\\
&=\left.\frac{d}{dt}\right\an{t=0}\langle
\tilde\varepsilon_2(\psi_{-t}^{X_1}(e)(\phi^{X_1}_t(m))),
T_{\phi^{X_1}_t(m)}\psi_{-t}^{X_1}(e)((\phi^{X_1}_{-t})^*(X)(\phi_t^{X_1}(m)))\rangle\\
&=\left.\frac{d}{dt}\right\an{t=0}
((\phi^{X_1}_{-t})^*(X))\langle\varepsilon_2, \psi_{-t}^{X_1}(e)\rangle(\phi_t^{X_1}(m))\\
&\qquad \qquad 
-\left.\frac{d}{dt}\right\an{t=0}
\langle \pr_{T^*M}\Delta_{(X_2,\varepsilon_2)}(\psi_{-t}^{X_1}(e),0),
(\phi^{X_1}_{-t})^*(X)\rangle(\phi_t^{X_1}(m))\\
&=\left(-[X_1,X]\langle\varepsilon_2,e\rangle +X_1X\langle \varepsilon_2, e\rangle
-X\langle \varepsilon_2, \pr_E\Delta_{(X_1,\varepsilon_1)}(e,0)\rangle-X_1\langle \pr_{T^*M}\Delta_{(X_2,\varepsilon_2)}(e,0), X\rangle\right.\\
&\qquad\left.
+\langle\pr_{T^*M}\Delta_{(X_2,\varepsilon_2)}(e,0), [X_1,X]\rangle+\langle\pr_{T^*M}\Delta_{(X_2,\varepsilon_2)}(\pr_E\Delta_{(X_1,\varepsilon_1)}(e,0),0), X\rangle
\right)(m).
\end{align*}
We also have 
\begin{align*}
\langle\dr_{e_m}\langle\tilde\varepsilon_1,\tilde X_2\rangle,
T_meX(m)\rangle&=X\left( \langle\tilde\varepsilon_1,\tilde X_2\rangle\circ
  e\right)\\
&=X\left(X_2\langle \varepsilon_1,e\rangle-\langle\pr_{T^*M}\Delta_{\nu_1}(e,0),
  X_2\rangle-\langle\varepsilon_1, \nabla_{\nu_2}e\rangle\right),
\end{align*}
which leads to
\begin{align*}
&\langle \ldr{\tilde X_1}\tilde\varepsilon_2-\ip{\tilde X_2}\dr\tilde\varepsilon_1,
T_meX(m)\rangle=\langle \ldr{\tilde X_1}\tilde\varepsilon_2-\ldr{\tilde X_2}\tilde\varepsilon_1+ \dr\langle\tilde\varepsilon_1,\tilde X_2\rangle,
T_meX(m)\rangle\\
=&\left(XX_1\langle \varepsilon_2, e\rangle
-X\langle \varepsilon_2, \nabla_{\nu_1}e\rangle-X_1\langle \pr_{T^*M}\Delta_{\nu_2}(e,0), X\rangle
+\langle\pr_{T^*M}\Delta_{\nu_2}(e,0), [X_1,X]\rangle\right.\\
&\qquad+\langle\pr_{T^*M}\Delta_{\nu_2}(\nabla_{\nu_1}e,0),
X\rangle-\cancel{XX_2\langle \varepsilon_1, e\rangle}
+\cancel{X\langle \varepsilon_1, \nabla_{\nu_2}e\rangle}\\
&\qquad +X_2\langle \pr_{T^*M}\Delta_{\nu_1}(e,0), X\rangle
-\langle\pr_{T^*M}\Delta_{\nu_1}(e,0), [X_2,X]\rangle-\langle\pr_{T^*M}\Delta_{\nu_1}(\nabla_{\nu_2}e,0),
X\rangle\\
&\qquad \left.+\cancel{XX_2\langle \varepsilon_1,e\rangle}-X\langle\pr_{T^*M}\Delta_{\nu_1}(e,0),
  X_2\rangle-\cancel{X\langle\varepsilon_1, \nabla_{\nu_2}e\rangle}
\right)(m)
\end{align*}
The first, second and last remaining terms add up to
$ X(X_1\langle \lb\nu_1, \nu_2\rb_\Delta, (e,0)\rangle -
\langle\Delta_{\nu_1}(e,0), \nu_2\rangle =X\langle \lb\nu_1,
\nu_2\rb_\Delta, (e,0)\rangle $.
The fifth remaining term can be written
$\langle\Delta_{\nu_2}(\pr_E\Delta_{\nu_1}(e,0),0) ,
(X,0)\rangle$. But this equals
\begin{equation*}
\begin{split}
&X_2\langle (\pr_E\Delta_{\nu_1}(e,0),0) , (X,0)\rangle-\langle
(\pr_E\Delta_{\nu_1}(e,0),0), \lb \nu_2, (X,0)\rb_\Delta\rangle\\
&=0-\langle \Delta_{\nu_1}(e,0), \lb \nu_2,
(X,0)\rb_\Delta\rangle+\langle \pr_{T^*M}\Delta_{\nu_1}(e,0), [X_2,X]\rangle,
\end{split}
\end{equation*}
which, together with the seventh remaining term, add up to
$-\langle \Delta_{\nu_1}(e,0), \lb \nu_2,
(X,0)\rb_\Delta\rangle$. 
This and the sixth remaining term add up to
$\langle\Delta_{\nu_2}\Delta_{\nu_1}(e,0), (X,0)\rangle$.
Similarly, the eighth, third and fourth remaining terms add up to 
$-\langle\Delta_{\nu_1}\Delta_{\nu_2}(e,0), (X,0)\rangle$. This leads to
\begin{align*}
&\langle \ldr{\tilde X_1}\tilde\varepsilon_2-\ip{\tilde X_2}\dr\tilde\varepsilon_1,
T_meX(m)\rangle\\
&=\left(X\langle \lb\nu_1, \nu_2\rb_\Delta, (e,0)\rangle-\langle \Delta_{\nu_1}\Delta_{\nu_2}(e,0)-\Delta_{\nu_2}\Delta_{\nu_1}(e,0), (X,0)\rangle\right)(m).
\end{align*}
Now we find that \eqref{first_eq} reads 
\begin{align*}
&\langle \tau, (X,\varepsilon)\rangle
=\langle
   \nabla^*_{\nu_1}\nabla^*_{\nu_2}\varepsilon-\nabla^*_{\nu_2}\nabla^*_{\nu_1}\varepsilon,
   e\rangle-[X_1,X_2]\langle\varepsilon, e\rangle-\langle
\Delta_{\nu_1}\Delta_{\nu_2}(e,0)-\Delta_{\nu_2}\Delta_{\nu_1}(e,0),
(X,0)\rangle\\
&=\langle R_{\nabla^*}(\nu_1,\nu_2)\varepsilon+\nabla^*_{\lb
   \nu_1,\nu_2\rb_\Delta}\varepsilon, e\rangle-[X_1,X_2]\langle\varepsilon, e\rangle-\langle
R_\Delta(\nu_1,\nu_2)(e,0)+\Delta_{\lb\nu_1,\nu_2\rb_\Delta}(e,0),
(X,0)\rangle\\
&=\langle -R_{\nabla}(\nu_1,\nu_2)e-\nabla_{\lb
   \nu_1,\nu_2\rb_\Delta}e, \varepsilon\rangle-\langle
R_\Delta(\nu_1,\nu_2)(e,0)+\Delta_{\lb\nu_1,\nu_2\rb_\Delta}(e,0),
(X,0)\rangle\\
&=\langle -R_\Delta(\nu_1, \nu_2)(e,0)-\Delta_{\lb\nu_1,\nu_2\rb_\Delta}(e,0), (X,\varepsilon)\rangle.
\end{align*}
In the last equality, we have used Lemma~\ref{lemma_before_main}.
This shows 
\[\tau=-R_\Delta(\nu_1, \nu_2)(e,0)-\Delta_{\lb\nu_1,\nu_2\rb_\Delta}(e,0).
\qedhere\]
\end{proof}

\section{The Lie algebroid structure on \texorpdfstring{$TA\oplus T^*A\to TM\oplus A^*$}{TA+T*A->TM+A*}}\label{big_lie_algebroids}
Let $(q_A\colon A\to M, \rho, [\cdot\,,\cdot])$ be a Lie algebroid. We
describe here the Lie algebroid structures on $TA\to TM$,
$T^*A\to A^*$ and $TA\oplus T^*A\to TM\oplus A^*$.  For simplicity, we
write $q:=q_A\colon A\to M$ and $q_*:=q_{A^*}\colon A^*\to M$ for the
vector bundle maps.

\subsubsection*{The Lie algebroid $TA\to TM$}
Recall that for $a\in\Gamma(A)$, we have two particular types 
of sections of $TA\to TM$:
the \emph{linear} sections 
$Ta\colon TM\to TA$, which are vector bundle morphisms 
over $a\colon M\to A$, and the \emph{core} sections
$a^\dagger\colon TM\to TA$, 
$a^\dagger(v_m)=T_m0^Av_m+_{p_A}\left.\frac{d}{dt}\right\an{t=0}t\cdot a(m)$.
The Lie algebroid bracket on sections of  $TA\to TM$
is given by 
\begin{align*}
[Ta_1, Ta_2]&=T[a_1,a_2],\qquad [Ta_1, a_2^\dagger]=[a_1,a_2]^\dagger,\qquad [a_1^\dagger, a_2^\dagger]=0
\end{align*}
and the anchor is given by
$\rho_{TA}(Ta)=\widehat{[\rho(a),\cdot]}\in\mx(TM)$,
$\rho_{TA}(a^\dagger)= (\rho(a))^\uparrow\in\mx(TM)$ (see
\cite{MaXu94}).

\subsubsection*{The Lie algebroid $T^*A\to A^*$}
There is an isomorphism of double vector bundles
\begin{align*}
\begin{xy}
\xymatrix{
T^*A^*\ar[r]^{r_{A^*}}\ar[d]_{c_{A^*}}&A\ar[d]\\
A^*\ar[r] &M
}
\end{xy}
\qquad \overset{R}\longrightarrow\qquad 
\begin{xy}
\xymatrix{
T^*A\ar[r]^{c_A}\ar[d]_{r_A}&A\ar[d]\\
A^*\ar[r] &M
}
\end{xy}
\end{align*}
over the identity on the sides, and $-\operatorname{id}_{T^*M}$ on
the core $T^*M$.
The map $R$ is given as follows: 
for $\theta\in\Omega^1(M)$, 
we have 
\[R(q_*^*\theta(\alpha_m))=\dr_{0^A_m}\ell_\alpha-q^*\theta(0^A_m)\]
and for $\alpha\in \Gamma(A^*)$ and $a\in\Gamma(A)$, we have 
\[R(\dr_{\alpha(m)}\ell_a)=\dr_{a(m)}(\ell_\alpha-q^*\langle\alpha, a\rangle)
\]
for all $m\in M$.  Hence, we find that for $\theta\in\Omega^1(M)$, the
core section $\theta^\dagger\in\Gamma_{A^*}(T^*A)$ is
$\alpha_m\mapsto R(-q_*^*\theta(\alpha_m))$.  For $a\in\Gamma(A)$, we write
$a^R\in\Gamma_{A^*}(T^*A)$ for the section
$\alpha_m\mapsto R(\dr_{\alpha(m)}\ell_a)$.

Recall that since $A$ is a Lie algebroid, its dual $A^*$ is endowed with a linear Poisson structure
given by
\begin{align*}
\{\ell_{a_1},\ell_{a_2}\}&=\ell_{[a,b]}, \qquad \{\ell_a,q_*^*f\}=q_*^*(\rho(a)(f)), \qquad 
\{q_*^* f, q_*^*g\}=0
\end{align*}
for all $a_1, a_2\in\Gamma(A)$ and $ f, g\in C^\infty(M)$.
Hence, there is a Lie algebroid structure on $T^*A^*\to A^*$ 
associated to this Poisson structure, and the Lie algebroid 
structure on $T^*A\to A^*$
is exactly such that 
the isomorphism
$R\colon  T^*A^*\to T^*A$
is an isomorphism of  Lie algebroids \cite{MaXu94,MaXu98}.

Therefore, we first give the Lie brackets and images under the anchor
map $\rho_{T^*A^*}$ of the sections $\dr \ell_a$ and
$q_*^*\theta\in\Omega^1(A^*)=\Gamma_{A^*}(T^*A^*)$, for
$\theta\in\Omega^1(M)$ and $a\in \Gamma(A)$.  By the definition of the
Lie algebroid structure $T^*A^*\to A^*$ associated to the linear
Poisson structure on $A^*$, one finds easily that the Lie algebroid
structure on $T^*A^*\to A^*$ is given by the following identities:
\begin{align*}
[\dr \ell_a,\dr \ell_b]&=\dr \ell_{[a,b]}, \qquad [\dr \ell_a, q_*^*\theta]=q_*^*(\ldr{\rho(a)}\theta), \qquad 
[q_*^*\theta, q_*^*\theta]=0,\\
\rho_{T^*A^*}(\dr \ell_a)&=\widehat{\ldr{a}}\in\mx(A^*), \qquad 
\rho_{T^*A^*}(q_*^*\theta)=(-\rho^t\theta)^\uparrow \in\mx(A^*)
\end{align*}
for $a,b\in\Gamma(A)$
and $\theta,\theta\in\Omega^1(M)$.
As a consequence, we find that 
the Lie algebroid structure on $T^*A\to A^*$ is given by 
\begin{align*}
[a_1^R, a_2^R]&=[a_1,a_2]^R, \qquad [a_1^R, \theta^\dagger]=(\ldr{\rho(a_1)}\theta)^\dagger, \qquad
[\theta_1^\dagger, \theta_2^\dagger]=0,\\
\rho_{T^*A}(a^R)&=\widehat{\ldr{a}}\in\mx(A^*), \qquad
\rho_{T^*A}(\theta^\dagger)=(\rho^t\theta)^\uparrow \in\mx(A^*)
\end{align*}
for $a_1,a_2\in\Gamma(A)$  
and $\theta_1,\theta_2\in\Omega^1(M)$.

\subsubsection*{The fibered product $TA\times_AT^*A\to TM\times_MA^*$}
The Lie algebroid $TA\oplus T^*A\to TM\oplus A^*$ is defined 
as the pullback to the diagonals $\Delta_A\to \Delta_M$
of the Lie algebroid $TA\times T^*A\to TM\times A^*$.
We have the special sections 
\[a^l:=(Ta, a^R)\colon TM\oplus A^*\to TA\oplus T^*A
\]
for $a\in \Gamma(A)$ 
and 
\[(a,\theta)^\dagger:=(a^\dagger, \theta^\dagger)\colon TM\oplus A^*\to TA\oplus T^*A
\]
for $(a,\theta)\in\Gamma(A\oplus T^*M)$.
The set of sections of $TA\oplus T^*A\to TM\oplus A^*$ is spanned as a $C^\infty(TM\oplus A^*)$-module
by these two types of sections.
We write $\pi\colon TM\oplus A^*\to M$ for the projection 
and $\Theta\colon TA\oplus T^*A\to T(TM\oplus A^*)$ for the anchor of $TA\oplus T^*A\to TM\oplus A^*$.

\begin{proposition}\label{structure_of_TAT*A}
The Lie algebroid $(TA\oplus T^*A, \Theta, [\cdot\,,\cdot])$ is described by the following identities
\begin{align*}
[a_1^l, a_2^l]&=[a_1,a_2]^l, \qquad [a^l, \tau^\dagger]=(\ldr{a}\tau)^\dagger, \qquad 
[\tau_1^\dagger, \tau_2^\dagger]=0\\
[a^l, \widetilde{\phi}]&=\widetilde{\ldr{a}\phi}, \qquad 
[\tau^\dagger,\widetilde{\phi}]=\widetilde{\phi((\rho,\rho^t)\tau},\qquad
[\widetilde{\phi},\widetilde{\psi}]=\widetilde{\psi\circ(\rho,\rho^t)\circ\phi}-\widetilde{\phi\circ(\rho,\rho^t)\circ\psi},\\
\Theta(a^l)&=\widehat{\ldr{a}}, \qquad
\Theta(\tau^\dagger)=((\rho,\rho^t)\tau)^\uparrow, \qquad
\Theta(\widetilde{\phi})=\widetilde{(\rho,\rho^t)\circ\phi}
\end{align*}
for $a,b\in\Gamma(A)$, $\sigma,\tau\in\Gamma(A\oplus T^*M)$ and  $\phi, \psi\in\Gamma(\Hom(TM\oplus A^*, A\oplus T^*M))$.
\end{proposition}

\begin{proof}
  We start by computing the anchor. Note first that for all
  $ f\in C^\infty(M)$ and $\tau'=(a',\theta')\in\Gamma(A\oplus T^*M)$, 
we have
$\pi^* f=\pr_{TM}^*p_M^* f=\pr_{A^*}^*q_*^* f$ and
$\ell_{\tau'}=\pr_{TM}^*\ell_{\theta'}+\pr_{A^*}^*\ell_{a'}$.
Thus, we get:
\begin{align*}
\Theta(a^l)(\ell_{\tau'})&=(\rho_{TA}\circ Ta, \rho_{T^*A}\circ a^R)(\pr_{TM}^*\ell_{\theta'}+\pr_{A^*}^*\ell_{a'})\\
&=\pr_{TM}^*(\rho_{TA}\circ Ta)(\ell_{\theta'})+\pr_{A^*}^*(\rho_{T^*A}\circ a^R)(\ell_{a'})\\
&=\pr_{TM}^*\ell_{\ldr{\rho(a)}\theta'}+\pr_{A^*}^*\ell_{[a,a']}=\ell_{\ldr{a}\tau'},\\
\Theta(a^l)(\pi^* f)&=(\rho_{TA}\circ Ta, \rho_{T^*A}\circ a^R)(\pr_{TM}^*p_M^* f)=\pr_{TM}^*p_M^*(\rho(a)( f))=\pi^*(\rho(a)( f)),\\
\Theta(\tau^\dagger)(\ell_{\tau'})&=(\rho_{TA}\circ {a}^\dagger ,
\rho_{T^*A}\circ \theta^\dagger)(\pr_{TM}^*\ell_{\theta'}+\pr_{A^*}^*\ell_{a'})\\
&=\pr_{TM}^*(\rho_{TA}\circ
{a}^\dagger)(\ell_{\theta'})+\pr_{A^*}^*(\rho_{T^*A}\circ
\theta^\dagger)(\ell_{a'})\\
&=\pr_{TM}^*p_M^*\langle \theta', \rho({a})\rangle
+\pr_{A^*}^*q_*^*\langle\theta, \rho({a'})\rangle=\pi^*\langle (\rho,\rho^t)\tau,
\tau'\rangle,\\
\Theta(\tau^\dagger)(\pi^* f)&=(\rho_{TA}\circ {a}^\dagger ,
\rho_{T^*A}\circ
\theta^\dagger)(\pr_{TM}^*p_M^* f)=0.
\end{align*}
For the last equality, 
note that a section $\phi\in\Gamma(\Hom(TM\oplus A^*, A\oplus T^*M))$
can be written as a sum $\phi=\sum_{i} \ell_{\chi_i}\cdot \tau_i$
with  $\chi_i,\tau_i\in\Gamma(A\oplus T^*M)$.
The corresponding section $\widetilde{\phi}\in\Gamma_{TM\oplus A^*}(TA\oplus T^*A)$
is then given by 
$\widetilde{\phi}=\sum_{i} \ell_{\chi_i}\cdot \tau_i^\dagger$
and we get 
\begin{align*}
\Theta(\widetilde{\phi})&=\sum_{i}
\ell_{\chi_i}\cdot \Theta(\tau_i^\dagger)=\sum_{i}
\ell_{\chi_i}\cdot ((\rho,\rho^t)\tau_i)^\uparrow=\widetilde{(\rho,\rho^t)\circ\phi}.
\end{align*}

Next we compute the Lie algebroid brackets.
For $a_1,a_2\in\Gamma(A)$ and $\theta_1,\theta_2\in\Omega^1(M)$, we have 
$[(Ta_1,{a_1}^R), (Ta_2, {a_2}^R)]=(T[a_1,a_2], [a_1,a_2]^R)$
by the considerations in the previous sections.
In the same manner, we show the next two identities:
$[a^l, (a', \theta')^\dagger]=([a,a'],
\ldr{\rho(a)}\theta')^\dagger$
and $[(a_1, \theta_1)^\dagger, (a_2,\theta_2)^\dagger]=(0,0)$.
For the last three brackets we assume without loss of
generality that $\phi=  \ell_{\chi_1}\cdot \tau_1$ and 
$\psi= \ell_{\chi_2}\cdot\tau_2$ with
$\chi_1,\chi_2,\tau_1,\tau_2\in\Gamma(A\oplus T^*M)$. Then
\begin{align*}
[a^l, \widetilde{\phi}]&=\left[a^l, 
  \ell_{\chi_1}\cdot \tau_1^\dagger\right]=\ell_{\ldr{a}\chi_1}\cdot \tau_1^\dagger
+ \ell_{\chi_1}\cdot (\ldr{a}\tau_1)^\dagger=
\widetilde{\ldr{a}\phi},
\end{align*}
since 
\begin{align*}
(\ldr{a}\phi)(\nu)&=\ldr{a}(\phi(\nu))-\phi(\ldr{a}(\nu))=
\ldr{a}(  \langle {\chi_1}, \nu\rangle\cdot {\tau_1})
-  \langle {\chi_1}, \ldr{a}(\nu)\rangle\cdot {\tau_1}\\
&=\langle\ldr{a}{\chi_1}, \nu\rangle\cdot {\tau_1}
+ \langle {\chi_1}, \nu\rangle\cdot\ldr{a}{\tau_1}
\end{align*}
for $\nu\in\Gamma(TM\oplus A^*)$, and 
\begin{align*}
[\tau^\dagger,\widetilde{\phi}]&=\left[\tau^\dagger, 
  \ell_{\chi_1}\cdot {\tau_1}^\dagger\right]= \pi^*\langle(\rho,\rho^t)\tau,{\chi_1}\rangle\cdot {\tau_1}^\dagger=\phi((\rho,\rho^t)\tau)^\dagger\\
[\widetilde{\phi},\widetilde{\psi}]&=\left[ \ell_{\chi_1}\cdot
  {\tau_1}^\dagger, \ell_{\chi_2}\cdot{\tau_2}^\dagger\right]=\ell_{\chi_1}\cdot \psi((\rho,\rho^t)\tau_1)^\dagger
- \ell_{\chi_2}\cdot\phi((\rho,\rho^t){\tau_2})^\dagger\\
&=\widetilde{\psi\circ(\rho,\rho^t)\circ\phi}-\widetilde{\phi\circ(\rho,\rho^t)\circ\psi}.
\qedhere\end{align*}
\end{proof}

\def\cprime{$'$} \def\polhk#1{\setbox0=\hbox{#1}{\ooalign{\hidewidth
  \lower1.5ex\hbox{`}\hidewidth\crcr\unhbox0}}} \def\cprime{$'$}
\providecommand{\bysame}{\leavevmode\hbox to3em{\hrulefill}\thinspace}
\providecommand{\MR}{\relax\ifhmode\unskip\space\fi MR }
\providecommand{\MRhref}[2]{%
  \href{http://www.ams.org/mathscinet-getitem?mr=#1}{#2}
}
\providecommand{\href}[2]{#2}

\end{document}